\documentclass[10pt]{amsart}
\usepackage[mathcal]{eucal}
\usepackage{amssymb, graphics, labelfig}

\usepackage{graphicx}

\newtheorem{thm}{Theorem}
\newtheorem{lem}[thm]{Lemma}
\newtheorem{pro}[thm]{Proposition}
\newtheorem{cor}[thm]{Corollary}
\theoremstyle{remark}
\newtheorem{rem}[thm]{Remark}

\newcommand{\PSL}{\mathrm{PSL}}
\newcommand{\SL}{\mathrm{SL}}
\newcommand{\Rep}{\mathcal{R}}
\newcommand{\Hom}{\mathrm{Hom}}
\newcommand{\CO}{\mathcal{C}^{\mathrm{Twist}}(\widehat{\lambda})}
\newcommand{\CH}{\mathcal C^{\mathrm{H\ddot ol}}}
\newcommand{\HO}{\mathcal{C}^{\mathrm{H\ddot ol}}(\widehat{\lambda})}
\newcommand{\La}{\widehat{\lambda}}
\newcommand{\St}{\widetilde{S}}
\newcommand{\ut}{\widetilde{u}}
\newcommand{\Sinf}{\partial_{\infty}\widetilde{S}}
\newcommand{\db}{/\kern -3.5pt/}
\newcommand{\leftn}{\left \vert \kern -1.5pt \left  \vert \kern -1.5pt \left \vert}
\newcommand{\rightn}{\right \vert \kern -1.5pt \right \vert \kern -1.5pt \right \vert}
\begin{document}

\title{Thurston's cataclysms for Anosov representations}

\author{Guillaume Dreyer}
\address {Department
of Mathematics,  University of Notre Dame, 255 Hurley Hall, Notre Dame, IN~46556, U.S.A.}
\email{dreyfactor@gmail.com, gdreyer@alumni.usc.edu, gdreyer@nd.edu}
\date{\today}
\thanks{This research was partially supported by the grant  DMS-0604866 from the National Science Foundation.}

%Such a shear deformation

\begin{abstract}
Given an Anosov representation $\rho \colon \pi_1(S) \to \PSL_{n}(\mathbb{R})$ and a maximal geodesic lamination $\lambda$ in a surface $S$, we construct shear deformations along the leaves of the geodesic lamination $\lambda$ endowed with a certain flag decoration, that is provided by the associated flag curve $\mathcal{F}_\rho\colon \Sinf \to \mathrm{Flag}(\mathbb{R}^n)$ of the Anosov representation $\rho$; these deformations generalize to Labourie's Anosov representations Thurston's cataclysms for hyperbolic structures on surfaces. A cataclysm is parametrized by a transverse $n$--twisted cocycle for the orientation cover $\La$ of $\lambda$. In addition, we establish various geometric properties for these deformations. Among others, we prove a variation formula for the associated length functions $\ell^i_\rho$ of the Anosov representation $\rho$.

%We define deformations for Labourie's Anosov representations that generalize Thurston's cataclysms for hyperbolic structures on surfaces. Given an Anosov representation $\rho : \pi_1(S) \to \PSL_{n}(\R)$, and a maximal geodesic lamination $\lambda$ in $S$, we construct shear deformations along the leaves of the geodesic lamination $\lambda$ endowed with a certain flag decoration that is provided by the associated flag curve $\mathcal{F}_\rho:\Sinf \to \mathrm{Flag}(\R^n)$ of the Anosov representation $\rho$. Such a shear operation is parametrized by a transverse $n$--twisted cocycle for the orientation cover $\La$ of $\lambda$. In addition, we establish various geometric properties for these deformations. In particular, we prove a variation formula for the associated length functions $\ell^i_\rho$ of the Anosov representation $\rho$.

\end{abstract}
\maketitle

Let $S$ be a closed, connected, oriented surface of genus $g\geq 2$. In \cite{La1}, F. Labourie introduced the notion of \emph{Anosov representation} to study elements of the $\PSL_{n}(\mathbb{R})$--character variety 
$$
\Rep_{\PSL_{n}(\mathbb{R})}(S)=\Hom\bigl(\pi_{1}(S),\PSL_{n}(\mathbb{R})\bigr)\db\PSL_{n}(\mathbb{R}),
$$ 
namely conjugacy classes of homomorphisms $\rho \colon \pi_1(S) \to \PSL_{n}(\mathbb{R})$ from the fundamental group $\pi_1(S)$ to the Lie group $\PSL_{n}(\mathbb{R})$ (equal to the special linear group $\SL_{n}(\mathbb{R})$ if $n$ is odd, and to  $\SL_{n}(\mathbb{R})/ \{\pm \mathrm{Id}\}$ if $n$ is even). A fundamental property of  these Anosov representations is the following.
\begin{thm}[Labourie \cite{La1}]
\label{thm:Labourie1}
Let $\rho \colon \pi_{1}(S) \rightarrow \PSL_{n}(\mathbb{R})$ be an Anosov representation. Then  $\rho$ is discrete and injective. In addition, the image $\rho(\gamma) \in \PSL_n(\mathbb{R})$ of any nontrivial $\gamma \in \pi_{1}(S)$ is diagonalizable, its eigenvalues are all real with distinct absolute values. 
\end{thm}

%The precise definition of the character variety $\mathcal R_{\PSL_{n}(\R)}(S) $ requires that the quotient be taken in the sense of geometric invariant theory \cite{Mum}: In particular, the space $\Rep_{\PSL_{n}(\R)}(S)$ inherits a structure of real algebraic variety. 

Important examples of Anosov representations are provided by \emph{Hitchin representations}, namely homomorphisms lying in \emph{Hitchin components} $\mathrm{Hit}_n(S)$. A Hitchin component $\mathrm{Hit}_n(S)$ is defined as a component of the character variety $\Rep_{\PSL_{n}(\mathbb{R})}(S)$ that contains some (conjugacy class of) $n$--\emph{Fuchsian representation}, namely some homomorphism $\rho \colon \pi_1(S) \to \PSL_n(\mathbb{R})$ of the form
$$
\rho=\iota\circ r
$$
where: $r\colon \pi_1(S) \to \PSL_2(\mathbb{R})$ is a discrete, injective homomorphism; and $\iota: \PSL_2(\mathbb{R})$ $ \to \PSL_n(\mathbb{R})$ is the preferred homomorphism defined by the $n$--dimensional, irreducible representation of $\SL_2(\mathbb{R})$ into $\SL_n(\mathbb{R})$. These preferred components $\mathrm{Hit}_n(S)$ were identified by N. Hitchin \cite{Hit} who first suggested the interest in studying their elements. 

Motivations for studying Hitchin representations find their origin in the case where $n=2$. Hitchin components $\mathrm{Hit}_2(S)$ then coincide with \emph{Teichm\"uller components} $\mathcal{T}(S)$ of $\Rep_{\PSL_{2}(\mathbb{R})}(S)$, whose elements, known as \emph{Fuchsian representations}, are of particular interest as they correspond to conjugacy classes of holonomies of hyperbolic structures on $S$. Moreover, every (representative of) element in $\mathcal{T}(S)$ is a discrete, injective homomorphism, and reversely, any such homomorphism lies in some component $\mathcal{T}(S)$ \cite{We,Mar}. It is a result due to W.~Goldman \cite{Gol} that $\Rep_{\PSL_{2}(\mathbb{R})}(S)$ possesses exactly two Teichm\"uller components $\mathcal{T}(S)$, and each of these components $\mathcal{T}(S)$ is known to be homeomorphic to $\mathbb{R}^{6g-6}$ \cite{Th1, FLP}. 

In the case where $n \geq 3$, there are one or two Hitchin components $\mathrm{Hit}_n(S)$  in $\Rep_{\PSL_{n}(\mathbb{R})} (S)$ depending on whether $n$ is odd or even, and a beautiful result of Hitchin is that each of these components $\mathrm{Hit}_n(S)$ is homeomorphic to $\mathbb{R}^{(2g-2)(n^2-1)}$. Hitchin's proof is based on the theory of Higgs bundles, and as observed by Hitchin, this complex analysis framework offers no information about the geometry of elements of $\mathrm{Hit}_n(S)$. The first geometric result for Hitchin representations is to due to S. Choi and W. Goldman \cite{ChGol} who showed that, in the case when $n=3$, the Hitchin component $\mathrm{Hit}_3(S)$ parametrizes the deformation space of \emph{real convex projective structures} on $S$. As a consequence of their work, they showed the faithfulness and the discreetness for the elements in $\mathrm{Hit}_3(S)$. 

%, and the introduction of \emph{Anosov representation}, flexible,

The powerful Anosov property for Hitchin representations discovered by Labourie \cite{La1} has the great advantage to provide a unified, dynamical-geometric approach to study all Hitchin representations, and also many more other surface group representations. Briefly, given a homomorphism $\rho\colon \pi_1(S)\to \PSL_n(\mathbb{R})$, consider the twisted, flat $M$--bundle $T^1S\times_\rho M=T^1S\times M / \pi_1(S) \to T^1S$, where: $T^1S$ is the unit tangent bundle of $S$; and where the fibre $M$ is the space of line decomposition of $\mathbb{R}^n$; let $(G_t)_{t\in \mathbb{R}}$ on $T^1S\times_\rho M$ be the flow that lifts the geodesic flow $(g_t)_{t\in \mathbb{R}}$ on $T^1S$. The representation $\rho$ is said to be \emph{Anosov} if there exists a flat section $\sigma_\rho\colon T^1S\to T^1S\times_\rho M$ with some Anosov properties for the flow $(G_t)_{t\in \mathbb{R}}$. The rigidity introduced by the Anosov dynamics guarantees the uniqueness of such a section: it is the \emph{Anosov section} $\sigma_\rho$ of the Anosov representation $\rho$, and is the central geometric feature of the Anosov representation $\rho$. In addition, the faithfulness and the discreetness, as well as the fundamental loxodromic property of Theorem~\ref{thm:Labourie1}, all come as consequences of the Anosov dynamics. Because of their properties, Anosov representations constitute a suitable higher-rank version of Fuchsian representations. As a result, we may expect that some concepts and invariants from classic Teichm\"uller theory extend to the framework of Anosov representations.

\begin{figure}[htbp]
\centerline{\AffixLabels{\includegraphics{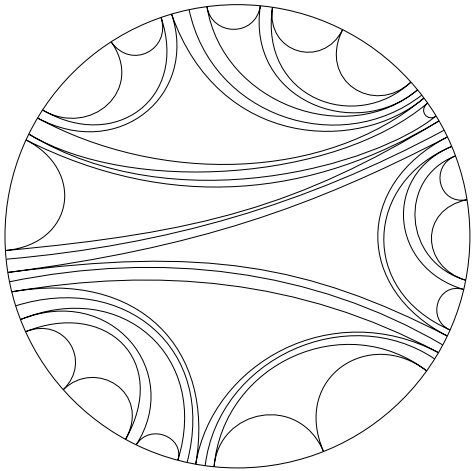}}}
\caption{The lift of a maximal geodesic lamination $\lambda$ in $S$ to the universal cover $\St$.}
\vskip 10pt
\label{Fig1maxgeodlam}
\end{figure}

\subsection*{Results}
\label{results}
We extend to Anosov representations \emph{cataclysm deformations} introduced by W. Thurston \cite{Th2, Bon1}, which themselves generalize (left) \emph{earthquakes} \cite{Th1, Ker}. Let $r\colon\pi_1(S)\to \PSL_2(\mathbb{R})$ be a Fuchsian representation; and let $\mu$ be a \emph{measured lamination}  supported in the geodesic lamination $\lambda\subset S$, namely $\lambda$ is a closed subset foliated by disjoint, complete, simple geodesics endowed with a transverse measure supported in $\lambda$ \cite{Th1,PeH,Bon4}. An earthquake is a deformation of the hyperbolic structure on $S$ of holonomy $r\colon\pi_1(S)\to \PSL_2(\mathbb{R})$ via a shear operation of the components in the complement $S-\lambda$ along the leaves of the geodesic lamination $\lambda$. Such a deformation yields another hyperbolic structure on $S$ of holonomy $\Lambda^\mu r \colon\pi_1(S)\to \PSL_2(\mathbb{R})$. The shear for each component of $S-\lambda$ is determined by the transverse measure $\mu$ which parametrizes the earthquake. A feature of  earthquakes is that every component of $S-\lambda$ moves in the left direction. Cataclysms are similar to earthquakes, with the difference that the shear is allowed to simultaneously occur to the left and to the right. In particular, a cataclysm is parametrized by a \emph{transverse cocycle} $\varepsilon$ for the geodesic lamination $\lambda$ \cite{Bon3, Bon1}, which can be thought as a transverse signed measure that is only finitely additive.

Let $\lambda \subset S$ be a \emph{maximal geodesic lamination}, i.e. the complement $S - \lambda$ is made of ideal triangles. Let $\La$ be its orientation cover (in the sense of foliation theory). Cataclysms for Anosov representations are parametrized by the (vector) space of \emph{transverse $n$--twisted cocycles $\CO$ for the oriented geodesic lamination $\La$}. Let $\Rep^{\mathrm{Anosov}}_{\PSL_{n}(\mathbb{R})}(S)$ be the set of Anosov representations; it is an open subset of $\Rep_{\PSL_{n}(\mathbb{R})}(S)$. 

%\colon \pi_1(S)\to \PSL_n(\mathbb{R})

\begin{thm} \textsc{(Cataclysm Theorem)}
Let $\rho$ be an Anosov representation. There exist a neighborhood $\mathcal{U}^\rho$ of $0\in \CO$, and a continuous, injective map
\begin{eqnarray*}
\Lambda\colon \mathcal{U}^\rho& \to &\Rep^{\mathrm{Anosov}}_{\PSL_{n}(\mathbb{R})}(S)\\
\varepsilon=(\varepsilon_1, \ldots , \varepsilon_n)&\mapsto&\Lambda^{\varepsilon}\rho
\end{eqnarray*}
that coincides, in the special case where $n=2$, with Thurston's cataclysm deformations for Fuchsian representations along the maximal geodesic lamination $\lambda$.
\end{thm}

The construction of our cataclysm deformations makes use of the geometry of Anosov representations. Indeed, let $\Sinf$ be the ideal boundary of $S$; this object is defined independently of the choice of a hyperbolic metric on $S$; see \cite{Ghys,Gro}. The following geometric property will play a central r\^ole in our construction.

\begin{thm}[Labourie \cite{La1}]
\label{thm:Labourie2}
Let $\rho\colon\pi_1(S)\to \PSL_n(\mathbb{R})$ be an Anosov representation. There exists a unique, H\"older continuous, $\rho$--equivariant flag curve $\mathcal{F}_\rho\colon\Sinf \to \mathrm{Flag}(\mathbb{R}^n)$.
\end{thm}

Note that the same invariant flag curve was similarly provided in the case of Hitchin representations by independent work of V.~Fock and A.~Goncharov \cite{FoGo}, who in addition established a certain positivity condition for this flag curve. Their approach also implies the faithfulness and the discreteness of Hitchin representations. The point of view of  Fock and Goncharov is algebaic geometric and relies on  G.~Lusztig's notion of positivity \cite{Lusz1, Lusz2}; in particular, it is very different from Labourie's.

The geometric intuition for our cataclysms is to deform an Anosov representation $\rho$ via a deformation of the associated flag curve $\mathcal{F}_\rho$. Let $\lambda \subset S$ be a geodesic lamination $\lambda \subset S$. By adding finitely many leaves, we can arrange that $\lambda$ is maximal. Let $\widetilde{\lambda}\subset \St$ that lifts the maximal geodesic lamination $\lambda\subset S$, where $\St$ is the universal cover of $S$; see Figure~\ref{Fig1maxgeodlam}. The flag curve $\mathcal{F}_\rho\colon\Sinf \to \mathrm{Flag}(\mathbb{R}^n)$ induces an equivariant flag decoration on the set of endpoints $\partial_\infty \widetilde{\lambda}\subset \Sinf$ of the geodesic lamination $\widetilde{\lambda}$. In particular, each ideal triangle in the complement $\St-\widetilde{\lambda}$ inherits a flag decoration $\big (\mathcal{F}_\rho(x),\mathcal{F}_\rho(y),\mathcal{F}_\rho(z)\big)\in \mathrm{Flag}(\mathbb{R}^n)$ on its three vertices $(x,y,z)\in \partial_\infty \widetilde{\lambda}$. Similarly as for Fuchsian representations, we define an equivariant shear operation for the \emph{flag decorated ideal triangles} in $\St-\widetilde{\lambda}$ along the leaves of the \emph{flag decorated geodesic lamination} $\widetilde{\lambda}$. The shear for each flag decorated triangle is determined by a transverse $n$--twisted cocycle $\varepsilon=(\varepsilon_1, \ldots , \varepsilon_n) \in\CO$ for the orientation cover $\La$. Such a shear deformation modifies the geometry of the flag curve $\mathcal{F}_\rho$, and so the Anosov representation $\rho$.

%$\mathcal{F}_\rho\colon\Sinf \to \mathrm{Flag}(\mathbb{R}^n)$

%The flag curve $\mathcal{F}_\rho\colon\Sinf \to \mathrm{Flag}(\mathbb{R}^n)$ induces an equivariant flag decoration $\mathcal{F}_\rho\colon \partial_\infty \widetilde{\lambda} \to \mathrm{Flag}(\mathbb{R}^n)$ on the set of endpoints $\partial_\infty \widetilde{\lambda}\subset \Sinf$ of the geodesic lamination $\widetilde{\lambda}$. In particular, each ideal triangle in the complement $\St-\widetilde{\lambda}$ inherits a flag decoration $\big (\mathcal{F}_\rho(x),\mathcal{F}_\rho(y),\mathcal{F}_\rho(z)\big)\in \mathrm{Flag}(\mathbb{R}^n)$ on its three vertices $(x,y,z)\in \partial_\infty \widetilde{\lambda}$. 

In \cite{Dr1}, the author generalizes to Anosov representations \emph{Thurston's length function} of Fuchsian representations \cite{Th1, Bon2, Bon4}, which is a fundamental tool in the study of $2$ and $3$--dimensional hyperbolic manifolds. Among others, one motivation for introducing cataclysms is to analyze the behavior of the lengths $\ell^\rho_i$ under such deformations. More precisely, fix a maximal geodesic lamination $\lambda \subset S$. Let $\HO$ be the (vector) space of \emph{transverse cocycles for the orientation cover $\La$}. Given an Anosov representation $\rho\colon \pi_1(S)\to \PSL_n(\mathbb{R})$, the construction in \cite{Dr1} provides, for every $i=1$, $\ldots$ , $n$, a continuous, linear function $\ell^\rho_i\colon \HO \to \mathbb{R}$. We prove the following variational formula.

\begin{thm} \textsc{(Variation of the lengths)}
Let $\rho'=\Lambda^{\varepsilon} \rho$ be a cataclysm deformation of an Anosov representation $\rho$ for some $\varepsilon=(\varepsilon_1, \ldots , \varepsilon_n)\in \CO$ along the maximal geodesic lamination $\lambda\subset S$. Let $\ell^\rho_{i}$ and $\ell^{\rho'}_{i}$ be the associated lengths of $\rho$ and $\rho'$, respectively. For every transverse cocycle $\alpha \in\CH(\widehat{\lambda})$, 
$$
\ell^{\rho'}_{i}(\alpha)=\ell^{\rho}_{i}(\alpha) + \tau(\alpha, \varepsilon_i)
$$
where the pairing $\tau\colon \CH (\widehat{\lambda} )\times \CH(\widehat{\lambda})\to \mathbb{R}$ is Thurston's intersection number.
\end{thm}

The nature of the above result is essentially algebraic topologic, and a large part of the proof consists of describing certain objects (co)homologically. A key idea is the homological interpretation of transverse cocycles of $\HO$ as elements of the first homology group $H_1(\widehat{U})$ where $\widehat{U}$ is a preferred open neighborhood for the oriented geodesic lamination $\La$. In particular, Thurston's intersection number \cite{PeH, Bon3, Bon4} on $\HO$, which  is a certain type of geometric intersection, turns out to be the same as the classic homology intersection pairing for $H_1(\widehat{U})$ (up to a nonzero scalar multiplication).

%,deRh,RuSu of $\mathcal{T}(S)$
%$r \colon \pi_1(S)\to \PSL_2(\mathbb{R})\in \mathcal{T}(S)$

\subsection*{Remarks}
\label{Remarks}
In the case where $n=3$, our cataclysms include \emph{bending deformations} for Fuchsian representations along a simple, closed curve $\gamma\subset S$, that were introduced by D. Johnson and J. Milson in \cite{JM}. Bendings are defined algebraically and provide examples of deformations of Fuchsian representations to Hitchin representations in $\mathrm{Hit}_3(S)$. Goldman \cite{Gol2} gives a geometric interpretation of bendings as deformations of hyperbolic structures to real convex projective structures on $S$. His description emphasizes the r\^ole played by the ideal boundary $\Sinf$ that, in the latter context, identifies with a convex projective curve embedded in $\mathbb{RP}^2$: bendings appear as explicit deformations of the convex boundary, and coincide with cataclysms along a simple, closed geodesic $\gamma\subset S$.  

%in this context,
%Moreover, cataclysms are real analytic.

A question that this article does not address is the completeness of cataclysms. In \cite{Dr2}, we define the notion of Anosov representation along a geodesic lamination $\lambda\subset S$, where these considerations find a more natural answer. Cataclysms extend to this class of Anosov representations, and we show the existence of cataclysm paths in this (open) subset of $\Rep_{\PSL(\mathbb{R})}(S)$. In addition, our analysis gives precise conditions for the existence of such paths in terms of the length functions $\ell_i^\rho$ introduced in \cite{Dr1}.

Another motivation for studying cataclysms is part of the development of a new system of coordinates for Hitchin components $\mathrm{Hit}_n(S)$. Let us recall Hitchin's result, namely that $\mathrm{Hit}_n(S)$ is diffeomorphic to $\mathbb{R}^{(2g-2)(n^2-1)}$. Hitchin's parametrization is based on Higgs bundle techniques, and in particular requires the initial choice of a complex structure on $S$. In a joint work with F. Bonahon \cite{BonDr1,BonDr2}, we construct a geometric, real analytic parametrization of Hitchin components $\mathrm{Hit}(\mathbb{R}^n)$. One feature of this parametrization is that it is based on topological data only. In essence, our coordinates are an extension of Thurston's shearing coordinates \cite{Th2,Bon1} on the Teichm\"uller space $\mathcal{T}(S)$, combined with Fock-Goncharov's coordinates on the moduli space of positive framed local systems of a punctured surface \cite{FoGo}.

\section{Anosov representations}

We begin with reviewing some material about Anosov representations. The main objects are the \emph{Anosov section} and the associated \emph{flag curve} of an Anosov representation, that will play a fundamental r\^ole throughout. Main references for this section are \cite{La1, Gui, GuiW1,GuiW2}.

For convenience, we fix once and for all a hyperbolic metric $m_0$ on $S$. It induces a $m_0$--geodesic flow $(g_t)_{t\in \mathbb{R}}$ on the unit tangent bundle $T^1S$: we refer to the associated orbit space as the $m_0$--geodesic foliation $\mathcal{F}$ of $T^1S$. 

%Bundle description(s) of an Anosov representation

\subsection{The Anosov bundle(s)}
\label{bundledescription}

We present two equivalent descriptions of an Anosov representation. 

\subsubsection{$M$--bundle description}
\label{Mbundledescription}

Let $M$ be the space of line decompositions of $\mathbb{R}^n$, namely $M$ is the set of $n$--tuplets of $1$--dimensional subspaces $(L_1, \dots ,L_n)$ such that $\mathbb{R}^n=L_1\oplus\cdots\oplus L_n$. Given a homomorphism $\rho\colon \pi_1(S) \to \PSL_n(\mathbb{R})$, consider the flat twisted $M$--bundle
$$
T^1 S \times_\rho M = T^1\widetilde S \times M / \pi_1(S)\to T^1S
$$
where: $T^1\widetilde S$ is the unit tangent bundle of the universal cover $\widetilde{S}$ of $S$; and where the action of $\pi_1(S)$ is defined by the property that
\begin{eqnarray*}
\gamma \cdot \big ( \ut,(L_1, \dots ,L_n) \big )=\big (\gamma \ut,  (\rho(\gamma)L_1, \dots ,\rho(\gamma)L_n)\big )
\end{eqnarray*}
for every $ \gamma \in \pi_{1}(S)$ and $ \big (\ut,(L_1, \dots ,L_n) \big) \in T^{1}\widetilde{S}\times M$. Via the flat connection, the geodesic flow $(g_t)_{t\in \mathbb{R}}$ on  $T^{1}S$ lifts to a flow $(G_{t})_{t\in \mathbb{R}}$ on the total space $T^{1}{S}\times_{\rho}M$; here, the ``flatness'' condition means that, if one looks at the situation in the universal cover $T^{1}\widetilde{S}\times M$, the lift $(\widetilde{G}_{t})_{t\in \mathbb{R}}$ acts on $T^1\widetilde{S}\times \mathbb{R}^n$ as the geodesic flow $(\widetilde{g})_{t\in \mathbb{R}}$ on the first factor, and trivially on the second factor. We shall refer to  $T^1 S \times_\rho M\to T^1S$ as the \emph{associated} $M$--\emph{bundle} of the homomorphism $\rho\colon \pi_1(S) \to \PSL_n(\mathbb{R})$.

A homomorphism $\rho \colon\pi_1(S)\to\PSL_n(\mathbb{R})$ is said to be \emph{Anosov} if the associated $M$--bundle admits a continuous section $\sigma: T^1S \to T^{1}{S}\times_{\rho}M$, $u \mapsto \big(V_1(u), \ldots ,V_n(u) \big)$ satisfying the two following properties:
\begin{enumerate}
\item The section $\sigma$ is \emph{flat}, namely if $\widetilde{\sigma}: T^{1}\St \to T^{1}\St \times M$, $\ut \mapsto \big(\widetilde{V}_1(\ut), \dots ,\widetilde{V}_n(\ut) \big)$ is a lift of $\sigma$, then for every $i=1$, $\ldots$ , $n$, for every $t\in \mathbb{R}$, the fibres $\widetilde{V}_i(\ut)$ and $\widetilde{V}_i(g_t(\ut))$  coincide as lines of $\mathbb{R}^n$; 

\medskip  

\item \label{AnosovProperty} Let  $T^1S \times_{\rho}\mathrm{End}(\mathbb{R}^n)\to T^1S$ be the flat twisted $\mathrm{End}(\mathbb{R}^n)$--bundle, where $\rho(\pi_1(S))$ acts by conjugation on the space of linear endomorphisms $\mathrm{End}(\mathbb{R}^n)$. Let $(\bar{G}_{t})_{t\in\mathbb{R}}$ be the lift on $T^1S \times_{\rho}\mathrm{End}(\mathbb{R}^n)$ of the geodesic flow $(g_t)_{t\in \mathbb{R}}$. The flat section $\sigma= (V_1, \dots ,V_n)$ induces a line splitting $\bigoplus_{1\leq i,j\leq n}V_i^*\otimes V_j$ of the flat bundle $T^1S \times_{\rho}\mathrm{End}(\mathbb{R}^n)\to T^1S$ with the property that each line sub-bundle $V_i^*\otimes V_j\to T^1S$ is invariant under the action of the flow $(\bar{G}_{t})_{t\in\mathbb{R}}$. We require the restriction of flow $(\bar{{G}_{t}}_{|V_i^*\otimes V_j})_{t\in \mathbb{R}}$ to each line sub-bundle $V_i^*\otimes V_j$ to be ``Anosov'' in the following sense: for every $i \neq j$, there exists a metric $\leftn \ \rightn$ on $V_i^*\otimes V_j$, and some constants $A\geq0$ and $a>0$ such that, $\forall u\in T^1S\text{, }\forall \psi_u\in V_i^*\otimes V_j(u)\text{, }\forall t > 0$,
%\begin{eqnarray*}
\begin{align*}
\text{if $i>j$, }&\leftn \bar{G}_t \psi_u \rightn_{g_t(u)} \leq A e^{-at} \leftn \psi_u \rightn_u;\\
\text{if $i<j$, }&\leftn \bar{G}_{-t} \psi_u \rightn_{g_{-t}(u)} \leq A e^{-at} \leftn \psi_u \rightn_u.
%\end{eqnarray*}
\end{align*}
\end{enumerate}

%\begin{rem}
%\label{Eigenbundle}

\subsubsection{$\bar{\mathbb{R}}^n$--bundle description}
%\label{$\bar{\mathbb{R}}^n$--bundle description}
\label{Rbundle description}

Here is an alternative description of an Anosov re\-presentation, with which it is sometimes easier to work in practice.

Let $\bar{\mathbb{R}}^n=\mathbb{R}^n/ \{\pm \mathrm{Id}\}$; note that $\PSL_n(\mathbb{R})$ acts on $\bar{\mathbb{R}}^n$. Given a homomorphism $\rho\colon \pi_1(S)\to \PSL_n(\mathbb{R})$, consider the flat twisted $\bar{\mathbb{R}}^n$--bundle $T^1S \times_\rho \bar{\mathbb{R}}^n\to T^1S$. Let $(G_t)_{t\in \mathbb{R}}$ be the lift on $T^1S\times_\rho \bar{\mathbb{R}}^n$ of the geodesic flow $(g_t)_{t\in \mathbb{R}}$. Then $\rho$ is an Anosov representation if the bundle $T^1S\times_\rho \bar{\mathbb{R}}^n$ splits as a sum of line sub-bundles $V_1\oplus\cdots\oplus V_n$ (for the obvious definition of direct sum of lines  in $\bar{\mathbb{R}}^n$) with the property that: each line sub-bundle $V_i\to T^1S$ is invariant under the action of the flow $(G_t)_{t\in \mathbb{R}}$; and the line sub-bundles $V_i\to T^1S$ satisfy the Anosov property (\ref{AnosovProperty}). Note that we abuse the terminology ``line bundle" as the fibre $V_i(u)$ of $V_i\to T^1S$ identifies with the quotient of a line of $\mathbb{R}^n$ by $\pm \mathrm{Id}$; this discrepancy will have no effect in the following.

%\end{rem}

As a consequence of the above alternative bundle description, we will often think of the components $V_i$ of the Anosov section $\sigma_\rho=(V_1, \ldots , V_n)$ as line (sub-)bundles $V_i \to T^1S$ that are invariant under the action of the flow $(G_t)_{t\in \mathbb{R}}$.

The Anosov property (\ref{AnosovProperty}) of the flat section $\sigma= (V_1, \ldots ,V_n)$ has several important consequences, that we now review.

\begin{thm}[Labourie \cite{La1}]
\label{AnosovSection}
Let $T^{1}{S}\times_{\rho}M\to T^1S$ be the associated flat $M$--bundle of an Anosov representation $\rho$. It admits a unique, flat, continuous section satisfying the Anosov property (\ref{AnosovProperty}) as above; we shall refer to it as the Anosov section $\sigma_\rho \colon T^{1}\St \to T^{1}\St \times M$ of the Anosov representation $\rho$. In addition, $\sigma_\rho$ is smooth along the leaves of the geodesic foliation $\mathcal{F}$ of  $T^1S$, and is transversally H\"older continuous.
%Let $\rho\colon\pi_1(S)\to\PSL_n(\mathbb{R})$ be an Anosov representation. The associated flat $M$--bundle $T^{1}{S}\times_{\rho}M\to T^1S$ admits a unique, flat, continuous section satisfying the Anosov property (\ref{AnosovProperty}) as above; we shall refer to it as the Anosov section $\sigma_\rho \colon T^{1}\St \to T^{1}\St \times M$ of the Anosov representation $\rho$. In addition, $\sigma_\rho$ is smooth along the leaves of the geodesic foliation $\mathcal{F}$ of  $T^1S$, and is transversally H\"older continuous.
\end{thm}

The following observation is an easy consequence of the uniqueness of the Anosov section, that we state as a lemma for future reference.
 
\begin{lem}
\label{lem:LineReversing}
Let $\sigma_\rho=(V_1, \ldots , V_n)$ be the Anosov section of some Anosov representation $\rho$, that lifts to $\widetilde{\sigma}_\rho=(\widetilde{V}_1, \ldots , \widetilde{V}_n)$. For $\ut\in T^1\St$ projecting to $u\in T^1S$, the fibres $\widetilde{V}_i(\ut)$ and $\widetilde{V}_{n-i+1}(-\ut)$ coincide as lines of $\mathbb{R}^n$.
\end{lem}

\begin{proof}
Consider the section $\bar{\sigma}_\rho(u)= \big(V_n(-u), \ldots , V_1(-u) \big)$, for $u\in T^1S$. Then $\bar{\sigma}_\rho$ is flat, continuous, and one easily verifies that, for every $t\in \mathbb{R}$, $\bar{\sigma}_\rho(g_t(u))=\big(V_n(g_{-t}(-u)), \ldots , V_1(g_{-t}(-u))\big)$. Moreover, since $\sigma_\rho=(V_1, \ldots , V_n)$ is the Anosov section, it follows that $\bar{\sigma}_\rho$ also satisfies the Anosov property (\ref{AnosovProperty}), hence $\bar{\sigma}_\rho=\sigma_\rho$.
\end{proof}

A fundamental property of Anosov representations is the following.
 
 \begin{thm}[Labourie \cite{La1}]
\label{thm:Labourie3}
Let $\rho \colon \pi_{1}(S) \rightarrow \PSL_{n}(\mathbb{R})$ be an Anosov representation. Then $\rho$ is  injective and discrete. In addition, the image  $\rho(\gamma) \in \PSL_n(\mathbb{R})$ of any nontrivial  $\gamma \in \pi_{1}(S)$ is diagonalizable, and its eigenvalues are all real with distinct absolute values. 
\end{thm}
 
By ``$\rho(\gamma) \in \PSL_n(\mathbb{R})$ is diagonalizable", we mean that every lift $\widetilde{\rho(\gamma)} \in \SL_n(\mathbb{R})$ is a diagonalizable matrix. When $n$ is odd, $\PSL_n(\mathbb{R})=\SL_n(\mathbb{R})$ and there is no ambiguity. When $n$ is even, $\rho(\gamma) \in \PSL_n(\mathbb{R})$ admits two lifts $\pm \widetilde{\rho(\gamma)} \in \SL_n(\mathbb{R})$; however, the absolute values of the eigenvalues of $\rho(\gamma)\in \PSL_n(\mathbb{R})$ are well defined.

We now make the content of Theorem~\ref{thm:Labourie3} more precise, and also much stronger. Let $\widetilde{\sigma}_\rho=(\widetilde{V}_1, \ldots ,\widetilde{V}_n)$ that lifts the Anosov section $\sigma_\rho=(V_1, \ldots , V_n)$. Pick a nontrivial element $\gamma\in \pi_1(S)$. Let $\widetilde{\rho(\gamma)}\in \SL_n(\mathbb{R})$ that lifts $\rho(\gamma)$. Consider the oriented geodesic $g_\gamma\subset \St$ fixed by the isometric action of $\gamma$. Let $\ut \in T^1\St$ be a unit vector directing $g_\gamma$, and let us set $\widetilde{V}_i(g_\gamma)=\widetilde{V}_i(\ut)\subset \mathbb{R}^n$; $\sigma_\rho$ being flat, $\widetilde{V}_i(g_\gamma)$ does not depend on the choice of the unit tangent vector $\ut$. Since $\gamma \widetilde{u}\in g_{\gamma}$, and $\widetilde{V}_i(\gamma\ut)=\rho(\gamma)\widetilde{V}_i(\ut)$ (it is the equivariance property of the lift $\widetilde{\sigma}_\rho$), it follows from the above discussion that each line $\widetilde{V}_i(g_\gamma)$ is an eigenspace for $\widetilde{\rho(\gamma)}$; let us denote by $\lambda^{\rho}_i(\gamma)\in \mathbb{R}$ the corresponding eigenvalue: we shall refer to it as the $i$--\emph{th eigenvalue of} $\rho(\gamma)$. Moreover, a strong consequence
of the Anosov property $(\ref{AnosovProperty})$ is the following control on the eigenvalues: for every nontrivial $\gamma\in \pi_1(S)$,
\begin{eqnarray*}
|\lambda^{\rho}_1(\gamma) |> |\lambda^{\rho}_2(\gamma)| > \dots > |\lambda^{\rho}_n(\gamma)| \label{nom}. 
\end{eqnarray*}

Finally, let $\Rep^{\mathrm{Anosov}}_{\PSL_{n}(\mathbb{R})}(S)\subset \Rep_{\PSL_{n}(\mathbb{R})}(S)$ be the set of Anosov representations.

\begin{thm}[Labourie \cite{La1}]
\label{thm:Labourie}
The set of Anosov representations $\Rep^{\mathrm{Anosov}}_{\PSL_{n}(\mathbb{R})}(S)$ is open in the character variety $\Rep_{\PSL_{n}(\mathbb{R})}(S)$. 
\end{thm}

For a general treatment of Anosov representations from a surface group to a semisimple Lie group, see \cite{GuiW1,GuiW2}.

\subsection{The flag curve of an Anosov representation}
\label{FlagDescription} Recall that a (complete) \emph{flag} $F$ of $\mathbb{R}^n$ consists of a nested sequence of vector subspaces
$$
F=F^{(1)}\subset F^{(2)} \cdots \subset F^{(n-1)}
$$
where each $F^{(i)}$ is a subspace of $\mathbb{R}^n$ of dimension $i$. We will denote by $\mathrm{Flag}(\mathbb{R}^n)$ the flag variety of $\mathbb{R}^n$. A fundamental property of Anosov representations is the existence of an associated equivariant flag curve.

\begin{thm}[Labourie \cite{La1}]
Let $\rho \colon \pi_1(S)\to \PSL_n(\mathbb{R})$ be an Anosov representation. There exists a unique, continuous, $\rho$--equivariant flag curve $\mathcal{F}_\rho:\Sinf\to \mathrm{Flag}(\mathbb{R}^n)$ that satisfies the following properties: 
\begin{enumerate}
\item $\mathcal{F}_\rho \colon \Sinf\to \mathrm{Flag}(\mathbb{R}^n)$ is H\"older continuous; 
\item $\mathcal{F}_\rho$ is $2$--hyperconvex, namely, for every $x\ne y \in \Sinf$,
$$
\mathcal{F}_\rho^{(i)}(x)\bigoplus\mathcal{F}_\rho^{(n-i)}(y)=\mathbb{R}^n.
$$
\end{enumerate}
\end{thm}
By \emph{$\rho$--equivariant}, we mean that, for every $\gamma\in \pi_1(S)$, for every $x\in \Sinf$,
$\mathcal{F}_\rho(\gamma x)=\rho(\gamma)\mathcal{F}_\rho(x)$.

%
%In particular, the flag curve $\mathcal{F}_\rho$ and the Anosov section $\sigma_\rho$ are related as follows.

The existence of the flag curve $\mathcal{F}_\rho\colon \Sinf\to \mathrm{Flag}(\mathbb{R}^n)$ comes again as a consequence of the Anosov dynamics. The flag curve $\mathcal{F}_\rho$ derives from the Anosov section $\sigma_{\rho} \colon T^1S\to T^{1}S\times_{\rho} M$, and both objects are related as follows. Let $\widetilde{\sigma}_\rho= (\widetilde{V}_1, \ldots ,\widetilde{V}_n )$ that lifts $\sigma_\rho= (V_1, \ldots ,V_n )$. For every $\widetilde{u}\in T^1\widetilde{S}$, let $g\subset \widetilde{S}$ be the oriented geodesic directed by $\widetilde{u}$, and let $x_g^+$ and $x_g^-\in \Sinf$ be its positive and negative endpoints, respectively. For every $i=1$, $\ldots$ , $n$, 
\begin{eqnarray}
\label{FlagVersusSection}
\widetilde{V}_{i}(\ut)=\mathcal{F}_\rho^{(i)}(x_g^+)\cap \mathcal{F}_\rho^{(n-i+1)}(x_g^-)\subset \mathbb{R}^n
\end{eqnarray}
with the consequence that
\begin{align*}
&\mathcal{F}_\rho^{(i)}(x_g^+)=\widetilde{V}_1(\ut)\subset\widetilde{V}_1(\ut)\oplus\widetilde{V}_2(\ut)\subset \cdots \subset \widetilde{V}_1(\ut)\oplus \cdots \oplus \widetilde{V}_i(\ut)  \\
&\mathcal{F}_\rho^{(i)}(x_g^-)=\widetilde{V}_n(\ut)\subset\widetilde{V}_{n-1}(\ut)\oplus\widetilde{V}_n(\ut)\subset \cdots \subset \widetilde{V}_{n-i+1}(\ut)\oplus \cdots \oplus \widetilde{V}_n(\ut).
\end{align*}

Note that, by the relation (\ref{FlagVersusSection}), one easily recover the Anosov section $\sigma_\rho$ starting from the flag curve $\mathcal{F}_\rho$. Note also that the $2$--hyperconvexity of $\mathcal{F}_\rho$ guarantees that  $\mathcal{F}_\rho^{(i)}(x_g^+)\cap \mathcal{F}_\rho^{(n-i+1)}(x_g^-)\ne \varnothing$.

%Note that the relation (\ref{FlagVersusSection}) also shows how, by starting from the flag curve $\mathcal{F}_\rho$, one can reconstruct the Anosov section $\sigma_\rho$. Note also that the $2$--hyperconvexity of $\mathcal{F}_\rho$ guarantees that  $\mathcal{F}_\rho^{(i)}(x_g^+)\cap \mathcal{F}_\rho^{(n-i+1)}(x_g^-)\ne \varnothing$.

Throughout, we will indifferently alternate between the point of view of the Anosov section $\sigma_{\rho} \colon T^1S\to T^{1}S\times_{\rho} M$, and the one of the flag curve $\mathcal{F}_\rho\colon\Sinf\to \mathrm{Flag}(\mathbb{R}^n)$, to our liking. The reader should simply keep in mind that manipulating one of the two objects is equivalent to manipulating the other.

We conclude this short review with one last comment about the flag curve $\mathcal{F}_\rho$.

\begin{thm}[Labourie \cite{La1}]
Let $\mathcal{F}_\rho\colon \Sinf\to \mathrm{Flag}(\mathbb{R}^n)$ be the flag curve of an Anosov representation $\rho$. The image $\mathcal{F}_\rho(\Sinf)$ is the limit set for the action of the subgroup $\rho(\pi_1(S))\subset \PSL_n(\mathbb{R})$, namely it is the intersection of all $\rho(\pi_1(S))$--invariant closed subsets in the flag variety $\mathrm{Flag}(\mathbb{R}^{n})$.
%Let $\rho:\pi_1(S)\to \PSL_n(\R)$ be an Anosov representation along with its associated flag curve $\mathcal{F}_\rho:\Sinf\to \mathrm{Flag}(\R^n)$. The image $\mathcal{F}_\rho(\Sinf)$ is the limit set for the action of the subgroup $\rho(\pi_1(S))\subset \PSL_n(\R)$, namely it is the intersection of all $\rho(\pi_1(S))$--invariant closed subsets in the flag variety $\mathrm{Flag}(\R^{n})$.
\end{thm} 

\section{Preliminaries}

\subsection{A bunch of estimates}
\label{sect:Inequalities}

We prove several estimates of which we will make great use throughout.

A \emph{geodesic lamination} $\lambda$ in $S$ is a closed subset of $S$ that is a union of disjoint, complete, simple geodesics \cite{PeH,Bon4}. $\lambda$ is said to be \emph{maximal} in $S$ if every component of the complement $S-\lambda$ is isometric to an ideal triangle; see Figure~\ref{Fig1maxgeodlam}.

%; note that, since $\lambda$ is maximal, the geodesics $g_{d_{0}}^{-}$ and  $g_{d_{0}}^{+}$ are asymptotic. 

Fix a maximal geodesic lamination $\lambda\subset S$. Let $k$ be a simple arc transverse to $\lambda$ that does not backtrack, so that $k$ intersects each leave of $\lambda$ at most once. Given a component $d_{0}$ of $k-\lambda$ that does not contain any endpoint of $k$, let us denote by $g_{d_{0}}^{-}$ and $g_{d_{0}}^{+}\subset \lambda$ the two asymptotic geodesic leaves passing by the endpoints of $d_0$. Consider all components $d \subset k-\lambda$ that are bounded by $g_{d_{0}}^{-}$ and $g_{d_{0}}^{+}$, namely every component $d$ such that $g_{d_{0}}^{-}$ and  $g_{d_{0}}^{+}$ are both passing by the endpoints of $d$. As shown on Figure~\ref{Fig2}, such a subarc $d$ lies in one of two regions delimited by the subarc $d_{0}$ and the two leaves $g_{d_{0}}^{-}$ and $g_{d_{0}}^{+}$. Besides, the metric $m_0$ on $S$ being negatively curved, the two asymptotic geodesics $g_{d_{0}}^{-}$ and $g_{d_{0}}^{+}$ spread out in the opposite direction. As a result, one of two regions contains finitely many subarcs $d \subset k-\lambda$. We define the \emph{divergence radius} $r(d_{0})\in \mathbb{N}$ as the smallest number of subarcs contained in one of these two regions.

%Fix a maximal geodesic lamination $\lambda\subset S$. Let $k$ be a simple arc transverse to $\lambda$ which does not backtrack so that $k$ intersects each leave of $\lambda$ at most once. Let $d_{0}$ be a component of $k-\lambda$ which does not contain any endpoint of $k$, and let $g_{d_{0}}^{-}$ and $g_{d_{0}}^{+}\subset \lambda$ be the two geodesic leaves passing by the endpoints of $d_0$; note that, $\lambda$ being maximal, the geodesics $g_{d_{0}}^{-}$ and  $g_{d_{0}}^{+}$ are asymptotic. Consider all components $d \subset k-\lambda$ that are bounded by $g_{d_{0}}^{-}$ and $g_{d_{0}}^{+}$, namely every component $d$ such that $g_{d_{0}}^{-}$ and  $g_{d_{0}}^{+}$ are both passing by the endpoints of $d$. As shown on Figure~\ref{Fig2}, each such subarc $d$ lies in one of two regions delimited by the subarc $d_{0}$ and the two leaves $g_{d_{0}}^{-}$ and $g_{d_{0}}^{+}$. Because of the negative curvature, the two asymptotic geodesics $g_{d_{0}}^{-}$ and $g_{d_{0}}^{+}$ spread out in the opposite direction. As a result, one of two regions contains finitely many subarcs $d \subset k-\lambda$; we define the \emph{divergence radius} $r(d_{0})\in \mathbb{N}$ as the smallest number of subarcs contained in one of the two above regions.

\begin{figure}[htbp]
\SetLabels
( .62*.24 ) $d_0$\\
( .43*.32 ) $d$\\
( .82*.86 ) $g^{+}_{d_0}$\\
( .82*.10 ) $g^{-}_{d_0}$\\
\endSetLabels
\centerline{\AffixLabels{\includegraphics[width=10cm, height=3cm]{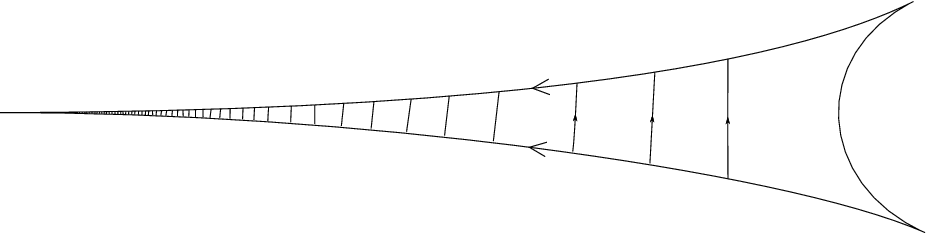}}}
\caption{A component of $S-\lambda$ that intersects a transverse, simple, nonbacktracking, oriented arc $k$.}
\vskip 10pt
\label{Fig2}
\end{figure}
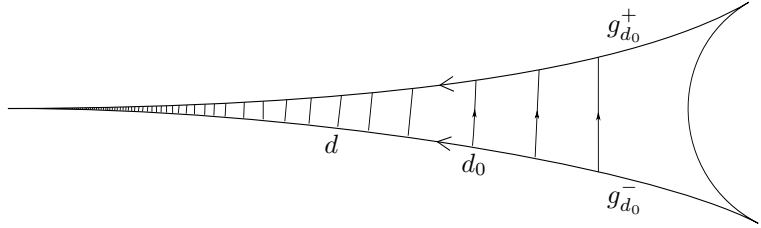

For every integer $r\geq 0$, let $D_{r}$ be the set of components $d \subset k-\lambda$ such that $r(d)=r$.

\begin{lem}
\label{lem01}
For every integer $r\geq 0$,
$$
\mathrm{Card}(D_{r} ) \leq 4g-4
$$
where $g$ is the genus of the surface $S$. 
\end{lem}

\begin{lem}
\label{lem02}
There exist some constant $A>0$, depending on $k$, such that, for every component $d \subset k-\lambda$,
$$
%and $B\geq0$
%\mathrm{length}(d)\leq Be^{-Ar(d)}.
\mathrm{length}(d)=\mathrm{O}\big (e^{-Ar(d)} \big ).
$$
\end{lem}

\begin{proof}
Since $\lambda$ is maximal, and $k$ is simple and does not backtrack, Lemma~\ref{lem01} comes as a consequence of the fact that the complement $S-\lambda$ is made of $4g-4$ ideal triangles. Lemma~\ref{lem02} follows from classical hyperbolic geometry estimates; see \cite{FLP} for instance.
\end{proof}

%\cite[\S1]{Bon1}

Let $\sigma_\rho=(V_1, \ldots, V_n)$ be the Anosov section of some Anosov representation $\rho$, that lifts to $\widetilde{\sigma}_\rho= (\widetilde{V}_1, \dots ,\widetilde{V}_n )$; see \S\ref{bundledescription}. Since $\sigma_\rho$ is flat, the lift $\widetilde{\sigma}_\rho=(\widetilde{V}_1, \ldots, \widetilde{V}_n)$ associates to every oriented geodesic $g\subset \St$ a line decomposition $\widetilde{V}_{1}(g) \oplus \cdots \oplus \widetilde{V}_{n}(g)$ of $\mathbb{R}^n$. 

Let $\widetilde{\lambda}\subset \widetilde{S}$ that lifts the maximal geodesic lamination $\lambda\subset S$; see Figure~\ref{Fig1maxgeodlam}. Consider a transverse, simple, nonbacktracking, oriented arc $k$ to $\widetilde{\lambda}$. Orient positively the leaves of $\widetilde{\lambda}$ intersecting $k$ for the transverse orientation determined by the oriented arc $k$, namely so that the angle between $k$ and every leaf of $\widetilde{\lambda}$ is positively oriented. For every component $d \subset k-\widetilde{\lambda}$, $g_{d}^{+}$ and $g_{d}^{-}\subset \widetilde{\lambda}$ denote respectively the oriented geodesics passing by the positive and the negative endpoints of the oriented subarc $d$. Finally, let $\mathrm{dist}_{T^1S}$ be the distance on the unit tangent bundle $T^1S$; and let $\mathrm{dist}_{\mathbb{RP}^{n-1}}$ be a metric on $\mathbb{RP}^{n-1}$.

% and $M\geq0$, both 
\begin{lem}
\label{cor03}
There exist some constant $K>0$, depending on $k$ and $\rho$, such that, for every $i=1$, $\ldots$ , $n$, for every component $d\subset k-\widetilde{\lambda}$,
$$
\mathrm{dist}_{\mathbb{RP}^{n-1}} \big (\widetilde{V}_i(g_{d}^{+}), \widetilde{V}_i(g_{d}^{-}) \big )=\mathrm{O}\big (e^{-Kr(d)}\big ).%\leq Me^{-Kr(d)}.
$$
\end{lem}

\begin{proof}
Let $u_d^-$ and $u_d^+\in T^1\St$ be respectively the unit vectors based at the positive and the negative endpoints of the oriented subarc $d\subset k-\widetilde{\lambda}$ that direct the oriented geodesics $g_{d}^{-}$ and $g_{d}^{+}$. Note that $g_{d}^{-}$ and $g_{d}^{+}$ both  converge to, or diverge from their common endpoint. Hence, by compacity of $k$, for every component $d\subset k-\widetilde{\lambda}$, $\mathrm{dist}_{T^1S}(u_d^+,u_d^-) \leq C\, \text{length}(d)$ for some $C\geq0$ (depending on $k$). Since $\widetilde{\sigma}_\rho(\ut)=\big (\widetilde{V}_1(\ut), \ldots, \widetilde{V}_n(\ut) \big)$ depends locally H\"older continuously on $\ut \in T^1\St$, $\mathrm{dist}_{\mathbb{RP}^{n-1}} \big (\widetilde{V}_i(g_{d}^{+}), \widetilde{V}_i(g_{d}^{-}) \big )\leq C'\, {\text{length}(d)}^\mu$ for some $C'\geq 0$ and some $\mu \in (0,1]$ (both $C'$ and $\mu$ depending on $k$ and $\rho$). An application of Lemma~\ref{lem02} then yields the desired estimate.
\end{proof}

For every $\varepsilon=(\varepsilon_{1}, \ldots, \varepsilon_{n})\in \mathbb{R}^n$, for every oriented geodesic $g\subset \widetilde{\lambda}$, we will denote by $T^{\varepsilon}_{g}\colon \mathbb{R}^n\to\mathbb{R}^n$ the linear map that acts on each line $\widetilde{V}_i(g)\subset \mathbb{R}^n$ by multiplication by $e^{\varepsilon_i}$.

Let $k$ be a transverse, simple, nonbacktracking, oriented arc to $\widetilde{\lambda}$. Orient positively the leaves of $\widetilde{\lambda}$ intersecting $k$ for the transverse orientation determined by $k$. Pick a norm $\left \Vert \ \right \Vert_{\mathbb{R}^n}$ on $\mathbb{R}^n$; let $\leftn \ \rightn$ be the induced norm on the vector space of linear endomorphisms $\mathrm{End}(\mathbb{R}^n)$. Finally, let $\left \Vert \ \right \Vert_{\mathrm{Mat}(\mathbb{R}^n)}$ be a norm on the vector space of square matrices $\mathrm{Mat}_n(\mathbb{R})$.

\begin{lem}
\label{lem04}
For every component $d \subset k-\widetilde{\lambda}$, for every $\varepsilon\in \mathbb{R}^n$, there exists some constant $K>0$, depending on $k$ and $\rho$, such that
$$
\leftn T^{\varepsilon}_{g_d^{-}}\circ T^{-\varepsilon}_{g_d^{+}}-\mathrm{Id} \rightn=\mathrm{O}\left (e^{2 \left \Vert \varepsilon \right \Vert_{\mathbb{R}^n}-Kr(d)} \right).
$$
\end{lem}

\begin{proof}
Let $\mathcal{B}_d$ be a basis of unit vectors for $\left \Vert \ \right \Vert_{\mathbb{R}^n}$ that is adapted to the line decomposition $\widetilde{V}_1(g_{d}^{-})\oplus \cdots \oplus \widetilde{V}_n(g_{d}^{-})=\mathbb{R}^n$. Let $\mathrm{Mat}_{\mathcal{B}_{d}}(T^{\varepsilon}_{g_d^{-}}\circ T^{-\varepsilon}_{g_d^{+}}-\mathrm{Id})$ be the matrix representation of the linear map $T^{\varepsilon}_{g_d^{-}}\circ T^{-\varepsilon}_{g_d^{+}}-\mathrm{Id}$ with respect to the basis $\mathcal{B}_{d}$. By an easy calculation, it follows from Lemma~\ref{cor03} that, for every $d\subset k-\widetilde{\lambda}$,
$$
\left \Vert  \mathrm{Mat}_{\mathcal{B}_d}(T^{\varepsilon}_{g_d^{-}}\circ T^{-\varepsilon}_{g_d^{+}}-\mathrm{Id}) \right \Vert_{\mathrm{Mat}_n(\mathbb{R})}\leq Me^{2 \left \Vert \varepsilon \right \Vert_{\mathbb{R}^n}-Kr(d)}
$$
for some $M\geq 0$ (depending on $k$ and $\rho$). In addition, since $k$ is compact, the set of adapted basis $\{\mathcal{B}_d\}_{d\subset k-\widetilde{\lambda}}$ as above lies in some compact subset of ${(\mathbb{R}^n)}^n$. Thus
$$
\leftn T^{\varepsilon}_{g_d^{-}}\circ T^{-\varepsilon}_{g_d^{+}}-\mathrm{Id}  \rightn  \leq M' \left \Vert  \mathrm{Mat}_{\mathcal{B}_d}(T^{\varepsilon}_{g_d^{-}}\circ T^{-\varepsilon}_{g_d^{+}}-\mathrm{Id}) \right \Vert_{\mathrm{Mat}_n(\mathbb{R})}
$$
for some $M'\geq 0$ (depending on $k$ and $\rho$). Hence, for every subarc $d\subset k-\widetilde{\lambda}$,
$$ 
\leftn T^{\varepsilon}_{g_d^{-}}\circ T^{-\varepsilon}_{g_d^{+}}-\mathrm{Id}  \rightn  \leq M''e^{2 \left \Vert \varepsilon \right \Vert_{\mathbb{R}^n}-Kr(d)}
$$
for some $M''\geq 0$ (depending on $k$ and $\rho$), which proves the requested estimate.
\end{proof}

\subsection{Transverse cocycles for geodesic laminations}
\label{sect:TransCocycles} 

We need to remind the reader of the definition of transverse cocycles for geodesic laminations, along with their main properties. See \cite{Bon1,Bon3} for complementary details. 

Let $\lambda\subset S$ be a geodesic lamination. A \emph{transverse cocycle} $\alpha$ for $\lambda$ can be thought as a transverse signed measure for $\lambda$ that is finitely additive only. More precisely, $\alpha$ assigns to every transverse arc $k$ to $\lambda$ a number $\alpha(k)$, with the property that $\alpha(k)=\alpha(k_{1})+\alpha(k_{2})$, whenever $k_{1}$ and $k_{2}$ are two subarcs of $k$ with disjoint interior and such that $k=k_{1} \cup k_{2}$. In addition, $\alpha$ is homotopy invariant, namely $\alpha(k)=\alpha(k')$ whenever the transverse arc $k$ can be mapped onto the transverse arc $k'$ via a homotopy preserving the leaves of $\lambda$. 

In the context of this paper, we will be mostly considering transverse cocycles for the \emph{orientation cover} $\La$ of a maximal geodesic lamination $\lambda\subset S$. Here, $\La$ is the orientation cover of $\lambda$ in the sense of foliation theory, namely $\La$ is a foliation that is a $2$--cover of the foliation $\lambda$ whose leaves are  oriented in a continuous fashion. To be able to talk about transverse cocycles for the orientation cover $\La$, we need an ``ambient surface'' for $\La$, so that we can consider transverse arcs to $\La$: let $U\subset S$ be an open neighborhood of $\lambda$ obtained after puncturing the interior of each ideal triangle in $S-\lambda$; the orientation cover $\La \to \lambda$ extends to a $2$--cover $\widehat{U}\to U$. Therefore, by \emph{transverse cocycle} for $\La$, we will always mean a transverse cocycle $\alpha$ for the geodesic lamination $\La$ embedded in some open surface $\widehat{U}$ as above, though we will often omit to mention the ``ambient'' surface $\widehat{U}$ and refer to it only when needed. Let $\HO$ be the vector space of transverse cocycles for $\La$. It follows from \cite[\S5]{Bon3} that the dimension of $\HO$ is finite, with actual dimension, since $\lambda$ is maximal, equal to $12g-11$.

%When working with the oriented cover $\La$ though, we will often omit to mention the ``ambient'' surface $\widehat{U}$ and refer to it only when needed. Let $\HO$ be the vector space of transverse cocycles for $\La$. It follows from \cite[\S5]{Bon3} that the dimension of $\HO$ is finite, with actual dimension equal to $12g-11$ (since $\lambda$ is maximal). 

Let $k\subset \widehat{U}$ be a transverse, nonbacktracking, simple arc to $\La$. Orient $k$ accor\-dingly, namely so that the angle between the oriented arc $k$ and each oriented leaf of $\La$ is positively oriented. For every component $d \subset k-\La$, $k_d$ is the subarc of $k$ joining the negative endpoint of $k$ to any point contained in $d$. Pick a norm $\left \Vert  \ \right \Vert_{\HO}$ on the vector space $\HO$.

\begin{lem}
\label{lem05}
There exists some constant $C\geq 0$, depending on $k$, such that, for every transverse cocycle  $\alpha \in \HO$, for every component $d \subset k-\La$,
$$
\alpha(k_d) \leq C \left \Vert  \alpha \right \Vert_{\HO} \left ( r(d)+1\right )
$$
where $r(d)$ is the divergence radius of $d$ (see \S\ref{sect:Inequalities}).
 \end{lem}
 
 \begin{proof}
By reducing the size of the open neighborhood $\widehat{U}$ of $\La$, we can make it a \emph{train track neighborhood for $\La$}; see for instance \cite{PeH, Bon4}. The above estimate then comes as a corollary that the quantity $\alpha(k_d)$ is a linear function of the finite system of weights on the \emph{edges} of the train track $\widehat{U}$ that is determined by the transverse cocycle $\alpha$; see \cite[\S1]{Bon3} for details.
 \end{proof}

Let $\mathfrak{R}:\La\to\La$ be the \emph{orientation reversing involution}. For every $\alpha\in \HO$, $\mathfrak{R}^*\alpha$ is the \emph{pullback transverse cocycle of $\alpha$ by} $\mathfrak{R}$. We define the vector space of \emph{transverse $n$--twisted cocycles} for $\La$ as
$$
\CO=\left \{ \varepsilon=(\varepsilon_{1}, \ldots, \varepsilon_{n})\in{\CH(\widehat{\lambda})}^n /\, \mathfrak{R}^*\varepsilon_{i}=-\varepsilon_{n-i+1}\text{, }\sum \varepsilon_i = 0 \right \}.
$$
%More concretely, let $k$ be a transverse arc to $\La$, and let $R(k)$ be its image by $R:\La\to\La$. The set of leaves of $\La$ that intersects $k$, and the set of leaves of $\La$ that intersects $R(k)$, both project to the same set of geodesic leaves in $\lambda\subset S$.  Thus the set of leaves of  intersecting $k$ can be viewed as the 

\begin{lem}
\label{lem06}
The dimension of the vector space $\CO$ is equal to $(n-1)(6g-6)+ \left \lfloor\frac{n-1}{2} \right \rfloor$.
\end{lem}

\begin{proof}
Set $E=\HO$. Let $\tau: E \to E$, $\alpha \mapsto \mathfrak{R}^{*}\alpha$ be the pullback endomorphism. $\tau$ being an involution, the space $E$ splits as a direct sum $E^+\oplus E^-$, where $E^{\pm}$ is the $\pm 1$--eigenspace. One easily verifies that the subspace $E^{+}$ corresponds to the space of transverse cocycles for the maximal geodesic lamination $\lambda$, whose dimension is $6g-6$ \cite[\S5]{Bon3}; the dimension of $E^-$ is thus is $6g-5$. Therefore, for every $i=1$, $\ldots$ , $n$, we have $\varepsilon_{i}=\varepsilon_{i}^+ +\varepsilon_{i}^-$ for $\varepsilon_{i}^{\pm} \in E^\pm$. Since $\varepsilon=(\varepsilon_{1}, \ldots, \varepsilon_{n})\in \CO$, for every $i=1$, $\ldots$ , $n$,
$$
\mathfrak{R}^*\varepsilon_{i}=-\varepsilon_{n-i+1}
$$
which is equivalent to
$$
\begin{cases}
\varepsilon_{i}^+=-\varepsilon_{n-i+1}^+\\
\varepsilon_{i}^-=\varepsilon_{n-i+1}^-
\end{cases}.
$$

By an easy calculation, it follows from the above identities that the dimension of the vector space
$$
\left \{ \varepsilon=(\varepsilon_{1}, \ldots, \varepsilon_{n})\in{\CH(\widehat{\lambda})}^n / \, \mathfrak{R}^*\varepsilon_{i}=-\varepsilon_{n-i+1} \right \}
$$
is equal to $(n-1)(6g-6)+\left \lfloor\frac{n-1}{2} \right \rfloor + (6g-5)$.
Besides, observe that the se\-cond condition $\sum \varepsilon_i = 0$ is in fact a condition on the component $\varepsilon_i^-\in E^-$ only. Hence the dimension of the vector space $\CO$ is equal to  $\big [(n-1)(6g-6)+\left \lfloor\frac{n-1}{2} \right \rfloor+ $ $(6g-5)\big ]-(6g-5)=(n-1)(6g-6)+\left \lfloor\frac{n-1}{2} \right \rfloor$. 
\end{proof}

%The latter fact implies that the dimension of the vector space $\CO$ is equal to $\left [(n-1)(6g-6)+\left \lfloor\frac{n-1}{2} \right \rfloor + (6g-5)\right ]-(6g-5)=(n-1)(6g-6)+\left \lfloor\frac{n-1}{2} \right \rfloor$. 

%one-to-one

We conclude these preliminaries with recalling the correspondence between transverse cocycles and transverse H\"older distributions for a given geodesic lamination $\lambda$; see \cite[\S6]{Bon3} for details. 

A \emph{transverse H\"older distribution for} $\lambda$ assigns to every transverse arc $k$ a H\"older distribution $\alpha_k$ on $k$, namely $\alpha_k$ is a continuous, linear function defined on the space of H\"older continuous functions; similarly as for transverse cocycles, this assignment is homotopy invariant. 

Let $k$ be a transverse arc to $\lambda$. Choose an arbitrary orientation for $k$; let $x_k^+$ be the positive endpoint of the oriented arc $k$; and let $x_d^+$ and $x_d^-$ be respectively the positive and negative endpoints of each component $d\subset k-\lambda$. Finally, let $k_d$ be the subarc of $d$ joining the negative endpoint $x_k^-$ of $k$ to an arbitrary point in $d$. The key ingredient of the above correspondence is the following fundamental formula. 

%\cite[\S5]{Bon3}[Labourie \cite{La1}]

\begin{thm}[Bonahon \cite{Bon3}] \textsc{(Gap Formula)}
\label{gap}
Let $\alpha$ be a transverse H\"older cocycle for a geodesic lamination $\lambda$. For every H\"older continuous function $\varphi\colon k \to \mathbb{R}$ defined on an oriented arc $k$ transverse to $\lambda$, set
$$
\alpha(\varphi)=\alpha(k) \varphi(x_k^+) + \sum_{d \subset k-\lambda}\alpha(k_d) ( \varphi(x_d^-) -\varphi(x_d^+))
$$
where the indexing $d$ ranges over all components of $k-\lambda$ (= gaps). The above summation is convergent and defines a transverse H\"older distribution for $\lambda$.
\end{thm}

\begin{proof}
See  \cite[\S5]{Bon3}.
\end{proof}

A transverse H\"older distribution $\alpha$ for the geodesic lamination $\lambda$ defines a transverse cocycle $\alpha$ in a natural fashion: let $k$ be a transverse arc to $\lambda$; consider a H\"older continuous function $\varphi\colon k \to \mathbb{R}$ defined on the arc $k$ that is identically equal to $1$ on $k\cap \lambda$; the value of the transverse cocycle $\alpha$ at $k$ is set to be $\alpha(k)=\alpha(\varphi)$. By Theorem~\ref{gap},  this definition is valid for any $\varphi$ as above, as the value $\alpha(\varphi)$ of a transverse H\"older distribution depends only on the values achieved by $\varphi$ on $k\cap \lambda$.

%; see Theorem~\ref{gap} below.%see \cite[\S4]{Bon3}.

Conversely, given a transverse cocycle $\alpha$ for $\lambda$, Theorem~\ref{gap} enables us to reconstruct the corresponding transverse H\"older distribution.

%for every transverse arc $k$ to $\lambda$, the gap formula of Theorem~\ref{gap} defines a H\"older distribution on $k$

%enables us to reconstruct the corresponding H\"older distribution on $k$.%, hence the corresponding transverse H\"older distribution.
 
%Choose an arbitrary orientation for $k$; let $x_k^+$ be the positive endpoint of the oriented arc $k$; and let $x_d^+$ and $x_d^-$ be respectively the positive and negative endpoints of each component $d\subset k-\lambda$. Finally, let $k_d$ be the subarc of $d$ joining the negative endpoint $x_k^-$ of $k$ to an arbitrary point in $d$.

% Given a transverse cocycle $\alpha$ for $\lambda$, the following formula enables us to reconstruct the corresponding transverse H\"older distribution.

\section{Cataclysms}
\label{sect:Cataclysms}

We now tackle the construction of cataclysm deformations for Anosov representations along a maximal geodesic lamination $\lambda\subset S$. The construction will mostly take place in the universal cover $\widetilde{S}\supset\widetilde{\lambda}$. As in \S\ref{sect:TransCocycles}, $\La$ will denote the orientation cover of $\lambda$.

\subsection{Shearing map between two ideal triangles}
\label{ShearTriangle}

Consider an Anosov representation $\rho \colon \pi_{1}(S) \rightarrow \PSL_{n}(\mathbb{R})$ along with its Anosov section $\sigma_\rho=(V_1, \ldots ,V_n)$. Let $P$ and $Q$ be two ideal triangles in the complement $\widetilde{S}-\widetilde{\lambda}$. We begin with defining the \emph{shearing map} $\varphi_{PQ}\in\SL_n(\mathbb{R})$ between the triangles $P$ and $Q$.

Let $\mathcal{P}_{PQ}$ be the set of ideal triangles in $\widetilde{S}-\widetilde{\lambda}$ lying between $P$ and $Q$. Let $k$ be a simple, nonbacktracking, oriented arc transverse to $\widetilde{\lambda}$ joining a point in the interior of $P$ to a point in the interior of $Q$. Orient positively the leaves of $\widetilde{\lambda}$ intersecting $k$ for the transverse orientation determined by the oriented arc $k$. 

Let $\varepsilon=(\varepsilon_{1},  \ldots, \varepsilon_{n}) \in \CO$ be a transverse $n$--twisted cocycle for the orientation cover $\La$ of $\lambda$; see \S\ref{sect:TransCocycles}. For every triangle $R\in \mathcal{P}_{PQ}$, let $g_{R}^{-}$ and $g_{R}^{+}\subset \widetilde{\lambda}$ be the two leaves bounding $R_j$ that are  the closest to the triangles $P$ and $Q$, respectively. As in \S\ref{sect:Inequalities}, the Anosov section $\sigma_\rho=(V_1, \ldots , V_n)$ enables us to associate to the oriented geodesics $g_{R}^{-}$ and $g_{R}^{+}$ the linear maps $T_{g_{R}^{-}}^{\varepsilon(P,R)}$ and $T_{g_{R}^{+}}^{-\varepsilon(P,R)}$, respectively, where $\varepsilon(P,R)\in \mathbb{R}^n$ is defined as follows. As in \S\ref{sect:TransCocycles}, let $U\subset S$ be an open neighborhood of the maximal geodesic lamination $\lambda$ together with its associated $2$--cover $\widehat{U}\supset \widehat{\lambda}$. Consider an oriented arc $k_{PR}\subset \widetilde{U} (\subset \widetilde{S})$ transverse to $\widetilde{\lambda}$ joining a point in the interior of the triangle $P$ to a point in the interior of the triangle $R$. The arc $k_{PR}$ projects onto an oriented arc $p(k_{PR})\subset U (\subset S)$ that is transverse to $\lambda$. Observe that the geodesic lamination $\La$ and the surface $\widehat{U}$ being both oriented, $\La\subset \widehat {U}$ inherits a well-defined transverse orientation. In particular, the oriented arc  $p(k_{PR})\subset U$ admits a preferred lift $\widehat{p(k_{PR})}\subset\widehat{U}$, namely $\widehat{p(k_{PR})}$ is the unique oriented arc transverse to $\La$ that lifts $p(k_{PR})$ and that is oriented accordingly (namely, the angle between the oriented arc $p(k_{PR})$ and each of the oriented leaves of $\La$ is positively oriented). We set $\varepsilon(P,R)=\varepsilon \big (\widehat{p(k_{PR})} \big )\in \mathbb{R}^n$.

%; recall  the two leaves bounding the triangle $R_j$ that are respectively the closest to the triangles $P$ and $Q$.

Given a finite subset $\mathcal{P}=\{R_{1}, R_{2}, \ldots , R_{m} \} \subset \mathcal{P_{PQ}}$, where the indexing $j$ of $R_j$ increases as one goes from $P$ to $Q$, consider the linear map
\smallskip
$$
\varphi_{\mathcal{P}}=T_{g_{1}^{-}}^{\varepsilon(P,R_{1})}\circ T_{g_{1}^{+}}^{-\varepsilon(P,R_{1})}\circ T_{g_{2}^{-}}^{\varepsilon(P,R_{2})}\circ T_{g_{2}^{+}}^{-\varepsilon(P,R_{2})} \circ \cdots \circ T_{g_{m}^{-}}^{\varepsilon(P,R_{m})}\circ T_{g_{m}^{+}}^{-\varepsilon(P,R_{m})} \circ T_{g_{Q}^{-}}^{\varepsilon(P,Q)}
$$
where we set $g_{j}^{+}=g_{R_j}^{+}$ and $g_{j}^{-}=g_{R_j}^{-}\subset \widetilde{\lambda}$ to alleviate notation. Finally, we need a bunch of norms. Pick a norm $\left \Vert \ \right \Vert_{\HO}$ on the vector space of transverse cocycles $\HO$, and endow the vector space of transverse $n$--twisted cocycles $\CO$ with the norm $\left \Vert \varepsilon Ê\right \Vert=\mathrm{max}_{\,i}\left  \Vert \varepsilon_i Ê\right \Vert_{\small \HO}$ for $\varepsilon=(\varepsilon_1, \ldots, \varepsilon_n)\in \CO$. Likewise, let $\left \Vert \ \right \Vert_{\mathbb{R}^n}$ be the \emph{max norm} on $\mathbb{R}^n$, namely $\left \Vert X \right \Vert_{\mathbb{R}^n}=\mathrm{max}_{\,i} \left \vert x_i \right \vert$ for $X=(x_1, \ldots , x_n)\in \mathbb{R}^n$, and let $\leftn \ \rightn$ be the induced norm on $\mathrm{End}(\mathbb{R}^n)$.

\begin{pro}
\label{pro1}
For $\varepsilon\in \CO$ small enough,
$$
\lim_{\mathcal{P} \to \mathcal{P}_{PQ}}\varphi_{\mathcal{P}}
$$ 
exists and is an element of $\SL_n(\mathbb{R})$.
\end{pro}

\begin{proof}
Set 
$$
\psi_{\mathcal{P}}=T_{g_{1}^{-}}^{\varepsilon(P,R_{1})}\circ T_{g_{1}^{+}}^{-\varepsilon(P,R_{1})}\circ T_{g_{2}^{-}}^{\varepsilon(P,R_{2})}\circ T_{g_{2}^{+}}^{-\varepsilon(P,R_{2})} \circ \cdots \circ T_{g_{m}^{-}}^{\varepsilon(P,R_{m})}\circ T_{g_{m}^{+}}^{-\varepsilon(P,R_{m})}.
$$ 

If the set of ideal triangles $\mathcal{P}_{PQ}$ is finite, there is nothing to prove. We thus assume that $\mathcal{P}_{PQ}$ is an infinite set.

We begin with showing that $\psi_{\mathcal{P}}$ is uniformely bounded.
$$
\leftn \psi_{\mathcal{P}}  \rightn  \leq  \leftn T_{g_{1}^{-}}^{\varepsilon(P,R_{1})} \circ T_{g_{1}^{+}}^{-\varepsilon(P,R_{1})} \rightn   \cdots \leftn  T_{g_{m}^{-}}^{\varepsilon(P,R_{m})}\circ T_{g_{m}^{+}}^{-\varepsilon(P,R_{m})} \rightn. 
$$
By Lemma~\ref{lem04}, for every $j=1$, $\ldots$ , $m$, 
$$
\leftn T^{\varepsilon(P,R_{j})}_{g_{j}^{-}} \circ T^{-\varepsilon(P,R_{j})}_{g_{j}^{+}}-\mathrm{Id} \rightn=\mathrm{O}\left ( e^ {2 \left \Vert \varepsilon(P,R_{j})\right \Vert_{\mathbb{R}^n}-Kr(k \cap R_j)}\right)
$$
for some $K>0$ (depending on $k$ and $\rho$). Therefore, there exists some $M\geq0$ (depending on $k$ and $\rho$) such that
$$
\leftn \psi_{\mathcal{P}}  \rightn \leq \prod_{j=1}^{m}\left (1+ M e^ {2 \left \Vert \varepsilon(P,R_{j})\right \Vert_{\mathbb{R}^n}-Kr(k\cap R_{j})} \right). 
$$
The convergence of the infinite product on the right-hand side is guaranteed whenever the series $\sum_{j=1}^{m}e^ {2 \left \Vert \varepsilon(P,R_{j})\right \Vert_{\mathbb{R}^n}-Kr(k\cap R_{j})}$ is convergent. By Lemma~\ref{lem05}, 
$$
\left \Vert \varepsilon(P,R_{j})\right \Vert_{\mathbb{R}^n}=\left \Vert \varepsilon \big (\widehat{p(k_{PR_j})} \big )\right \Vert_{\mathbb{R}^n} \leq C \left \Vert  \varepsilon \right \Vert \left ( r \big (\widehat{p(k)}\cap \widehat{R}_j \big )+1\right )
$$
where $\widehat{R}_j \subset \widehat{U}$ is the lift of the (punctured) triangle $R_j\subset U$. Since $\widehat{U}\to U$ is a $2$--cover, $r \big(\widehat{p(k)}\cap \widehat{R}_j \big)\leq r(k\cap R_j)$, where $r(.)$ is the divergence radius (see \S\ref{sect:Inequalities}). Hence
$$
\sum_{j=1}^{m}e^ {2 \left \Vert \varepsilon(P,R_{j})\right \Vert_{\mathbb{R}^n}} e^{-Kr(k\cap R_{j})} \leq \sum_{j=1}^{m}e^ {2 C \left \Vert  \varepsilon \right \Vert \left ( r(p(k)\cap R_j)+1\right )} e^{-Kr(k\cap R_{j})}.
$$
By Lemma~\ref{lem01}, the above series is bounded by finitely many series of the form $\sum_{r=0}^{\infty}e^ {2 C \left \Vert \varepsilon \right \Vert(r+1)} e^{-Kr}$; this implies that $\leftn \psi_{\mathcal{P}} \rightn$ is uniformly bounded whenever $\left \Vert \varepsilon \right \Vert<K/2C$.

We now prove that $\psi_{\mathcal{P}}$ converges as $\mathcal{P}$ goes to $\mathcal{P}_{PQ}$. Let $\mathcal{P}_{m}$ be an increasing sequence of finite ideal triangles converging to $\mathcal{P}_{PQ}$ with $\mathrm{Card}(\mathcal{P}_{m})=m$. Consider the maps $\psi_{\mathcal{P}_{m}}$ and $\psi_{\mathcal{P}_{m+1}}$. Since $\mathcal{P}_{m+1}$ contains one more triangle $R$ than $\mathcal{P}_{m}$, and $\leftn \psi_{\mathcal{P}} \rightn$ is uniformely bounded,
$$
\psi_{\mathcal{P}_{m}}=\psi_{\mathcal{P}}\psi_{\mathcal{P^{'}}}\text{ and   }\,\psi_{\mathcal{P}_{m+1}}=\psi_{\mathcal{P}}\circ T_{g_{R}^{-}}^{\varepsilon(P,R)} \circ T_{g_{R}^{+}}^{-\varepsilon(P,R)}  \circ\psi_{\mathcal{P^{'}}}
$$
where $\mathcal{P}_{m}=\mathcal{P}\cup\mathcal{P}^{'}$. By Lemma~\ref{lem04},
\begin{eqnarray*}
\leftn \psi_{\mathcal{P}_{m+1}} - \psi_{\mathcal{P}_{m}}  \rightn &\leq& \leftn \psi_{\mathcal{P}}  \rightn \, \leftn T_{g_{R}^{-}}^{\varepsilon(P,R)} \circ T_{g_{R}^{+}}^{-\varepsilon(P,R)} - \mathrm{Id} \rightn \, \leftn \psi_{\mathcal{P}^{'}} \rightn  \\ \\
&\leq& M' e^ {2 \left \Vert \varepsilon(P,R)\right \Vert_{\mathbb{R}^n}-Kr(k\cap R)}
\end{eqnarray*}
for some $M'\geq0$ (depending on $k$ and $\rho$). Since $\mathcal{P}_{PQ}$ is an infinite set, Lemma~\ref{lem01} implies that $\lim_{m  \to \infty,\text{ } R\in \mathcal{P}_{m}} r(k\cap R) = \infty$. In particular, the sequence $\psi_{\mathcal{P}_{m}}$ is Cauchy,  and thus convergent whenever $\left \Vert \varepsilon \right \Vert<K/2C$. In fine, $\lim_{\mathcal{P} \to \mathcal{P}_{PQ}} \psi_{\mathcal{P}}$, and so $\lim_{\mathcal{P} \to \mathcal{P}_{PQ}} \varphi_{\mathcal{P}}=\psi_{\mathcal{P}}\circ T_{g_{Q}^{-}}^{\varepsilon(P,Q)}$, are well-defined maps for $\varepsilon\in \CO$ small enough. 
\end{proof}

The above proof also provides the following estimate, that will come handy later.
\begin{cor}
\label{cor08}
There exists some constant $B>0$, depending on $k$ and $\rho$, such that, for $\varepsilon \in \CO$ small enough, 
$$
\varphi_{PQ}=\psi_{PQ}\circ T_{g_{Q}^{-}}^{\varepsilon(P,Q)}
$$ 
where $\psi_{PQ}=\mathrm{Id}+\mathrm{O}\left (  \sum_{R \in {\mathcal{P}}_{PQ}} e^{-Br(k\cap R)}  \right )$.
\end{cor}

%\in \SL_{n}(\mathbb{R})

We emphasize the fact that the shearing map $\varphi_{PQ}\in \SL_{n}(\mathbb{R})$ is parametrized by the transverse $n$--twisted cocycle $\varepsilon \in \CO$. 

\subsection{Composition of shearing maps}
\label{ShearComposition}
Let $\varphi_{\mathcal{P}}$ be as in Proposition~\ref{pro1}. The shearing map $\varphi_{PQ}$ satisfies the following properties.

\begin{thm}
\label{ShearComp}
For $\varepsilon\in \CO$ small enough, for every plaques $P$, $Q$, $R$ of $\widetilde{S}-\widetilde{\lambda}$, the map $\varphi_{\mathcal{P}}$ converges to a linear map $\varphi_{PQ}\in \SL_n(\mathbb{R})$ as $\mathcal{P}$ tends to $\mathcal{P}_{PQ}$. In addition, $\varphi_{QP}=\varphi_{PQ}^{-1}$ and $\varphi_{PR}=\varphi_{PQ}\varphi_{QR}$.
\end{thm}

\begin{proof}
[Proof of Theorem~\ref{ShearComp}]
The demonstration will take place in several steps. We will consider an alternative description for the shearing map $\varphi_{PQ}$ for which the composition property is immediate by construction. 

Let $k$ be a transverse, simple, nonbacktracking, oriented arc to $\widetilde{\lambda}$ joining a point in the interior of $P$ to a point in the interior of $Q$. For every integer $r>0$, let $\mathcal{P}^{r}_{PQ}$ be the finite set of triangles $R\in \mathcal{P}_{PQ}$ such that the divergence radius $r(k\cap R)\leq r$; see \S\ref{sect:Inequalities}. Index the elements of $\mathcal{P}_{PQ}^{r}$ as $R_{1}$, $R_{2}$, $\ldots$ , $R_{m}$ so that the indexing $j$ of $R_{j}$ increases as one goes from $P$ to $Q$. For every $j=1$, $\ldots$ , $m$, pick a geodesic $h_{j}$ separating the interior of $R_{j}$ from the interior of $R_{j+1}$. Pick also  a geodesic $h_{0}$ between $P$ and $R_{1}$, and a geodesic $h_{m}$ between $R_{m}$ and $Q$, and orient positively the $h_{j}$ for the transverse orientation determined by the oriented arc $k$.

Set
\begin{eqnarray*}
\mathcal{\varphi}^{r}_{PQ}=T_{h_{0}}^{\varepsilon(P,R_{1})}\circ T_{h_{1}}^{\varepsilon(R_{1},R_{2})}\circ T_{h_{2}}^{\varepsilon(R_{2},R_{3})}\circ\cdots\circ T_{h_{m}}^{\varepsilon(R_{m},Q)}.
\end{eqnarray*}

\begin{pro}
\label{ShearMap2}
For $\varepsilon\in \CO$ small enough, $\mathcal{\varphi}^{r}_{PQ}$ is convergent as $r$ tends to $\infty$ and
$$
\textstyle \lim_{r \to \infty} \mathcal{\varphi}^{r}_{PQ}=\mathcal{\varphi}_{PQ}.
$$
\end{pro}

\begin{proof}
We will first estimate the difference between the map $\psi_{\mathcal{P}^{r}_{PQ}}$ of the proof of Proposition~\ref{pro1} and the map $\psi^{r}_{PQ}=\varphi^{r}_{PQ}\circ T_{h_{m}}^{-\varepsilon(PQ)}$. By reordering the terms in the expression of $\psi^{r}_{PQ}$, we have
$$
\psi^{r}_{PQ}=T_{h_{0}}^{\varepsilon(P,R_{1})}\circ T_{h_{1}}^{-\varepsilon(P,R_{1})}
\circ T_{h_{1}}^{\varepsilon(P,R_{2})}\circ T_{h_{2}}^{-\varepsilon(P,R_{2})}
\circ \cdots \circ T_{h_{m-1}}^{\varepsilon(P,R_{m})}\circ T_{h_{m}}^{-\varepsilon(P,R_{m})}
$$
and as previously,
$$
\psi_{\mathcal{P}^{r}_{PQ}}=T_{g_{1}^{-}}^{\varepsilon(P,R_{1})}\circ T_{g_{1}^{+}}^{-\varepsilon(P,R_{1})}\circ T_{g_{2}^{-}}^{\varepsilon(P,R_{2})}\circ T_{g_{2}^{+}}^{-\varepsilon(P,R_{2})} \circ \cdots \circ T_{g_{m}^{-}}^{\varepsilon(P,R_{m})}\circ T_{g_{m}^{+}}^{-\varepsilon(P,R_{m})}.
$$
Note that the map $\psi^{r}_{PQ}$ is obtained from $\psi_{\mathcal{P}^{r}_{PQ}}$ by replacing each term $T_{g_{j}^{-}}^{\varepsilon(P,R_{j})}\circ T_{g_{j}^{+}}^{-\varepsilon(P,R_{j})}$ by $T_{h_{j-1}}^{\varepsilon(P,R_{j})}\circ T_{h_{j}}^{-\varepsilon(P,R_{j})}$.

% \cite{PeH, Bon4} 

Consider a train track neighborhood of the maximal geodesic lamination $\widetilde{\lambda}$ for which the transverse arc $k$ is a \emph{switch}; see \cite{PeH, Bon4}. 

\begin{lem}
\label{edgepath}
The two geodesics $g_{j}^{+}$ and $g_{j+1}^{-}\subset \widetilde{\lambda}$ follow the same edge-path of length $2r$ in the above train track.
\end{lem}

% associated to the transverse arc $k$.

\begin{proof}
If they do not, there exists an ideal triangle $R$ between $R_{j}$ and $R_{j+1}$ whose sides $g_{R}^{-}$ and $g_{R}^+$ follow the same edge-paths of length $2r$ as $g_{j}^{+}$ and $g_{j+1}^{-}$, respectively. $g_{R}^-$ and $g_{R}^+$ being  asymptotic, it implies that $g_{R}^-$ and $g_{R}^+$ must follow the same edge-path of length $r$, hence $R\in \mathcal{P}^{r}_{PQ}$ which contradicts the assumption.
\end{proof}

In one hand, since the geodesic $h_{j}$ lies between $g_{j}^{+}$ and $g_{j+1}^{-}$, it follows the same edge-path of length $2r$ in the train track. In particular, the distance between any two of these three geodesics is thus a $\mathrm{O}(e^{-Ar})$ for some $A\geq0$ (depending on $k$). On the other hand, by Lemma~\ref{edgepath} and Lemma~\ref{lem02}, the distance between $g_{j}^{-}$ and $g_{j}^{+}$ is a $\mathrm{O}(e^{-Ar(k\cap R_{j})})$. Recall that $R_{j}\in \mathcal{P}^r_{PQ}$, so $r(k \cap R_{j})\leq r$. The above discussion implies that the distance between $h_{j}$ and $h_{j+1}$ is also a $\mathrm{O}(e^{-Ar(k\cap R_{j})})$. Following the arguments in the proof of Lemma~\ref{lem04}, the previous estimates show that we can find some constants  $M\geq0$ and $K>0$ (both depending on $k$ and $\rho$) for which, for every $j$,
$$
\leftn T_{h_{j}}^{\varepsilon(P,R_{j})} \circ T_{h_{j+1}}^{-\varepsilon(P,R_{j})} -\mathrm{Id}\rightn \leq Me^ {2 C \left \Vert \varepsilon\right \Vert (r(k \cap R_{j})+1)}e^{-K r(k\cap R_{j})}
$$
and 
$$
\leftn T_{g_{j}^{-}}^{\varepsilon(P,R_{j})}\circ T_{g_{j}^{+}}^{-\varepsilon(P,R_{j})} -\mathrm{Id}\rightn \leq Me^ {2 C \left \Vert \varepsilon\right \Vert (r(k \cap R_{j})+1)}e^{-K r(k\cap R_{j})}.
$$

Let $\psi$ be any map obtained from $\psi_{\mathcal{P}^{r}_{PQ}}$ by replacing some of the $m$ terms  $T_{g_{j}^{-}}^{\varepsilon(P,R_{j})}\circ T_{g_{j}^{+}}^{-\varepsilon(P,R_{j})}$ by $T_{h_{j-1}}^{\varepsilon(P,R_{j})}\circ T_{h_{j}}^{-\varepsilon(P,R_{j})}$ or by the identity. As in the proof of Proposition~\ref{pro1}, it follows from the latter estimates that  
\begin{eqnarray*}
\log \leftn \psi \rightn &=&\mathrm{O}\left ( \sum_{j=1}^{m}e^ {2 C \left \Vert \varepsilon\right \Vert (r(k \cap R_{j})+1)} e^{-K r(k \cap R_{j})} \right )\\ \\
&=&\mathrm{O}\left (\sum_{r=0}^{\infty} e ^{2 C \left \Vert \varepsilon\right \Vert (r+1)} e ^{-K r}\right ).
\end{eqnarray*}
Consequently, the norm of such a map $\psi$ is uniformely bounded, whenever $\left \Vert \varepsilon \right \Vert\leq K/2C$.

Let $\psi_{l}$ be obtained from $\psi_{\mathcal{P}_{PQ}^{r}}$ by replacing each $T_{g_{j}^{-}}^{\varepsilon(P,R_{j})}\circ T_{g_{j}^{+}}^{-\varepsilon(P,R_{j})}$ with $j\leq l$ by $T_{h_{j-1}}^{\varepsilon(P,R_{j})}\circ T_{h_{j}}^{-\varepsilon(P,R_{j})}$, so that $\psi_{0}=\psi_{\mathcal{P}_{PQ}^{r}}$ and $\psi_{m}=\psi^{r}_{PQ}$. Again, as in the proof of Proposition~\ref{pro1}, we estimate the difference between $\psi_{l-1}$ and $\psi_{l}$. We have $\psi_{l-1}=\psi \circ T_{g_{l}^{-}}^{\varepsilon(P,R_{l})} \circ T_{g_{l}^{+}}^{-\varepsilon(P,R_{l})} \circ \psi'$ and  $\psi_{l}=\psi \circ T_{h_{l-1}}^{\varepsilon(P,R_{l})}\circ T_{h_{l}}^{-\varepsilon(P,R_{l})} \circ \psi'$, where $\psi$ and $\psi'$ are obtained from replacing some $T_{g_{j}^{-}}^{\varepsilon(P,R_{j})}\circ T_{g_{j}^{+}}^{-\varepsilon(P,R_{j})}$ by $T_{h_{j-1}}^{\varepsilon(P,R_{j})}\circ T_{h_{j}}^{-\varepsilon(P,R_{j})}$ or the identity. As observed above, $\leftn \psi  \rightn$ and $\leftn \psi'  \rightn$ are uniformely bounded. Hence
\begin{eqnarray*}
\leftn \psi_{l-1}-\psi_{l}  \rightn &\leq& \leftn \psi  \rightn \, \leftn T_{g_{l}^{-}}^{\varepsilon(P,R_{l})} \circ T_{g_{l}^{+}}^{-\varepsilon(P,R_{l})} - T_{h_{l-1}}^{\varepsilon(P,R_{l})}\circ T_{h_{l}}^{-\varepsilon(P,R_{l})} \rightn \, \leftn \psi' \rightn\\ \\
&=&\mathrm{O} \left ( e^{4C\left \Vert \varepsilon \right \Vert (r+1)-K r}  \right ).
\end{eqnarray*}
Therefore,
$$
\leftn \psi^{r}_{PQ}-\psi_{\mathcal{P}_{PQ}^{r}}  \rightn=\leftn \psi_{m}-\psi_{0} \rightn \leq m\, \mathrm{O} \left ( e^{4C\left \Vert \varepsilon \right \Vert (r+1)-K r}  \right )=\mathrm{O}\left (r e^{4C\left \Vert \varepsilon \right \Vert (r+1)-K r}  \right )
$$
since $m=\mathrm{Card}(\mathcal{P}^r_{PQ})=\mathrm{O}(r)$ by Lemma~\ref{lem01}.
We conclude that $\psi^{r}_{PQ}$ and $\psi_{\mathcal{P}_{PQ}^{r}}$ have the same limit as $r$ tends to $\infty$ whenever $\left \Vert  \varepsilon \right \Vert < K/4C$ (recall that, by Proposition~\ref{ShearMap2}, $\psi_{\mathcal{P}_{PQ}^{r}}$ converges whenever $\left \Vert \varepsilon \right \Vert<K /2C$). At last, observe that $h_{m}$ converges to $g_{Q}^{-}$, which implies that both $\varphi^{r}_{PQ}=\psi^{r}_{PQ}\circ T_{h_{m}}^{\varepsilon(P,Q)}$ and $\varphi_{\mathcal{P}_{PQ}^{r}}=\psi_{\mathcal{P}_{PQ}^{r}}\circ T_{g_{Q}^{-}}^{\varepsilon(P,Q)}$ converge to the same limit $\varphi_{PQ}$. 
\end{proof}

\begin{cor}
\label{cor008}
Let $k$ be a transverse, simple, nonbacktracking, oriented arc to $\widetilde{\lambda}$. For $\varepsilon\in \CO$ small enough, for every triangles $P$, $Q$, $R\subset \widetilde{S}-\widetilde{\lambda}$ intersecting the arc $k$, $\varphi_{QP}=(\varphi_{PQ})^{-1}$ and $\varphi_{PR}=\varphi_{PQ}\varphi_{QR}$.
\end{cor}

\begin{proof}
Let $k$ be as above. First, suppose that the oriented arc $k$ intersects the triangles $P$, $Q$ and $R$ in this order. Then, the composition property $\varphi_{PR}=\varphi_{PQ}\varphi_{QR}$ is a straightforward consequence of the 	alternative description of Proposition~\ref{ShearMap2} for the shearing map $\varphi_{PQ}$.

%(exactly) 

Likewise, let $\mathfrak{R}(k)$ be the arc $k$, but oriented in the opposite direction; in particular, the oriented arc $\mathfrak{R}(k)$ is oriented from $Q$ to $P$. Orient positively the leaves of $\widetilde{\lambda}$  that intersects $\mathfrak{R}(k)$ for the transverse orientation determined by the oriented arc $\mathfrak{R}(k)$. Then, with the same notations as in Proposition~\ref{ShearMap2}, we have that $\mathcal{\varphi}_{QP}=\textstyle \lim_{r \to \infty} \mathcal{\varphi}^{r}_{QP}$, where
\begin{eqnarray*}
\mathcal{\varphi}^{r}_{QP}=T_{\mathfrak{R}(h_{m})}^{\varepsilon(Q,R_{m})}\circ T_{\mathfrak{R}(h_{m-1})}^{\varepsilon(R_{m},R_{m-1})}\circ\cdots\circ  T_{\mathfrak{R}(h_{1})}^{\varepsilon(R_{2},R_{1})}\circ T_{\mathfrak{R}(h_{m})}^{\varepsilon(R_{1},P)}
\end{eqnarray*}
with the difference that each oriented geodesic $h_j$ has been replaced by $\mathfrak{R}(h_j)$, which denotes the same geodesic $h_j$ but oriented in the opposite direction. Let us consider the general term $T_{\mathfrak{R}(h_{j})}^{\varepsilon(R_{j+1},R_{j})}$. Since $\varepsilon\in \CO$, it follows from the definition of $\varepsilon(R_{j+1},R_{j})=\big(\varepsilon_1(R_{j+1},R_{j}), \ldots , \varepsilon_n(R_{j+1},R_{j})\big)$ (see \S\ref{ShearTriangle}) that
\begin{eqnarray*}
\varepsilon(R_{j+1},R_{j})&=&(\mathfrak{R}^*\varepsilon)(R_{j},R_{j+1})\\
&=&\big (-\varepsilon_n(R_{j},R_{j+1}), \ldots , -\varepsilon_1(R_{j},R_{j+1}) \big).
\end{eqnarray*}
Moreover, by Lemma~\ref{lem:LineReversing}, the line decomposition associated with the oriented geodesic $\mathfrak{R}(h_{j})$ is $\widetilde{V}_{n}(h_{j}) \oplus \cdots \oplus \widetilde{V}_{1}(h_{j})=\mathbb{R}^n$. As a result, by definition of the linear map $T_{\mathfrak{R}(h_{j})}^{\varepsilon(R_{j+1},R_{j})}$ (see \S\ref{sect:Inequalities}),
$$
T_{\mathfrak{R}(h_{j})}^{\varepsilon(R_{j+1},R_{j})}=\Big ( T_{h_{j}}^{\varepsilon(R_{j},R_{j+1})}\Big)^{-1}
$$ 
and we conclude immediately that $\mathcal{\varphi}_{QP}=\textstyle \lim_{r \to \infty} \mathcal{\varphi}^{r}_{QP}=(\mathcal{\varphi}_{PQ})^{-1}$. The general case follows from these two special cases. 
\end{proof}

In all previous statements, the size of the transverse $n$--twisted cocycle $\varepsilon\in \CO$ depends on the considered transverse arc $k$ and on the Anosov representation $\rho$. 

\begin{lem}
\label{lem11}
For $\varepsilon\in \CO$ small enough, for every triangles $P$, $Q$, $R$ of $\widetilde{S}-\widetilde{\lambda}$, the map $\varphi_{\mathcal{P}}$ converges to a linear map $\varphi_{PQ}\in \SL_n(\mathbb{R})$, as $\mathcal{P}$ tends to $\mathcal{P}_{PQ}$. In addition, $\varphi_{QP}=\varphi_{PQ}^{-1}$ and $\varphi_{PR}=\varphi_{PQ}\varphi_{QR}$.
\end{lem}

\begin{proof}
Pick in the surface $S$ finitely many tranverse arcs $k_{1}$, $\ldots$ , $k_{N}$ to $\lambda$ such that each component of $S-\lambda$ meets at least one of the $k_{i}$. Given two triangles $P$ and $Q$ in $\widetilde{S}-\widetilde{\lambda}$, there is a finite sequence of triangles $R_{0}=P$, $R_{1}$, $\ldots$ , $R_{N}$, $R_{N+1}=Q$ such that each $R_{j}$ separates $R_{j-1}$ from $R_{j+1}$, and such that $R_{j}$ and $R_{j+1}$ meets the same lift $\widetilde{k}_{i_{j}}$. Choose $\varepsilon$ small enough so that the convergence of the $\varphi_{R_{j}, R_{j+1}}$ is guaranteed for every $j=1$, $\ldots$ , $N$. It follows that $\varphi_{PQ}=\lim_{\mathcal{P}\to\mathcal{P}_{PQ}}\varphi_{\mathcal{P}}$ exists, and is equal to $\varphi_{R_{0}R_{1}}\varphi_{R_{1}R_{2}} \ldots \varphi_{R_{N}R_{N+1}}$.
\end{proof}

This achieves the proof of Theorem~\ref{ShearComp}.
\end{proof}

As a consequence of Lemma~\ref{lem11}, we can find $\varepsilon \in \CO$ small enough for which the shearing map $\varphi_{PQ}$ is defined for every triangles $P$ and $Q$ in the complement $\St-\widetilde{\lambda}$, and satisfies the composition property of Lemma~\ref{cor008}. The condition ``$\varepsilon \in \CO$ small enough" becomes a condition depending on the Anosov representation $\rho$ only.

%the condition ``$\varepsilon \in \CO$ small enough" becomes a condition depending on the Anosov representation $\rho$ only. 

\subsection{Cataclysm deformations}
\label{CataDeformations}

Let  $\varepsilon\in \CO$ be a transverse $n$--twisted cocycle sufficiently small. Fix a triangle $P_0\subset \widetilde{S}-\widetilde{\lambda}$. The $\varepsilon$--\emph{cataclysm deformation of the Anosov representation $\rho$ along the maximal geodesic lamination $\lambda$} is the homomorphism $\Lambda^{\varepsilon}\rho\colon \pi_1(S)\to \PSL_n(\mathbb{R})$ defined as follows: for every $\gamma\in \pi_1(S)$,
$$
\Lambda^{\varepsilon}\rho(\gamma)=\varphi_{P_0\gamma P_0}\circ \rho(\gamma)\\
$$
where $\varphi_{P_0\gamma P_0}$ is the shearing map between the two triangles $P_0$ and $\gamma P_0\subset \St-\widetilde{\lambda}$; see \S\ref{ShearTriangle}.

We must verify that $\Lambda^{\varepsilon}\rho\colon\pi_1(S)\to \PSL_n(\mathbb{R})$ is a group homomorphism. Put $\rho'=\Lambda^{\varepsilon}\rho$. By definition of the shearing map $\varphi_{PQ}$, one easily verifies that it satisfies the following equivariant property: for every $\gamma\in \pi_1(S)$, for every $P$, $Q\subset \widetilde{S}-\widetilde{\lambda}$, $\varphi_{\gamma P \gamma Q}=\rho(\gamma)\circ\varphi_{PQ}\circ \rho(\gamma)^{-1}$. Thus, for every $\gamma_1$, $\gamma_2\in \pi_1(S)$,
\begin{eqnarray*}
\rho'(\gamma_1\gamma_2)&=&\varphi_{P_0\gamma_1\gamma_2 P_0} \circ \rho(\gamma_1\gamma_2)\\
&=&\varphi_{P_0\gamma_1 P_0}\circ\varphi_{\gamma_{1} P_0\gamma_1 \gamma_2 P_0}\circ \rho(\gamma_1)\circ\rho(\gamma_2)\\
&=&\varphi_{P_0\gamma_1 P_0}\circ\rho(\gamma_1)\circ \varphi_{ P_0\gamma_2 P_0}\circ\rho(\gamma_2)\\
&=&\rho'(\gamma_1)\rho'(\gamma_2).
\end{eqnarray*}

Note that a different choice of triangle $P_0\subset \widetilde{S}-\widetilde{\lambda}$ yields another homomorphism $\rho''$ that is conjugate to the previous $\rho'$; in particular, $\rho'=\Lambda^{\varepsilon}\rho$ defines without any ambiguity a point in the character variety $\Rep_{\PSL_{n}(\mathbb{R})}(S)$.

Recall that the set of Anosov representations $\Rep^{\text{Anosov}}_{\PSL_{n}(\mathbb{R})}(S)$  is open in the character variety $\Rep_{\PSL_{n}(\mathbb{R})}(S)$ \cite{La1, GuiW2}. We can now state the main result of this section.

\begin{thm}
\label{CataTheorem}
Let $\rho$ be an Anosov representation. There exist a neighborhood $\mathcal{U}^\rho$ of $0\in \CO$, and a continuous, injective map
\begin{eqnarray*}
\Lambda\colon \mathcal{U}^\rho& \to &\Rep^{\mathrm{Anosov}}_{\PSL_{n}(\mathbb{R})}(S)\\
\varepsilon=(\varepsilon_1, \ldots , \varepsilon_n)&\mapsto&\Lambda^{\varepsilon}\rho
\end{eqnarray*}
such that $\Lambda^{0}\rho=\rho$.
\end{thm}

We will refer to the map $\Lambda\colon \mathcal{U}^\rho \to \Rep^{\mathrm{Anosov}}_{\PSL_{n}(\mathbb{R})}(S)$ as the \emph{cataclysm map based at $\rho$ along the maximal geodesic $\lambda$}. The transverse $n$--twisted cocycle $\varepsilon\in\CO$ is the \emph{shear parameter} of the cataclysm deformation $\Lambda^{\varepsilon}\rho$; it determines the ``magnitude'' of the cataclysm. The injectivity of the cataclysm map $\Lambda$ will be proved in \S\ref{DiffCocy}; see Corollary~\ref{Cor36}.

\section{Cataclysms and flag curves}
\label{Sect:Catac}

We now study the effect of a cataclysm deformation on the associated equivariant flag curve $\mathcal{F}_{\rho}\colon\Sinf \rightarrow \mathrm{Flag}(\mathbb{R}^{n})$ of some Anosov representation $\rho$.

Let $\rho'=\Lambda^{\varepsilon}\rho$ be the $\varepsilon$--cataclysm deformation of $\rho$ along a maximal geodesic lamination $\lambda\subset S$ for some transverse $n$--twisted cocycle $\varepsilon \in \CO$ small enough. Fix an ideal triangle $P_0\subset \widetilde{S}-\widetilde{\lambda}$, and consider the equivariant family of shearing maps $\{\varphi_{P_0P}\}_{P\subset \widetilde{S}-\widetilde{\lambda}}\subset \SL_n(\mathbb{R})$; see \S\ref{ShearTriangle}. Let $\mathcal{V}_{\lambda}\subset \Sinf$ be the set of vertices of the ideal triangles in $\widetilde{S}-\widetilde{\lambda}$; note that the set $\mathcal{V}_{\lambda}\subset \Sinf$ is $\pi_1(S)$--invariant. For every $x\in \mathcal{V}_{\lambda}$, let $P\subset \widetilde{S}-\widetilde{\lambda}$ be an ideal triangle whose $x$ is a vertex, and let $\varphi_{P_0P}\in \PSL_n(\mathbb{R})$ be the associated shearing map. Set
$$
\mathcal{F}_{\rho}'(x)=\varphi_{P_0P}\big(\mathcal{F}_{\rho}(x) \big)
$$
where $\mathcal{F}_{\rho}(x)$ is the image of the vertex $x\in \mathcal{V}_{\lambda}$ by the flag curve $\mathcal{F}_{\rho}\colon\Sinf \rightarrow \mathrm{Flag}(\mathbb{R}^{n})$.

\begin{lem}
\label{FlagExt1}
The above relation defines a $\rho'$--equivariant flag map $\mathcal{F}_{\rho}'\colon \mathcal{V}_{\lambda}\to \mathrm{Flag} \big ( \mathbb{R}^n \big )$. 
\end{lem}

\begin{proof}
There is an ambiguity in the definition of $\mathcal{F}_{\rho}'(x)$, as a vertex $x\in \mathcal{V}_{\lambda}$ may belong to several ideal triangles  in $\widetilde{S}-\widetilde{\lambda}$. Suppose that there is another triangle $Q\subset \widetilde{S}-\widetilde{\lambda}$ whose $x$ is a vertex, and let us compare the images $\varphi_{P_0P}\big(\mathcal{F}_{\rho}(x)\big)$ and $\varphi_{P_0Q}\big(\mathcal{F}_{\rho}(x)\big)$. By the composition property of Lemma~\ref{lem11}, 
$$
\varphi_{P_0Q}\big(\mathcal{F}_{\rho}(x)\big)= \varphi_{P_0P}\big (\varphi_{PQ}\big(\mathcal{F}_{\rho}(x)\big)\big).
$$ 
Observe that, since $P$ and $Q$ share the same vertex $x\in \mathcal{V}_{\lambda}$, any ideal triangle $R$ between $P$ and $Q$ admits $x$ as one of its vertices. Therefore, for every such triangle $R$, both the linear maps $T_{g_R^{-}}^{\varepsilon(P,R)}$ and $T_{g_R^{+}}^{-\varepsilon(P,R)}$ (see \S\ref{ShearTriangle}) fix the flag $\mathcal{F}_{\rho}(x)$. It follows from the definition of the shearing map $\varphi_{PQ}$ that $\varphi_{PQ}\big(\mathcal{F}_{\rho}(x)\big)=\mathcal{F}_{\rho}(x)$, and thus that $\varphi_{P_0Q}\big(\mathcal{F}_{\rho}(x)\big)= \varphi_{P_0P}\big(\mathcal{F}_{\rho}(x)\big)$. The $\rho'$--equivariance comes as a straightforward consequence of the equivariance property of the flag curve $\mathcal{F}_\rho$ and of the family of shearing maps $\{\varphi_{P_0P}\}_{P\subset \widetilde{S}- \widetilde{\lambda}}$.
\end{proof}

Let ${\partial_{\infty}\widetilde{\lambda}}\subset \Sinf$ be the set of ideal endpoints of all leaves contained in the geodesic lamination $\widetilde{\lambda}\subset \widetilde{S}$; note that ${\partial_{\infty}\widetilde{\lambda}}\supset \mathcal{V}_\lambda$. We wish to extend the previous flag map $\mathcal{F}_{\rho}' \colon \mathcal{V}_{\lambda}\to \mathrm{Flag}(\mathbb{R}^n)$ to a flag map $\mathcal{F}_{\rho}' \colon \partial_{\infty}\widetilde{\lambda} \to \mathrm{Flag}(\mathbb{R}^n)$. To this end, we generalize the way to define $\mathcal{F}_{\rho}' \colon \mathcal{V}_{\lambda}\to \mathrm{Flag}(\mathbb{R}^n)$ of Lemma~\ref{FlagExt1}.

Let $g \subset \widetilde{\lambda}$ be a geodesic leaf. Consider a triangle $P\subset\widetilde{S}-\widetilde{\lambda}$ such that there exists a simple, nonbacktracking, oriented arc $k$ transverse to $\widetilde{\lambda}$ joining a point in the interior of $P_0$ to a point in the interior of $P$, and intersecting the leaf $g$. Orient positively the leaves of $\widetilde{\lambda}$ for the transverse orientation defined by the oriented arc $k$. Let $\mathcal{P}_{P_0g}$ be the set of ideal triangles of $\widetilde{S}-\widetilde{\lambda}$ lying between the ideal triangle $P_0$ and the geodesic $g$. Similarly as in \S\ref{ShearTriangle}, set
$$
\psi_{\mathcal{P}}= \prod_{j=1}^m \left (T_{g_{j}^{-}}^{\varepsilon(P_0,R_j)}\circ T_{g_{j}^{+}}^{-\varepsilon(P_0,R_j)} \right )
$$
where $\mathcal{P}=\{R_{1}, R_{2}, \ldots , R_{m} \} \subset \mathcal{P}_{P_0g}$ is a finite subset, and where the indexing $j$ of $R_j$ increases as one goes from $P_0$ to $g$. 

\begin{lem}
\label{ShearingMaps}
For $\varepsilon\in \CO$ small enough, for every leaf $g\subset \widetilde{\lambda}$,
$$
\psi_{P_0g}=\lim_{\mathcal{P} \to \mathcal{P}_{P_0g}}\psi_{\mathcal{P}}
$$ 
exists and is an element of $\SL_n(\mathbb{R})$.
\end{lem}

%\in\SL_n(\mathbb{R})

\begin{proof}
By Theorem~\ref{ShearComp}, whenever $\varepsilon\in\CO$ is small enough, for every $P \subset \St-\widetilde{\lambda}$, the linear map $\psi_{P_0P}$ is well defined. The map $\psi_{P_0g}$ is obtained  via the same infinite product that defines $\psi_{P_0P}$, with the difference that some factors are replaced by the identity: it is thus convergent for $\varepsilon \in \CO$ small enough. Besides, similarly as in Corollary~\ref{cor08}, the following estimate
$$
\psi_{P_0g}=\mathrm{O}\left (  \sum_{R \in {\mathcal{P}}_{P_0g}} e^{-Br(k\cap R)}  \right )
$$
holds for some $B>0$ (depending on $k$ and on $\rho$).
\end{proof}

Having defined the family of linear maps $\{\psi_{P_0g}\}_{g\subset \widetilde{\lambda}}$, for every $x \in {\partial_{\infty}\widetilde{\lambda}}$, set
$$
\mathcal{F}_{\rho}'(x)=\psi_{P_0g}\big (\mathcal{F}_{\rho}(x) \big)
$$
where $g\subset \widetilde{\lambda}$ is a geodesic whose $x$ is an endpoint.

\begin{lem}
\label{lem31}
The above relation defines a $\rho'$--equivariant flag map $\mathcal{F}_{\rho}' \colon \partial_{\infty}\widetilde{\lambda} \to \mathrm{Flag}(\mathbb{R}^n)$ that extends the flag map $\mathcal{F}_{\rho}' \colon \mathcal{V}_{\lambda} \to \mathrm{Flag}(\mathbb{R}^n)$ of Lemma~\ref{FlagExt1}.
\end{lem}
\begin{proof}
Again, we must check that there is no ambiguity in the definition of the map $\mathcal{F}_{\rho}' \colon\partial_{\infty}\widetilde{\lambda} \to \mathrm{Flag}(\mathbb{R}^n)$, and that the newly defined flag map coincides with the map $\mathcal{F}_{\rho}' \colon \mathcal{V}_{\lambda} \to \mathrm{Flag}(\mathbb{R}^n)$ of Lemma~\ref{FlagExt1}.

Observe that if $x\in {\partial_{\infty}\widetilde{\lambda}}-\mathcal{V}_\lambda$, there is a unique geodesic $g\subset \widetilde{\lambda}$ with $x$ as an endpoint. The above relation thus associates to such a point $x$ a unique flag $\mathcal{F}_{\rho}'(x)\in \mathrm{Flag}(\mathbb{R}^n)$.

Now, suppose that $x\in \mathcal{V}_\lambda$, namely $g$ is one of two edges $g^{-}_{P}$ and $g^{+}_{P}$ bounding some triangle $P\subset \widetilde{S}-\widetilde{\lambda}$.
Let $\varphi_{P_0P}$ be the shearing map associated with $P$. We must verify that 
$$
\psi_{P_0g^{\pm}_{P}} \big (\mathcal{F}_{\rho}(x) \big)=\varphi_{P_0P} \big (\mathcal{F}_{\rho}(x) \big ).
$$
If $g=g^{-}_{P}$, then 
$$  
\psi_{P_0g^{-}_{P}} \big (\mathcal{F}_{\rho}(x)\big )=\psi_{P_0P}\big (\mathcal{F}_{\rho}(x)\big )=\psi_{P_0P}\circ T_{g_{P}^{-}}^{\varepsilon(P_0,P)}\big (\mathcal{F}_{\rho}(x)\big )=\varphi_{P_0P}\big (\mathcal{F}_{\rho}(x)\big )
$$
since the flag $\mathcal{F}_{\rho}(x)$ is fixed by the linear map $T_{g_{P}^{-}}^{\varepsilon(P_0,P)}$. If $g = g^{+}_{P}$,
\begin{eqnarray*}
\psi_{P_0,g^{+}_{P}}\big (\mathcal{F}_{\rho}(x)\big )=\psi_{P_0P}\circ T_{g_{P}^{-}}^{\varepsilon(P_0,P)}\circ T_{g_{P}^{+}}^{-\varepsilon(P_0,P)}\big (\mathcal{F}_{\rho}(x)\big )&=&\varphi_{P_0P}\big (\mathcal{F}_{\rho}(x)\big )
\end{eqnarray*}
since the flag $\mathcal{F}_{\rho}(x)$ is fixed by $T_{g^{+}_{P}}^{-\varepsilon(P_0,P)}$. As a result, $\mathcal{F}_{\rho}' \colon \partial_{\infty}\widetilde{\lambda} \to \mathrm{Flag}(\mathbb{R}^n)$ is a well-defined map, that extends the previous map  $\mathcal{F}_{\rho}' \colon \mathcal{V}_{\lambda} \to \mathrm{Flag}(\mathbb{R}^n)$  of Lemma~\ref{FlagExt1}. 

In particular, the restriction ${\mathcal{F}_{\rho}'}_{|\mathcal{V}_{\lambda}} \colon  \mathcal{V}_{\lambda}\to\mathrm{Flag}(\mathbb{R}^n)$ is $\rho'$--equivariant. Let $x\in {\partial_{\infty}\widetilde{\lambda}}-\mathcal{V}_\lambda$ be an endpoint of the leaf $g\subset \widetilde{\lambda}$, and let $(g_n)_n\subset \widetilde{\lambda}$ be a sequence of leaves converging to $g$, where each $g_n$ bounds some triangle $R_n\subset \widetilde{S}-\widetilde{\lambda}$. Since $\lim_{n\to \infty}\psi_{Pg_n}=\psi_{Pg}$, and $\mathcal{F}_{\rho}  \colon \Sinf  \to \mathrm{Flag}(\mathbb{R}^n)$ is continuous,
$$
\lim_{n\to \infty}\psi_{Pg_n}\big (\mathcal{F}_{\rho}(x_{g_n})\big )=\psi_{Pg}\big (\mathcal{F}_{\rho}(x)\big )
$$
where $(x_{g_n})_n\subset \partial_{\infty}\widetilde{\lambda}$ is a sequence of endpoints of $(g_n)_n$ that converges to $x\in {\partial_{\infty}\widetilde{\lambda}}-\mathcal{V}_\lambda$. The $\rho'$--equivariance property thus extends to the flag map $\mathcal{F}_{\rho}' \colon \partial_{\infty}\widetilde{\lambda} \to \mathrm{Flag}(\mathbb{R}^n)$ by limiting process.
\end{proof}

Now, let $\mathcal{F}_{\rho'} \colon \Sinf\rightarrow\mathrm{Flag}(\mathbb{R}^n)$ be the equivariant flag curve associated with the Anosov representation $\rho'=\Lambda^{\varepsilon}\rho$. 
\begin{thm}
\label{CataFlag}
The restriction ${\mathcal{F}_{\rho'}}_{| \partial_{\infty}\widetilde{\lambda}}\colon \partial_{\infty}\widetilde{\lambda}\to\mathrm{Flag}(\mathbb{R}^n)$ coincides with the flag map $\mathcal{F}_{\rho}' \colon\partial_{\infty}\widetilde{\lambda}\to \mathrm{Flag}(\mathbb{R}^n)$ of Lemma~\ref{lem31}.
\end{thm}

\begin{proof}
[Proof of Theorem~\ref{CataFlag}]

It is convenient to switch back to the Anosov section point of view. Indeed, it is the Anosov dynamics which makes everything work here.

We begin with a lemma. Let $T^1S\times_{\rho} \bar{\mathbb{R}}^n\to T^1S$ be the flat $\bar{\mathbb{R}}^n$--bundle of an Anosov representation $\rho$; see \S\ref{Rbundledescription}. Let $(G_t)_{t\in \mathbb{R}}$ be the flow on $T^1S\times_{\rho} \bar{\mathbb{R}}^n$ that lifts the geodesic flow $(g_t)_{t\in\mathbb{R}}$ on $T^1S$. The Anosov section $\sigma_{\rho}=(V_1, \ldots , V_n)$ provides a line decomposition $V_1\oplus \cdots\oplus V_n$ of $T^1S\times_{\rho} \bar{\mathbb{R}}^n\to T^1S$ with the property that each line sub-bundle $V_i\to T^1S$ is invariant under the action of the flow $(G_t)_{t\in \mathbb{R}}$. Finally, pick a Riemannian metric $\left \Vert \ \right \Vert_{u}$ on $T^1S\times_{\rho} \bar{\mathbb{R}}^n\to T^1S$.

\begin{lem}
\label{FlowingVector}
For every $i>j$, for every $u\in T^1S$, for every vectors $X_i(u)\in V_i(u)$ and $X_j(u)\in V_j(u)$, 
$$
\lim_{t\to + \infty}\frac{\left \Vert G_t X_j(u)\right \Vert_{g_t(u)}}{\left \Vert G_t X_i(u)\right \Vert_{g_t(u)}} =0.
$$
\end{lem} 

\begin{proof}
The lift $\leftn  \ \rightn_{\ut}$ of $\leftn  \ \rightn_{u}$ defines a $\pi_1(S)$--invariant norm on $\mathbb{R}^n$. Let $\leftn  \ \rightn_{\ut}$ be the induced norm on the vector space of linear endomorphisms $\mathrm{End}(\mathbb{R}^n)$, namely, for every $\psi \in \mathrm{End}(\mathbb{R}^n)$, for every $\ut \in T^1 \St$,
\begin{eqnarray*}
\leftn  \psi \rightn_{\ut}=\sup_{X \in \mathbb{R}^n}  \frac{\left \Vert \psi X \right \Vert_{\ut}}{\left \Vert X \right \Vert_{\ut}}.
\end{eqnarray*}
By construction, $\leftn  \ \rightn_{\ut}$ is $\pi_1(S)$--invariant, and thus descends to a metric $\leftn  \ \rightn_{u}$ on the flat bundle $T^1S\times_\rho \mathrm{End}(\mathbb{R}^n)\to T^1S$. In particular, by restricting, $\leftn  \ \rightn_{u}$ provides a metric on each line sub-bundle $V_i^*\otimes V_j\to T^1S$ of the bundle $T^1S\times_\rho \mathrm{End}(\mathbb{R}^n)\to T^1S$; see \S\ref{bundledescription}.

Given $X_i(u)\in V_i(u)$ and $X_j(u)\in V_j(u)$, consider the vector $(X_i(u))^*\otimes X_j(u)\in V_i^*\otimes V_j(u)$. Recall that the action of the flow $(\bar{G}_t)_{t\in \mathbb{R}}$ on the line bundle $V_i^*\otimes V_j$ is contracting; see \S\ref{bundledescription}. Hence, for every $t>0$,
\begin{eqnarray*}
\left \Vert G_t X_j(u)  \right \Vert_{g_t(u)} &=& \left \Vert \left [\bar{G}_t \big (  (X_i(u))^*\otimes X_j(u)\big ) \right ] \big (G_t X_i(u) \big ) \right \Vert_{g_t(u)} \\
&\leq& \leftn \bar{G}_t \big (  (X_i(u))^*\otimes X_j(u)\big )  \rightn_{g_t(u)} \left \Vert G_t  X_i(u) \right \Vert_{g_t(u)} \\
&\leq& A e^{-at} \leftn (X_i(u))^*\otimes X_j(u)  \rightn_u  \left \Vert G_t  X_i(u) \right \Vert_{g_t(u)}
\end{eqnarray*}
for some  $A\geq 0$ and $a>0$, which proves the assertion. Note that the very first line makes use of the ``flatness'' property for the lines $V_i\to T^1S$. %Note that in the above calculation, we use the fact that the connection on $T^1S\times_\rho \mathrm{End}(\mathbb{R}^n)\to T^1S$ is flat in a crucial way.
\end{proof}

Consider the associate flat $M$--bundle $T^1S\times_{\rho'}M \to T^1S$ of the Anosov representation $\rho'=\Lambda^{\varepsilon}\rho$. Identify the oriented geodesic lamination $\La$ with its corres\-ponding subset in $T^1S$; note that $\widehat{\lambda}\subset T^1S$ is a compact subset that is the union of some leaves of the geodesic foliation $\mathcal{F}$ of $T^1S$, and that is invariant under the action of the geodesic flow $(g_t)_{t\in\mathbb{R}}$. Let $\mathcal{F}_{\rho}' \colon \partial_{\infty}\widetilde{\lambda}\to \mathrm{Flag}(\mathbb{R}^n)$ be the flag map of Lemma~\ref{lem31}. Making use of $\mathcal{F}_{\rho}'$, we define a flat section $\sigma_{\rho}'=(W_1, \ldots, W_n)$ over the geodesic lamination $\La$ as follows. Let $\widetilde{\La} \subset T^1\St$ that lifts $\La\subset T^1S$. For every $i=1$, $\ldots$ , $n$, for every $\ut\in \widetilde{\La}$, set
$$
\widetilde{W}_{i}(\ut)={\mathcal{F}'_\rho}^{(i)}(x_g^+)\cap {\mathcal{F}'_\rho}^{(n-i+1)}(x_g^-)\subset \mathbb{R}^n
$$
where $x_g^+$ and $x_g^-\in \Sinf$ are respectively the positive and the negative endpoints of the oriented geodesic $g\subset \widetilde{\lambda}$ directed by the unit vector $\ut$. The $\rho'$--equivariance of the flag map $\mathcal{F}_{\rho}':  \partial_{\infty}\widetilde{\lambda}\to \mathrm{Flag}(\mathbb{R}^n)$ implies that the flat section $\widetilde{\sigma}'_\rho=(\widetilde{W}_1, \ldots, \widetilde{W}_n)$, that is defined over $\widetilde{\La}$, is $\rho'$--equivariant. In particular, it descends to a flat section $\sigma_{\rho}'=(W_1, \ldots, W_n)$ of the bundle $T^1S\times_{\rho'}M \to T^1S$, that is defined over the geodesic lamination $\La\subset T^1\St$. Note that the flatness property implies that $\sigma_{\rho}'$ is continuous along the leaves of $\La$.

%continuous, 

Now, let $\sigma_{\rho'}=(V'_1, \ldots , V'_n)$ be the Anosov section of the Anosov representation $\rho'=\Lambda^{\varepsilon}\rho$, and let ${\sigma_{\rho'}}_{| \La}=(V'_1, \ldots , V'_n)_{| \La}$ be its restriction to $\La\subset T^1S$. We will show that the two flat sections ${\sigma_{\rho'}}_{|\La}=(V'_1, \ldots , V'_n)_{| \La}$ and $\sigma_{\rho}'=(W_1, \ldots, W_n)$ defined over $\La$ coincide.

%Pick a Riemannian metric $\left \Vert \ \right \Vert_{u}$ on $T^1S\times_{\rho'}\bar{\mathbb{R}}^n$.

Let $p \colon T^1S\times_{\rho'}\bar{\mathbb{R}}^n \to T^1S$ be the flat $\bar{\mathbb{R}}^n$--bundle of \S\ref{Rbundle description}.   Let $u\in \La$. Pick a nonzero vector $Y_{i}(u)$ in the fibre $W_{i}(u)\subset p^{-1}(u)$. Let $Y_{i}(u)=\sum_{j=1}^{n}X'_j(u)$ be its decomposition with respect to the line decomposition $V'_1(u)\oplus \cdots \oplus V'_n(u)$ of the fibre $p^{-1}(u)$. We will prove that $W_i(u)=V'_i(u)$.

%Let $\widetilde{Y}_{i}(\ut)=\sum_{j=1}^{n}\widetilde{X}'_j(\ut)$ be its decomposition with respect to the line decomposition $\widetilde{V}'_1(u)\oplus \cdots \oplus \widetilde{V}'_n(u)= \mathbb{R}^n$. We will prove that $\widetilde{W}_i(\ut)=\widetilde{V}'_i(\ut)$.
%The connection on $T^1S\times_{\rho'}\bar{\R}^n$ being flat,

 Let $(t_k)_k \to +\infty$ such that the sequence $(g_{t_k}u)_k$ converges to $u_\infty\in \La$ (such a $(t_k)_k$ exists since $\La$ is compact). Consider the vector $\frac{G_{t_k} Y_{i}(u)}{\left \Vert G_{t_k} X'_{i_0}(u) \right \Vert_{g_{t_k}(u)}} \in p^{-1}(g_{t_k}(u))$, where $i_0$ is the largest integer $j$ such that the component $X'_{j}(u)\neq 0$. The section ${\sigma_{\rho'}}_{|\La}=(V'_1, \ldots, V'_n)_{| \La}$ being flat, for every $k$,
\begin{eqnarray}
\label{Decomp}
\frac{G_{t_k} Y_{i}(u)}{\left \Vert G_{t_k} X'_{i_0}(u) \right \Vert_{g_{t_k}(u)}}=\frac{G_{t_k} X'_{i_0}(u)}{\left \Vert G_{t_k} X'_{i_0}(u) \right \Vert_{g_{t_k}(u)}} +  \sum_{j=1}^{i_0-1} \frac{G_{t_k} X'_j(u)}{\left \Vert G_{t_k} X'_{i_0}(u) \right \Vert_{g_{t_k}(u)}}. 
\end{eqnarray}
with, for every $j=1$, $\ldots$ , $i_0$,
\begin{eqnarray}
\label{SeqVect}
\frac{G_{t_k} X'_j(u)}{\left \Vert G_{t_k} X'_{i_0}(u) \right \Vert_{g_{t_k}(u)}}\in V'_j(g_{t_k}(u)).
\end{eqnarray}
By Lemma~\ref{FlowingVector}, and by continuity of ${\sigma_{\rho'}}_{|\La}=(V'_1, \ldots, V'_n)_{| \La}$ along the leaves of $\La$, for every $j=1$, $\ldots$ , $i_0$, the sequence (\ref{SeqVect}) converges to a vector in the fibre $V'_{j}(u_\infty)$; this vector is the zero vector for all $j\leq i_0-1$; and it is a unit vector for $j=i_0$; let $Z_0\in V'_{i_0}(u_\infty)$ be this vector. It follows from (\ref{Decomp}) that
$$
\lim_{k\to + \infty} \frac{G_{t_k} Y_{i}(u)}{\left \Vert G_{t_k} X'_{i_0}(u) \right \Vert_{g_{t_k}(u)}}=\lim_{k\to + \infty}\frac{G_{t_k} X'_{i_0}(u)}{\left \Vert G_{t_k} X'_{i_0}(u) \right \Vert_{g_{t_k}(u)}}=Z_{i_0}\in V'_{i_0}(u_\infty).
$$
On the other hand, the section $\sigma_{\rho}'=(W_1, \ldots, W_n)$ being flat and continuous along the leaves of $\La$, for every $k$, 
$$
\frac{G_{t_k} Y_{i}(u)}{\left \Vert G_{t_k} X'_{i_0}(u) \right \Vert_{g_{t_k}(u)}}\in W_{i}(g_{t_k}(u)),
$$
hence $Z_0=\lim_{k\to + \infty} \frac{G_{t_k} Y_{i}(u)}{\left \Vert G_{t_k} X'_{i_0}(\ut) \right \Vert_{g_{t_k}(u)}}\in W_i(u_\infty)$. Since $\left \Vert Z_0 \right \Vert_{u_\infty}=1$, $Z_0\neq 0$ in particular. Therefore, $W_i(u_\infty)=V'_{i_0}(u_\infty)$, which implies by flatness that the fibres $W_i(u)$ and $V'_{i_0}(u)$ coincide too. 

The lines $W_1(u)$, $\ldots$ , $W_n(u)$ being linearly independent, it follows from the above discussion that, for every $u\in \La$, %for every $t>0$, 
$$
\sigma'_{\rho}(u)=\big (W_1(u), \ldots, W_n(u)\big )=\big (V'_{i_1(u)}(u), \ldots, V'_{i_n(u)}(u)\big )
$$ 
for some permutation $(i_1(u), \ldots , i_n(u))$ of the set $\{1, \ldots , n\}$ that depends on the point $u\in \La$. Naturally, $\sigma'_{\rho}$ and $\sigma_{\rho'}$ being flat, the $n$-tuplet $u\mapsto(i_1(u), \ldots , i_n(u))$ is a constant map along the leaves of $\La$. 

%By continuity of the line bundles $W_i\to \La$ and $V'_i\to \La$, w

Now, observe that both the sections ${\sigma_{\rho'}}_{|\La}=(V'_1, \ldots, V'_n)_{| \La}$ and $\sigma_{\rho}'=(W_1, \ldots, W_n)$ are transversally continuous: ${\sigma_{\rho'}}_{|\La}$ is transversally continuous as it is the restriction of the Anosov section $\sigma_{\rho'}$; and $\sigma_{\rho}'$ is transversally continuous as a consequence of the estimate in the proof of Lemma~\ref{ShearingMaps}. Besides, since $\lambda$ is maximal, the geodesic lamination $\La$ is connected.      It follows from the above facts that the function $u\mapsto (i_1(u), \ldots , i_n(u))$ is constant on $\La$. Hence, for every $u\in \La$, 
$$
\sigma'_{\rho}(u)=\big (W_1(u), \ldots, W_n(u)\big )=\big (V'_{i_1}(u), \ldots, V'_{i_n}(u)\big )
$$
for some permutation $(i_1, \ldots , i_n)$ of the set $\{1, \ldots , n\}$. 

Finally, consider $\widetilde{\sigma}'_{\rho}=\big (\widetilde{V}'_{i_1}, \ldots, \widetilde{V}'_{i_n}\big )$ that lifts $\sigma'_{\rho}$. Let $\ut_0\in \widetilde{\La}$ be a point along a leaf that projects to a geodesic leaf bounding the ideal triangle $P_0\subset \St-\widetilde{\lambda}$. By construction of $\sigma'_{\rho}$, for every $j=1$, $\ldots$ , $n$, 
$$
\widetilde{V}'_{i_j}(\ut_0)= \varphi_{P_0P_0} \big (\widetilde{V}_j(\ut_0)\big )
$$ 
where $\varphi_{P_0P_0}\in \SL_n(\mathbb{R})$ is the shearing map associated with the triangle $P_0$, and where $\widetilde{\sigma}_\rho=(\widetilde{V}_1, \ldots , \widetilde{V}_n)$ lifts the Anosov section $\sigma_\rho$ of the initial Anosov representation $\rho$. Since $\varphi_{P_0P_0}=\mathrm{Id}$, $\widetilde{V}'_{i_j}(\ut_0)=\widetilde{V}_j(\ut_0)$ for every $j$, which implies that $i_j=j$. We conclude that, for every $u\in \La$, $\sigma'_{\rho}(u)=\sigma_{\rho'}(u)$; equivalently, the flag maps $\mathcal{F}'_{\rho}$ and ${\mathcal{F}_{\rho'}}_{| \partial_{\infty}\widetilde{\lambda}}$ coincide on $\partial_{\infty}\widetilde{\lambda}\subset \Sinf$. This achieves the proof of Theorem~\ref{CataFlag}.
\end{proof}

\begin{figure}
\vskip 10pt
\SetLabels
( .17*.68 ) $P$\\
( .21*.64 ) $k$\\
( .26*.34 ) $P_0$\\
( .52*.92 ) $\varphi_{P_0P}$\\
( .5*.21) $\varphi_{P_0P_0=\mathrm{Id}}$\\
( .2*1.03 ) $\mathcal{F}_\rho(z)$\\
( .48*.69 ) $\mathcal{F}_\rho(x)$\\
( -.01*.77 ) $\mathcal{F}_\rho(y)$\\
( .91*.97 ) $\mathcal{F}_{\rho'}(z)$\\
( .69*1.01 ) $\mathcal{F}_{\rho'}(y)$\\
( 1.01*.79 ) $\mathcal{F}_{\rho'}(x)$\\
\endSetLabels
\centerline{\AffixLabels{\includegraphics{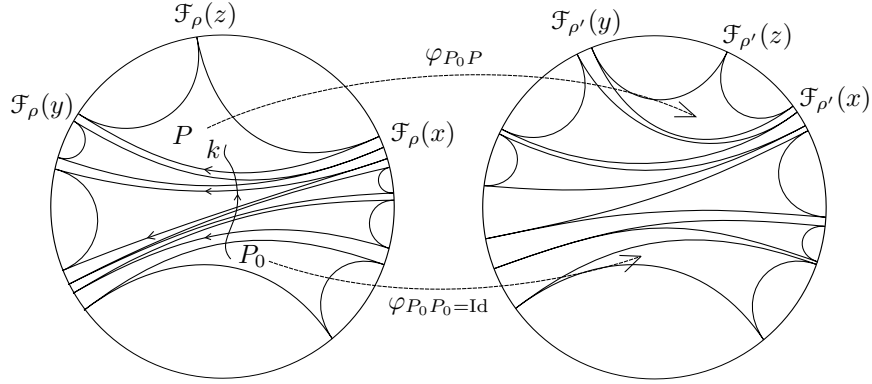}}}
\caption{Shearing maps between $\rho$ and its cataclysm deformation $\rho'=\Lambda^{\varepsilon}\rho$.}
\label{Fig3}
\end{figure}
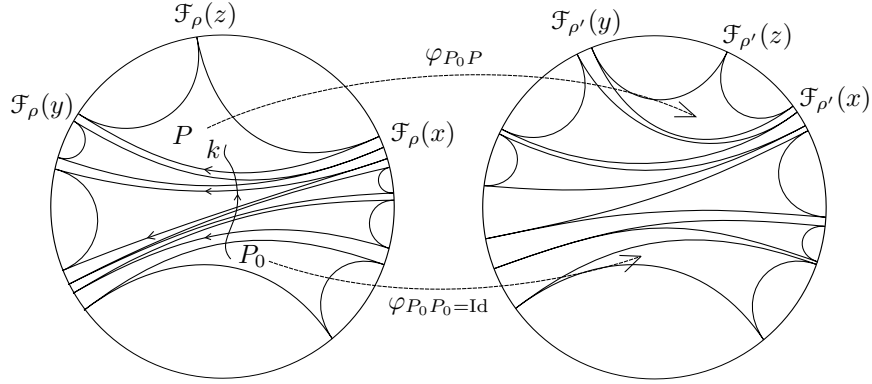

\begin{rem}
\label{RemFlagCurve}
Theorem~\ref{CataFlag} gives a simple, geometric description of a $\varepsilon$--cataclysm deformation $\rho'=\Lambda^{\varepsilon}\rho$: the $\rho$--equivariant flag curve $\mathcal{F}_{\rho}\colon\Sinf\rightarrow\mathrm{Flag}(\mathbb{R}^n)$ is mapped onto the $\rho'$--equivariant flag curve $\mathcal{F}_{\rho'}\colon\Sinf\rightarrow\mathrm{Flag}(\mathbb{R}^n)$ via the equivariant family of shearing maps $\Lambda^\varepsilon=\{\varphi_{P_0P}\}_{P\subset\St-\widetilde{\lambda}} \subset \SL_n(\mathbb{R})$. More precisely, if $x$, $y$, $z\in \Sinf$ are the vertices of some ideal triangle $P\subset \St-\widetilde{\lambda}$, the shearing map $\varphi_{P_0P}$ sends the flag triplet  $\mathcal{F}_{\rho}(P)=\big (\mathcal{F}_{\rho}(x), \mathcal{F}_{\rho}(y), \mathcal{F}_{\rho}(z)\big)$ to the flag triplet $\mathcal{F}_{\rho'}(P)=\big(\mathcal{F}_{\rho'}(x), \mathcal{F}_{\rho'}(y), \mathcal{F}_{\rho'}(z)\big)$; see Figure~\ref{Fig3}. In particular, a cataclysm should be understood as a deformation of the Anosov representation $\rho$ via a deformation of its associated flag curve $\mathcal{F}_{\rho}$. 

Equivalently, in terms of Anosov sections, the cataclysm map $\Lambda^\varepsilon$ sends the Anosov section $\sigma_\rho=(V_1, \ldots , V_n)$ to the Anosov section $\sigma_{\rho'}=(V'_1, \ldots , V'_n)$; see \S\ref{FlagDescription}.
\end{rem}

\section{Geometric properties of cataclysms}
\label{GeoProperties}

We now establish some geometric properties of cataclysms. In particular, the main result of this section is the variation formula of Theorem~\ref{IntCycle} for the length functions $\ell^\rho_i$ \cite{Dr1} of an Anosov representation $\rho$.

\subsection{The shear as a summation}
\label{DiffCocy}%{ShearTriangle}

Given a cataclysm deformation $\rho'=\Lambda^{\varepsilon}\rho$, we give a description of the shear $\varepsilon\in \CO$ as a certain summation.

Let $k$ be a transverse, simple, nonbacktracking, oriented arc to $\widetilde{\lambda}$. Orient posi\-tively the leaves of $\widetilde{\lambda}$ intersecting $k$ for the transverse orientation determined by the oriented arc $k$. As in \S\ref{sect:Inequalities}, for every component $d \subset k-\widetilde{\lambda}$, $g_d^+$ and $g_d^-\subset \widetilde{\lambda}$ are the two leaves passing by the positive and the negative endpoints of the oriented subarc $d \subset k-\widetilde{\lambda}$. Let $\ut^{+}_{d}$ and $\ut^{-}_{d}\in T^1\St$ be respectively the unit tangent vectors based at the positive and the negative endpoints of each oriented subarc $d \subset k-\widetilde{\lambda}$, that direct the oriented leaves $g_d^+$ and $g_d^-\subset \widetilde{\lambda}$. Finally, fix a triangle $P_0\subset \St-\widetilde{\lambda}$. Let $R_d\subset \widetilde{S}-\widetilde{\lambda}$ be the ideal triangle containing the subarc $d$; and let  $\varphi_d=\varphi_{P_0R_d}\in \SL_n(\mathbb{R})$ be the associated shearing map.

Let $\sigma_{\rho}=(V_1, \ldots, V_n)$ and $\sigma_{\rho'}=(V'_1, \ldots, V'_n)$ be respectively the Anosov sections of $\rho$ and $\rho'=\Lambda^{\varepsilon} \rho$, that lift to $\widetilde{\sigma}_{\rho}=(\widetilde{V}_1, \ldots, \widetilde{V}_n)$ and to $\widetilde{\sigma}_{\rho'}=(\widetilde{V}'_1, \ldots, \widetilde{V}'_n)$. By Theorem~\ref{CataFlag} and Remark~\ref{RemFlagCurve}, for every subarc $d \subset k-\widetilde{\lambda}$,
\begin{eqnarray}
\label{LinebundleRelation}
\varphi_{d}\big (\widetilde{V}_{i}(\ut^{\pm}_d) \big )=\widetilde{V}'_i(\ut^{\pm}_d).
\end{eqnarray}

Let $T^1S\times_{\rho}\bar{\mathbb{R}}^n\to T^1S$ and $T^1S\times_{\rho'}\bar{\mathbb{R}}^n\to T^1S$ be respectively the flat bundles of the Anosov representations $\rho$ and $\rho'=\Lambda^{\varepsilon}\rho$ (see \S\ref{Rbundle description}), endowed with the metrics $\left \Vert \ \right \Vert_u$ and $\left \Vert \ \right \Vert'_u$, respectively. In particular, by restricting, for every $i=1$, $\ldots$ , $n$, this provides a metric on each of the line sub-bundles $V_i\to T^1S$ and $V'_i\to T^1S$. Identify the oriented geodesic lamination $\La$ with its corresponding subset in $T^1S$. Pick  a unit section $X_i\colon \La\to V_i$ (i.e. $\left \Vert  X_i(u)  \right \Vert_u=1$ for every $u \in \La$),  that lifts to $\widetilde{X}_i(\ut)\in \widetilde{V}_i(\ut)\subset \mathbb{R}^n$, $\ut\in \widetilde{\widehat{\lambda}}$ (such a section $\widetilde{X}_i\colon \widetilde{\La}\to \widetilde{V}_i$ is not necessarily continuous). By (\ref{LinebundleRelation}), for every subarc $d\subset k-\widetilde{\lambda}$, 
$$
\varphi_{d}\widetilde{X}_{i}(\ut^{\pm}_d)\in \widetilde{V}'_i(\ut^{\pm}_d)\subset \mathbb{R}^n.
$$ 

For every $i=1$, $\ldots$ , $n$, let $\delta^{\rho\rho'}_i(k)$ be the sum defined as 
\begin{eqnarray*}
\delta^{\rho\rho'}_i(k)=\sum_{ \substack{d \subset k-\widetilde{\lambda} \\ d \neq d^{\pm}}} \log \frac{\Big\| \varphi_{d} \widetilde{X}_{i}(\ut^{-}_{d}) \Big\|'_{\ut^{-}_{d}}}{\Big\| \varphi_{d} \widetilde{X}_{i}(\ut^{+}_{d}) \Big\|'_{\ut^{+}_{d}}}&-& \log\left \Vert  \varphi_{d^-} \widetilde{X}_{i}(\ut_{d^-}^{+}) \right \Vert'_{\ut_{d^-}^{+}}\\ 
&+& \log\Big\|  \varphi_{d^+} \widetilde{X}_{i}(\ut_{d^+}^{-}) \Big\|'_{\ut_{d^+}^{-}}
\end{eqnarray*}
where the indexing $d$ ranges over all the components in $k-\widetilde{\lambda}$, and where $d^{+}$ and $d^{-}$ are the two components containing respectively the positive and the negative endpoints of the oriented arc $k$. Note that the value of the sum $\delta^{\rho\rho'}_i(k)$ is clearly independent of the choice of the lift $\widetilde{X}_i(\ut)\in \widetilde{V}_i(\ut)$, $\ut \in \widetilde{\La}$.

\begin{lem}
\label{lem28}
For $\varepsilon\in \CO$ small enough, for every transverse, simple, nonbacktracking, oriented arc $k$ to $\widetilde{\lambda}$, for every $i=1$, $\ldots$ , $n$, the series $\delta^{\rho\rho'}_i(k)$ is absolutely convergent. 
\end{lem}

%$\widetilde{V_i}_{| k \cap \widetilde{\La}}$ ${\widetilde{V_i}}'_{|k\cap \widetilde{\La}}$
%on the line bundle $\widetilde{V}'_i\to  k \cap \widetilde{\La}$ is equivalent to the restriction of the metric $\left \Vert \ \right \Vert$ to $\widetilde{V}'_i \to k\cap \widetilde{\La}$. 

\begin{proof}
Fix an arc $k$ as above. Pick a metric $\left \Vert \ \right \Vert$ on $\mathbb{R}^n$. Since $k\cap \widetilde{\lambda}$ is compact, the lifted metric $\left \Vert \ \right \Vert'_{|k\cap \widetilde{\La}}$ on the line bundle  ${\widetilde{V_i}}'_{|k\cap \widetilde{\La}}$ is equivalent to the restriction of the metric $\left \Vert \ \right \Vert$ to  ${\widetilde{V_i}}'_{|k\cap \widetilde{\La}}\subset \mathbb{R}^n$. In particular, to prove the absolute convergence of the series $\delta_i^{\rho\rho'}$, it is sufficient to show the convergence of the series
\begin{eqnarray}
\label{Series}
\sum_{d \subset k-\widetilde{\lambda}}\ \left \vert  \log \frac{\left \Vert  \varphi_{d} \widetilde{X}_{i}(\ut^{-}_{d}) \right \Vert}{\left \Vert \varphi_{d} \widetilde{X}_{i}(\ut^{+}_{d}) \right \Vert} \right \vert.
\end{eqnarray}
To do so, we begin with finding an estimate for each term $\left \vert  \log \frac{\left \Vert  \varphi_{d} \widetilde{X}_{i}(\ut^{-}_{d}) \right \Vert}{\left \Vert \varphi_{d} \widetilde{X}_{i}(\ut^{+}_{d}) \right \Vert} \right \vert$ of this series.

By Theorem~\ref{AnosovSection}, the fibre $V_i(u)$ depends H\"older continuously on the point $u\in\La$. Since the choice of the lift $\widetilde{X}_i(\ut)$ in (\ref{Series}) is irrelevant, we may assume the lift $\widetilde{X}_i(\ut)\in\mathbb{R}^n$, $\ut\in \widetilde{\La}$ to be locally H\"older continuous for the norm $\left \Vert \ \right \Vert$ of $\mathbb{R}^n$. By Lemma~\ref{lem02}, for every subarc $d\subset k-\widetilde{\lambda}$ whose divergence radius $r(d)$ (see \S\ref{sect:Inequalities}) is large enough,
$$
\left \Vert \widetilde{X}_{i}(\ut^{-}_{d})-\widetilde{X}_{i}(\ut^{+}_{d}) \right \Vert=\mathrm{O} \big (e^{-Kr(d)} \big )
$$
for some $K>0$ (depending on $k$ and $\rho$). By Corollary~\ref{cor08}, $\varphi_{d}=\psi_{P_0R_d}\circ T^{\varepsilon(P_0,R_{d})}_{g^{-}_{R_{d}}}$, and thus
\begin{eqnarray*}
\left \Vert \varphi_{d}\widetilde{X}_{i}(\ut^{-}_{d})-\varphi_{d}\widetilde{X}_{i}(\ut^{+}_{d}) \right \Vert &\leq&  \left\Vert \psi_{d}\right\Vert     \Big\|     T^{\varepsilon(P_0,R_{d})}_{g^{-}_{R_{d}}}    \Big\|    \left \Vert \widetilde{X}_{i}(\ut^{-}_{d})-\widetilde{X}_{i}(\ut^{+}_{d}) \right \Vert.
\end{eqnarray*} 
The estimates in Corollary~\ref{cor08} and in Lemma~\ref{lem05} then show that, for every subarc $d\subset k-\widetilde{\lambda}$ whose divergence radius $r(d)$ is large enough,
\begin{eqnarray}
\label{Ineq1}
\left \Vert \varphi_{d}\widetilde{X}_{i}(\ut^{-}_{d})-\varphi_{d}\widetilde{X}_{i}(\ut^{+}_{d}) \right \Vert =\mathrm{O} \big  ( e^{C \left \Vert \varepsilon \right \Vert (r(d)+1)}e^{-Kr(d)} \big )
\end{eqnarray} 
for some $C\geq 0$ (depending on $k$ and $\rho$).

We now determine a lower and upper bound for  the term $\log\frac{\left \Vert  \varphi_{d} \widetilde{X}_{i}(\ut^{-}_{d}) \right \Vert}{\left \Vert \varphi_{d} \widetilde{X}_{i}(\ut^{+}_{d}) \right \Vert}$. For every subarc $d\subset k-\widetilde{\lambda}$,
\begin{eqnarray}
\label{Ineq2}
\frac{  \Big\| \varphi_{d} \widetilde{X}_{i}(\ut^{+}_{d}) \Big\|}{\left \Vert \varphi_{d} \widetilde{X}_{i}(\ut^{-}_{d}) \right \Vert} &\leq& 1 + \frac{1}{\left \Vert \varphi_{d} \widetilde{X}_{i}(\ut^{-}_{d}) \right \Vert} \left \Vert \varphi_{d}\widetilde{X}_{i}(\ut^{-}_{d})-\varphi_{d}\widetilde{X}_{i}(\ut^{+}_{d}) \right \Vert.
\end{eqnarray}
Again, let us write $\varphi_{d}=\psi_{P_0R_d}\circ T^{\varepsilon(P_0,R_{d})}_{g^{-}_{R_{d}}}$. The estimate in Corollary~\ref{cor08} shows that the family $\{\psi_{P_0R_d}\}_{d\subset k-\widetilde{\lambda}}\subset \SL_n(\mathbb{R})$ is bounded, and in addition, that it remains bounded away from $0\in \mathrm{Mat}_n(\mathbb{R})$, whenever $\varepsilon \in \CO$ is small enough. Therefore, for every $X\in \mathbb{R}^n$,
\begin{eqnarray}
\label{Ineq3}
\left \Vert \varphi_{d}(X) \right \Vert \geq m \, \Big\| T^{\varepsilon(P_0,R_{d})}_{g^{-}_{R_{d}}}(X) \Big\|
\end{eqnarray}
for some $m>0$ (depending on $k$  and $\rho$). By combining estimates (\ref{Ineq1}), (\ref{Ineq2}) and (\ref{Ineq3}), 
\begin{eqnarray*}
\frac{\left \Vert  \varphi_{d} \widetilde{X}_{i}(\ut^{+}_{d}) \right \Vert}{\left \Vert \varphi_{d} \widetilde{X}_{i}(\ut^{-}_{d}) \right \Vert} &\leq& 1 + \frac{1}{m}\frac{1}{ \Big\|  T^{\varepsilon(P_0,R_{d})}_{g^{-}_{R_{d}}}  \big (\widetilde{X}_{i}(\ut^{-}_{d})\big )\Big\|}\left \Vert \varphi_{d}\widetilde{X}_{i}(\ut^{-}_{d})-\varphi_{d}\widetilde{X}_{i}(\ut^{+}_{d}) \right \Vert\\ \\
&\leq& 1 + \frac{1}{m}e^{-\varepsilon_{i}(P_0, R_{d})} \, \mathrm{O}\big (e^{C \left \Vert \varepsilon \right\Vert(r(d)+1)}e^{-Kr(d)}\big )\\ \\
&\leq& 1 + \mathrm{O} \big (e^{2 C \left \Vert \varepsilon \right \Vert(r(d)+1)-Kr(d)}\big )
\end{eqnarray*}
Hence, for every subarc $d\subset k-\widetilde{\lambda}$ whose divergence radius $r(d)$ is large enough,
\begin{eqnarray}
\label{Ineq4}
- \mathrm{O} \big (e^{2 C \left \Vert \varepsilon \right \Vert(r(d)+1)-Kr(d)}\big )\leq \log \frac{\left \Vert  \varphi_{d} \widetilde{X}_{i}(\ut^{-}_{d}) \right \Vert}{\left \Vert \varphi_{d}  \widetilde{X}_{i}(\ut^{+}_{d}) \right \Vert}.
\end{eqnarray}

Likewise, a similar calculation yields
\begin{eqnarray*}
\frac{\left \Vert  \varphi_{d}  \widetilde{X}_{i}(\ut^{-}_{d}) \right \Vert}{\left \Vert \varphi_{d}  \widetilde{X}_{i}(\ut^{+}_{d}) \right \Vert} &\leq& 1 + \frac{1}{m}\frac{1}{\Big\| T^{\varepsilon(P_0,R_{d})}_{g^{-}_{R_{d}}}  \big ( \widetilde{X}_{i}(\ut^{+}_{d}) \big )\Big\|}\left \Vert \varphi_{d} \widetilde{X}_{i}(\ut^{-}_{d})-\varphi_{d} \widetilde{X}_{i}(\ut^{+}_{d}) \right \Vert.
\end{eqnarray*}
Note that
\begin{align*}
\Big\| T^{\varepsilon(P_0,R_{d})}_{g^{-}_{R_{d}}}  \big ( \widetilde{X}_{i}(\ut^{-}_{d}) \big)-T^{\varepsilon(P_0,R_{d})}_{g^{-}_{R_{d}}}  \big ( \widetilde{X}_{i}(\ut^{+}_{d})\big  )\Big\| &=\mathrm{O} \Big (\Big\| T^{\varepsilon(P_0,R_{d})}_{g^{-}_{R_{d}}} \Big\| \, \left \Vert \widetilde{X}_{i}(\ut^{-}_{d})-\widetilde{X}_{i}(\ut^{+}_{d}) \right \Vert\Big)\\
&=\mathrm{O} \Big (e^{C \left\Vert \varepsilon \right \Vert (r(d)+1)}e^{-Kr(d)}\Big).%\left \Vert \widetilde{X}_{i}(\ut^{-}_{d})-\widetilde{X}_{i}(\ut^{+}_{d}) \right \Vert
\end{align*}
Hence 
\begin{eqnarray*}
\Big\|  T^{\varepsilon(P_0,R_{d})}_{g^{-}_{R_{d}}}  \big ( \widetilde{X}_{i}(\ut^{+}_{d}) \big )\Big\| &\geq&\Big\| T^{\varepsilon(P_0,R_{d})}_{g^{-}_{R_{d}}}  \big ( \widetilde{X}_{i}(\ut^{-}_{d}) \big) \Big\|  - \mathrm{O} \Big (e^{C \left\Vert \varepsilon \right \Vert (r(d)+1)}e^{-Kr(d)}\Big)\\
&\geq& e^{-\varepsilon_{i}(P_0,R_{d})} - \mathrm{O} \Big (e^{C \left\Vert \varepsilon \right \Vert (r(d)+1)}e^{-Kr(d)}\Big)\\
&\geq& e^{-C \left\Vert \varepsilon\right\Vert(r(d)+1)} - \mathrm{O} \Big (e^{C \left\Vert \varepsilon \right \Vert (r(d)+1)}e^{-Kr(d)}\Big)\\ 
&\geq& e^{-C \left\Vert \varepsilon\right\Vert(r(d)+1)}\Big (1 - \mathrm{O}\Big(e^{2C \left\Vert \varepsilon \right \Vert (r(d)+1)-Kr(d)}\Big)\Big ).
\end{eqnarray*}
Observe that, whenever $\left\Vert \varepsilon\right\Vert<K/2C$, for every subarc $d\subset k-\widetilde{\lambda}$ whose divergence radius $r(d)$ is large enough, the right-hand side is positive. Therefore,
\begin{eqnarray*}
\frac{\left \Vert  \varphi_{d}  \widetilde{X}_{i}(\ut^{-}_{d}) \right \Vert}{\left \Vert \varphi_{d} 
 \widetilde{X}_{i}(\ut^{+}_{d}) \right \Vert}&\leq& 1 + \frac{1}{m}\frac{\mathrm{O} \big (e^{C\left \Vert \varepsilon \right \Vert(r(d)+1)-Kr(d)}\big )}{e^{-C \left\Vert \varepsilon\right\Vert(r(d)+1)}\Big (1 - \mathrm{O}\Big(e^{2C \left\Vert \varepsilon \right \Vert (r(d)+1)-Kr(d)}\Big)\Big )}\\
&\leq& 1 + \frac{1}{m}\frac{\mathrm{O} \big (e^{2C\left \Vert \varepsilon \right \Vert(r(d)+1)-Kr(d)}\big )}{\Big (1 - \mathrm{O}\Big(e^{2C \left\Vert \varepsilon \right \Vert (r(d)+1)-Kr(d)}\Big)\Big )}\\ \\
&\leq& 1 +\mathrm{O} \big (e^{2C\left \Vert \varepsilon \right \Vert(r(d)+1)-Kr(d)}\big ).
\end{eqnarray*}
Hence, for every subarc $d\subset k-\widetilde{\lambda}$ whose divergence radius $r(d)$ is large enough,
\begin{eqnarray}
\label{Ineq5}
\log \frac{\left \Vert  \varphi_{d}  \widetilde{X}_{i}(\ut^{-}_{d}) \right \Vert}{\left \Vert \varphi_{d} 
 \widetilde{X}_{i}(\ut^{+}_{d}) \right \Vert} \leq \mathrm{O} \big (e^{2C\left \Vert \varepsilon \right \Vert(r(d)+1)-Kr(d)}\big ).
\end{eqnarray}
The convergence of the series~(\ref{Series}) then follows from estimates (\ref{Ineq4}) and (\ref{Ineq5}), and from an application of Lemma~\ref{lem01}, whenever $\varepsilon\in\CO$ is small enough. 

Finally, note the following additivity property. Let $k_{1}$ and $k_{2}$ be two subarcs of $k$ with disjoint interior such that $k=k_{1} \cup k_{2}$, and assume that both series $\delta^{\rho\rho'}_i(k_{1})$ and $\delta^{\rho\rho'}_i(k_{2})$ are absolutely convergent. Then $\delta^{\rho\rho'}_i(k)=\delta^{\rho\rho'}_i(k_{1})+\delta^{\rho\rho'}_i(k_{2})$, which implies that $\delta^{\rho\rho'}_i(k)$ is also absolutely convergent. The same argument as in the proof of Lemma~\ref{lem11} then shows that we can find $\varepsilon\in\CO$ small enough so that, for every transverse, simple, nonbacktracking, oriented arc $k$ to $\widetilde{\lambda}$, the series $\delta^{\rho\rho'}_i(k)$ is absolutely convergent.
\end{proof}

\begin{rem}
\label{rem29}
A consequence of the absolute convergence in Lemma~\ref{lem06} is that, for $\varepsilon\in \CO$ small enough, the series 
$$
\delta^{\rho\rho'}_i(k)=\sum_{ \substack{ d \subset k-\widetilde{\lambda} \\ d \neq d^{\pm}}} \log \frac{\left\Vert \varphi_{d} \widetilde{X}_{i}(\ut^{-}_{d}) \right\Vert'_{\ut^{-}_{d}}}{\left \Vert \varphi_{d} \widetilde{X}_{i}(\ut^{+}_{d}) \right \Vert'_{\ut^{+}_{d}}} - \log\left \Vert  \varphi_{d^-} \widetilde{X}_{i}(\ut_{d^-}^{+}) \right \Vert'_{\ut_{d^-}^{+}} + \log\Big\|  \varphi_{d^+} \widetilde{X}_{i}(\ut_{d^+}^{-}) \Big\|'_{\ut_{d^+}^{-}}
$$  
is \emph{commutatively} convergent. 
\end{rem}

\begin{pro}
\label{pro34}
For $\varepsilon\in \CO$ small enough, for every transverse, simple, nonbacktracking arc $k$ to $\widetilde{\La}$, the $n$--tuplet $\delta^{\rho\rho'}(k)=\big (\delta^{\rho\rho'}_{1}(k), \ldots , \delta^{\rho\rho'}_{n}(k)\big )$ is equal to the $n$--uplet $\varepsilon(k)=\big (\varepsilon_{1}(k), \ldots , \varepsilon_{n}(k)\big )\in \mathbb{R}^n$.
\end{pro}

In the above statement, the transverse $n$--twisted cocycle $\varepsilon\in \CO$ is regarded as a $\pi_1(S)$--invariant transverse $n$--twisted cocycle for the lift $\widetilde{\La}$.

\begin{proof}
Fix an arc $k$ as above. By Lemma~\ref{lem28} and Remark~\ref{rem29}, for  every $i=1$, $\ldots$ , $n$, whenever $\varepsilon\in \CO$ is small enough,
\begin{eqnarray*}
\delta^{\rho\rho'}_{i}(k)=\lim_{r\to \infty}\sum_{\substack{d \subset k-\widetilde{\lambda} \\ d \neq d^\pm \\ r(d)\leq r}} \log \frac{\left \Vert  \varphi_{d} \widetilde{X}_{i}(\ut^{-}_{d}) \right \Vert'_{u^{-}_{d}}}{\left \Vert \varphi_{d} \widetilde{X}_{i}(\ut^{+}_{d}) \right \Vert'_{\ut^{+}_{d}}} &-& \log\left \Vert  \varphi_{d^-} \widetilde{X}_{i}(\ut_{d^-}^{+}) \right \Vert'_{\ut_{d^-}^{+}}\\ &+& \log\left \Vert  \varphi_{d^+} \widetilde{X}_{i}(\ut_{d^+}^{-}) \right \Vert'_{\ut_{d^+}^{-}}
\end{eqnarray*}
where $r(d)$ is the divergence radius of the subarc $d\subset k-\widetilde{\lambda}$ (see \S\ref{sect:Inequalities}). We wish to show that $\delta^{\rho\rho'}_{i}(k)=\varepsilon_i(k)$.

With the same notation as in \S\ref{sect:Cataclysms}, let $P$ and $Q\subset \St-\widetilde{\lambda}$ be the two ideal triangles whose interiors are joined by the oriented transverse arc $k$ to $\widetilde{\lambda}$. Put $m_r=\mathrm{Card}(\mathcal{P}^r_{PQ})$. Index the elements of $\mathcal{P}_{PQ}^{r}$ as $R^{r}_{1}$, $R^{r}_{2}$, $\ldots$ , $R^{r}_{m}$ so that the indexing $j$ of $R^{r}_{j}$ increases as one goes from $P$ to $Q$, and for convenience, set $R^r_{0}=P$ and $R^r_{m+1}=Q$. Finally, let us set $d_j=k\cap R^r_j$. Then
\begin{eqnarray*}
\delta^{\rho\rho'}_{i}(k)=\lim_{r \to \infty} \,\sum_{j=1}^{m_r} \log \frac{\Big\|  \varphi_{d_{j}} \widetilde{X}_{i}(\ut^{-}_{j}) \Big\|'_{\ut^{-}_{j}}}{\Big\| \varphi_{d_{j}} \widetilde{X}_{i}(\ut^{+}_{j}) \Big\|'_{\ut^{+}_{j}}} &-& \log\left \Vert  \varphi_{d_{0}} \widetilde{X}_{i}(\ut_{0}^{+}) \right \Vert'_{\ut_{0}^{+}} \\ &+& \log\left \Vert  \varphi_{d_{m+1}} \widetilde{X}_{i}(\ut_{m+1}^{-}) \right \Vert'_{\ut_{m+1}^{-}}.
\end{eqnarray*}
We should emphasize the fact that in the above series, the endpoints $\ut_j^\pm \in k\cap \widetilde{\La}$ of each oriented subarc $d_j\subset k-\widetilde{\lambda}$ all depend on the integer $r$. By reordering the terms,
$$
\delta^{\rho\rho'}_{i}(k)=\lim_{r \to \infty}\sum_{j=0}^{m_r} \log  \frac{\left \Vert  \varphi_{d_{j+1}} \widetilde{X}_{i}(u^{-}_{j+1}) \right \Vert'_{u_{j+1}^{-}}}{\left \Vert  \varphi_{d_{j}} \widetilde{X}_{i}(u^{+}_{j}) \right \Vert'_{u_{j}^{+}}}. 
$$
%As in the proof of Lemma~\ref{lem28}, We simply mean to alleviate (a very little) the already heavy notations. 

Pick a metric $\left \Vert \ \right \Vert$ on $\mathbb{R}^n$. $k\cap \widetilde{\lambda}$ being compact, the lifted metric $\left \Vert \ \right \Vert'_{|k\cap \widetilde{\La}}$ on the line bundle ${\widetilde{V_i}}'_{|k\cap \widetilde{\La}}$ is equivalent to the restriction of $\left \Vert \ \right \Vert$ to ${\widetilde{V_i}'}_{|k\cap \widetilde{\La}} \subset \mathbb{R}^n$. Note that, by definition of $\mathcal{P}_{PQ}^{r}$, we have $r(d) > r$ for every subarc $d\subset k-\widetilde{\lambda}\setminus \bigcup d_j$. Thus, by Lemma~\ref{lem02}, for every $j=1$, $\ldots$ , $m_r$,
$$
\mathrm{dist}_{T^1S}\big (\ut^-_{j+1},\ut^+_{j}\big )= \mathrm{O} \big(e^{-Ar} \big )
$$ 
for some $A>0$ (depending on $k$ and $\rho$). The lifted metric $\left \Vert \ \right \Vert'_{\ut}$ on $\mathbb{R}^n$ depending smoothly on $\ut\in T^1\St$, it follows from the above estimate that 
$$
\delta^{\rho\rho'}_{i}(k)=\lim_{r \to \infty}\sum_{j=0}^{m_r} \log  \frac{\left \Vert  \varphi_{d_{j+1}} \widetilde{X}_{i}(\ut^{-}_{j+1}) \right \Vert'_{\ut_{j+1}^{-}}}{\left \Vert  \varphi_{d_{j}} \widetilde{X}_{i}(\ut^{+}_{j}) \right \Vert'_{\ut_{j}^{+}}}=\lim_{r \to \infty}\sum_{j=0}^{m_r} \,\log  \frac{\left \Vert  \varphi_{d_{j+1}} \widetilde{X}_{i}(\ut^{-}_{j+1}) \right \Vert}{\left \Vert  \varphi_{d_{j}} \widetilde{X}_{i}(\ut^{+}_{j}) \right \Vert}. 
$$

We now focus attention on the series on the right-hand side and calculate its value. To do so, we begin with finding an estimate for each term of this series. By applying Corollary~\ref{cor08},
\begin{eqnarray*}
\varphi_{d_{j+1}}\widetilde{X}_{i}(\ut^-_{j+1})&=&\varphi_{d_{j}}\psi_{R^r_{j}R^r_{j+1}} T_{g^-_{R^r_{j+1}}}^{\varepsilon(R^r_{j},R^r_{j+1})}\big (\widetilde{X}_{i}(\ut^-_{j+1})\big )\\ \\
&=&e^{\varepsilon_{i}(R^r_{j},R^r_{j+1})}\varphi_{d_{j}}\psi_{R^r_{j}R^r_{j+1}}\big (\widetilde{X}_{i}(\ut^-_{j+1})\big )\\ \\
&=&e^{\varepsilon_{i}(R^r_{j},R^r_{j+1})}\varphi_{d_{j}} \Big( \widetilde{X}_{i}(\ut^-_{j+1}) + \phi_{j} \big (\widetilde{X}_{i}(\ut^-_{j+1} ) \big ) \Big )
\end{eqnarray*}
where $\phi_{j}\in \SL_n(\mathbb{R})$ is a linear map such that $\leftn \phi_{j}\rightn=\mathrm{O}\Big ( \sum_{R \in \mathcal{P}_{R^r_{j}R^r_{j+1}}} e^{-Br(k\cap R)}  \Big )$ for some $B$ (depending on $k$ and $\rho$). Thus 
\begin{eqnarray*}
\varphi_{d_{j+1}}\widetilde{X}_{i}(\ut^-_{j+1})&=&e^{\varepsilon_{i}(R^r_{j},R^r_{j+1})}\varphi_{d_{j}} \left(  \widetilde{X}_{i}(\ut^+_{j})+\big (\widetilde{X}_{i}(\ut^-_{j+1})-\widetilde{X}_{i}(\ut^+_{j}) \big ) + \phi_{j}\big (\widetilde{X}_{i}(\ut^-_{j+1} \big )     \right ) \\ \\
&=&e^{\varepsilon_{i}(R^r_{j},R^r_{j+1})}\varphi_{d_{j}}\widetilde{X}_{i}(\ut^+_{j})   +    e^{\varepsilon_{i}(R^r_{j},R^r_{j+1})}\varphi_{d_{j}}\big (\widetilde{X}_{i}(\ut^-_{j+1})-\widetilde{X}_{i}(\ut^+_{j}) \big )\\ \\ &&\text{$+$ }  e^{\varepsilon_{i}(R^r_{j},R^r_{j+1})}\varphi_{d_{j}} \phi_{j}\big (\widetilde{X}_{i}(\ut^-_{j+1}) \big ).
\end{eqnarray*}

Therefore,
\begin{align*}
\frac{\left \Vert \varphi_{d_{j+1}}\widetilde{X}_{i}(\ut^-_{j+1})\right \Vert}{\left \Vert \varphi_{d_{j}} \widetilde{X}_{i}(\ut^+_{j})  \right \Vert} = e^{\varepsilon_i(R^r_{j}, R^r_{j+1})}  \Bigg [    1 &+  \mathrm{O} \Bigg ( \frac{\left \Vert \varphi_{d_{j}} \big( \widetilde{X}_{i}(\ut^-_{j+1})-\widetilde{X}_{i}(\ut^+_{j}) \big ) \right \Vert}{\left \Vert \varphi_{d_{j}} \widetilde{X}_{i}(\ut^+_{j})  \right \Vert} \Bigg ) \\
& \qquad \qquad \qquad+ \mathrm{O} \Bigg (\frac{\left \Vert \varphi_{d_{j}} \phi_{j}\big (\widetilde{X}_{i}(\ut^-_{j+1}) \big )\right \Vert}{\Big\| \varphi_{d_{j}} \widetilde{X}_{i}(\ut^+_{j}) \Big\|} \Bigg ) \Bigg ]
\end{align*}
which implies that
\begin{align*}
\log\frac{\left \Vert \varphi_{d_{j+1}}\widetilde{X}_{i}(\ut^-_{j+1})\right \Vert}{\left \Vert \varphi_{d_{j}} \widetilde{X}_{i}(\ut^+_{j})  \right \Vert}=\varepsilon_i(R^r_{j}, R^r_{j+1}) &+   \mathrm{O} \Bigg (\frac{\left \Vert \varphi_{d_{j}} \big( \widetilde{X}_{i}(\ut^-_{j+1})-\widetilde{X}_{i}(\ut^+_{j}) \big ) \right \Vert}{\left \Vert \varphi_{d_{j}} \widetilde{X}_{i}(\ut^+_{j})  \right \Vert} \Bigg )\\ 
&\qquad \qquad \qquad +   \mathrm{O} \Bigg (\frac{\left \Vert \varphi_{d_{j}} \phi_{j}\big (\widetilde{X}_{i}(\ut^-_{j+1}) \big )\right \Vert}{\Big\| \varphi_{d_{j}} \widetilde{X}_{i}(\ut^+_{j}) \Big\|} \Bigg ).
\end{align*}
Similar arguments as in the proof of Lemma~\ref{lem28} show that, provided that $\left\Vert \varepsilon\right\Vert<K/2C$, for $r$ large enough, for every $j=1$, $\dots$ , $m_r$,
$$
\frac{\left \Vert \varphi_{d_{j}} \big( \widetilde{X}_{i}(\ut^-_{j+1})-\widetilde{X}_{i}(\ut^+_{j}) \big ) \right \Vert}{\left \Vert \varphi_{d_{j}} \widetilde{X}_{i}(\ut^+_{j})  \right \Vert}=\mathrm{O}(e^{2C\left \Vert \varepsilon \right \Vert(r+1) -Kr})
$$ 
for some $K>0$ and $C$ (both depending on $k$ and $\rho$). Likewise, 
$$
\frac{\left \Vert \varphi_{d_{j}} \phi_{j}\big (\widetilde{X}_{i}(\ut^-_{j+1}) \big )\right \Vert}{\Big\| \varphi_{d_{j}} \widetilde{X}_{i}(\ut^+_{j}) \Big\|} =\mathrm{O}(e^{2C\left \Vert \varepsilon \right \Vert(r+1) -Kr}).
$$

As a result,
\begin{align*}
\delta_i^{\rho\rho'}(k)= \lim_{r\to \infty}\,\sum_{j=0}^{m_r}  \log \frac{ \Big\| \varphi_{d_{j+1}}\widetilde{X}_{i}(\ut^-_{j+1}) \Big\|}{ \Big\| \varphi_{d_{j}} \widetilde{X}_{i}(\ut^+_{j})  \Big\|}  &= \lim_{r\to \infty}\,\sum_{j=0}^{m_r}  \varepsilon_i(R^r_j, R^r_{j+1}) \\
&\qquad \qquad+ \lim_{r\to \infty}\, \sum_{j=0}^{m_r} \mathrm{O}(e^{2C\left \Vert \varepsilon \right \Vert(r+1) -Ar}).
\end{align*}
Finally, observe that $\varepsilon_{i}(k)=\sum_{j=0}^{m_r}  \varepsilon_i(R^r_j, R^r_{j+1})$, and that, by Lemma~\ref{lem01}, $m_r=\mathrm{Card}(\mathcal{P}^r_{PQ})=\mathrm{O}(r)$. We conclude that
$$
\delta_i^{\rho\rho'}(k)=\varepsilon_{i}(k).
$$
\end{proof}

\begin{cor}
\label{Cor36}
Let $\rho$ be an Anosov representation, and let $\mathcal{U}^\rho$ be some open neighborhood of $0\in\CO$ small enough. The cataclysm map 
\begin{eqnarray*}
\Lambda: \mathcal{U}^\rho& \to &\Rep^{\text{Anosov}}_{\PSL_{n}(\mathbb{R})}(S)\\
\varepsilon&\mapsto&\Lambda^{\varepsilon}\rho
\end{eqnarray*}
is injective.
\end{cor}
\begin{proof}
Let $\varepsilon$ and $\varepsilon'\in \CO$ be such that $\Lambda^{\varepsilon}\rho=\Lambda^{\varepsilon'}\rho$. Then $\varepsilon=\delta^{\rho\rho'}=\varepsilon'$ by Proposition~\ref{pro34}.
\end{proof}

\subsection {Length functions of an Anosov representation}
\label{LengthFunctions}

In \cite{Dr1}, we extend \emph{Thurs\-ton's length function of Fuchsian representations} \cite{Th1, Th2, Bon4, Bon1} to an important class of Anosov representations known as \emph{Hitchin representations} \cite{La1, Gui, FoGo}. More precisely, to every Hitchin representation $\rho\colon \pi_1(S)\to \PSL_n(\mathbb{R})$, we associate $n$ length functions $\ell_i^\rho \colon \CH(S)\to \mathbb{R}$ defined on the space of \emph{H\"older geodesic currents} $\CH(S)$ \cite{Bon2}. The construction of the lengths $\ell^\rho_i$ extends to every Anosov representation and we begin with reviewing some of this construction.

%Given an Anosov representation $\rho \colon \pi_1(S)\to \PSL_n(\mathbb{R})$, consider the flat bundle $T^1S\times_{\rho} \bar{\mathbb{R}}^n\to T^1S$ of Remark~\ref{Eigenbundle}. 

Consider the flat, $\bar{\mathbb{R}}^n$--bundle $T^1S\times_{\rho} \bar{\mathbb{R}}^n\to T^1S$ of an Anosov representation $\rho$ as in \S\ref{Rbundle description}. Let $(G_t)_{t\in \mathbb{R}}$ be the flow on $T^1S\times_{\rho} \bar{\mathbb{R}}^n$ that lifts the geodesic flow $(g_t)_{t\in \mathbb{R}}$ on $T^1S$. The Anosov section $\sigma_{\rho}=(V_1, \ldots , V_n)$ provides a line decomposition $V_1\oplus \cdots\oplus V_n$ of the bundle $T^1S\times_{\rho} \bar{\mathbb{R}}^n\to T^1S$ with the property that each line sub-bundle $V_i\to T^1S$ is invariant under the action of the flow $(G_t)_{t\in \mathbb{R}}$. Finally, pick a Riemannian metric $\left \Vert \ \right \Vert_{u}$ on $T^1S\times_{\rho} \bar{\mathbb{R}}^n\to T^1S$. 

Let $\mathcal{F}$ be the geodesic foliation of the unit tangent bundle $T^1S$. Let $u_0\in T^1S$, and let $X_i(u_0)\in V_i(u_0)$ be a vector. For every $u \in T^1S$ lying on the same geodesic leaf as $u_0$, set
$$
\omega^\rho_i(u)=-{\frac{d}{dt}\log \Big\| G_{t}X_i(u_0) \Big\|_{g_t(u_0)} dt}_{|t=t_u}
$$
where $t_u \in \mathbb{R}$ is such that $u=g_{t_u}(u_0)$. The above expression defines a $1$--form $\omega^\rho_i$ on $T^1S$ along the leaves of the geodesic foliation $\mathcal{F}$. One easily verifies that the definition of $\omega^\rho_i$ is independent of the choices of $u_0\in T^1S$ and $X_i(u_0)\in V_i(u_0)$, and thus only depends on the metric $\left \Vert \ \right \Vert_u$. In addition, because of the regularity of the line bundles $V_i\to T^1S$ (see Theorem~\ref{AnosovSection}), the $1$--forms $\omega^\rho_i$ satisfy the following properties: they are smooth, closed along the leaves of the geodesic foliation $\mathcal{F}$; and are transversally H\"older continuous.  We refer the reader to \cite{Dr1} for details.

Fix a maximal geodesic lamination $\lambda\subset S$. Let $\La$ be its orientation cover as in \S\ref{sect:TransCocycles}. In the context of this article, we will be interested in length functions $\ell_i^\rho$ defined on the vector space $\HO$ of transverse cocycles for $\La$ only. In particular, we now give an alternative definition of the lengths $\ell_i^\rho$ in the special case of $\HO$, that differs from the one in \cite{Dr1}, but that better suits our purposes here. %, needs and the context of this article.

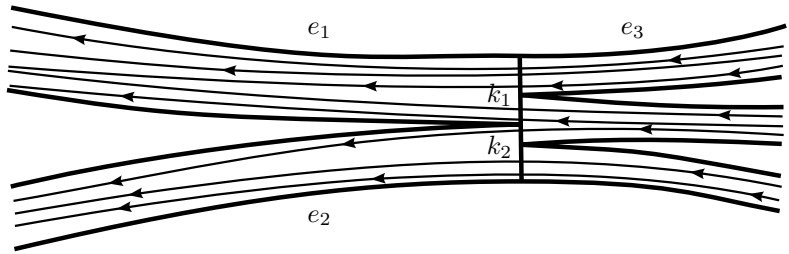
\begin{figure}[htbp]
\SetLabels
( .4*.88 ) $e_1$\\
( .4*.13 ) $e_2$\\
( .8*.88 ) $e_3$\\
( .63*.61 ) $k_1$\\
( .63*.4 ) $k_2$\\
\endSetLabels
\vskip 10pt
\centerline{\AffixLabels{\includegraphics{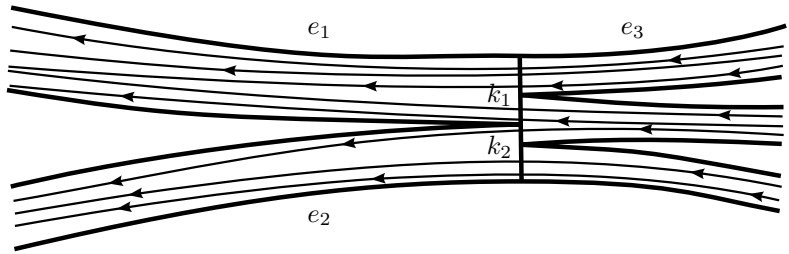}}}
\caption{An oriented geodesic lamination $\La$ within a train track $\widehat{U}$.}
\label{Fig5*}
\end{figure}

Identify the oriented geodesic lamination $\widehat{\lambda}$ with its corresponding subset in $T^1S$; note that $\La\subset T^1S$ is a closed subset that is the union of some leaves of the geodesic foliation $\mathcal{F}$. In particular, the geodesic lamination $\La$ inherits $n$ $1$--forms $\omega^{\rho}_i$ that are smooth, closed along its leaves, and transversally H\"older continuous. Let $\widehat{U}\supset \La$ be an open surface as in \S\ref{sect:TransCocycles}. We may assume without loss of generality that $\widehat{U}$ is a train track \cite{PeH, Bon4} for the oriented lamination $\widehat{\lambda}$. Let  $e_1$, $\ldots$ , $e_m$ be the oriented edges of the train track $\widehat{U}$; and let $k_1$, $\ldots$ , $k_m \subset \widehat{U}$ be the ingoing lids of each of the corresponding edge $e_j$; see Figure~\ref{Fig5*}. For every edge $e_j$, complete the partial foliation induced by $\La \cap e_j$ in a full foliation of $e_j$. By integrating the $1$--form $\omega_i^\rho$ along each oriented plaque in the edge $e_j$, and considering the negative endpoint of each oriented plaque, we define a function $h_j:k_j\to \mathbb{R}$ on the transverse arc $k_j$, that is H\"older continuous due to the regularity of $\omega_i^\rho$. Let $\alpha\in \HO$ be a transverse cocycle. By Theorem~\ref{gap}, $\alpha$ assigns on each transverse arc $k_j$ a H\"older distribution $\alpha_{k_j}$. The length $\ell_i^\rho(\alpha)$ of the transverse cocycle $\alpha\in \HO$ is defined as
$$
\ell_i^\rho(\alpha)=\sum_{j=1}^m \alpha_{k_j}(h_j)=\int_{\La} \omega_i^\rho d\alpha
$$
where $\alpha_{k_j}(h_j)$ is the value of the distribution $\alpha_{k_j}$ at the function $h_j$. One easily verifies that the value $\ell_i^\rho(\alpha)$ is independent of the choice of the train track $\widehat{U}\supset \La$. In addition, a homological argument shows that the value $\ell_i^\rho(\alpha)$ is independent of the metric  $\left \Vert \ \right \Vert_u$ that we chose on the bundle $T^1S\times_{\rho} \bar{\mathbb{R}}^n\to T^1S$, and of which we made use to define the $1$--form $\omega_i^\rho$. Finally, note that the length 
$$
\ell_i^\rho\colon \HO\to \mathbb{R}
$$ 
is a linear function on the vector space of transverse cocycles $\HO$.

\subsection{$1$--forms $\Delta_i^{\rho\rho'}$}
\label{DifferenceForms}

Given an Anosov representation $\rho$, let $\rho'=\Lambda^{\varepsilon}\rho$ be a $\varepsilon$--cataclysm deformation along a maximal geodesic lamination $\lambda\subset S$ for some transverse $n$--twisted cocycle $\varepsilon \in \CO$ small enough. For every $i=1$, $\ldots$ , $n$, set
$$
\Delta_i^{\rho\rho'}=\omega^{\rho'}_i- \omega^{\rho}_i
$$
where $ \omega^{\rho}_i$ and $ \omega^{\rho'}_i$ are the $1$--forms as in \S\ref{LengthFunctions}. Therefore, $\Delta_i^{\rho\rho'}$ defines a $1$--form that is smooth, closed along the leaves of the oriented geodesic lamination $\La$, and is transversally H\"older continuous. We wish to relate the $1$--form $\Delta^{\rho\rho'}_i$ to the shear $\varepsilon\in \CO$. In particular, the main result of this section is Proposition~\ref{cor46}.

Let $T^1S\times_{\rho}\bar{\mathbb{R}}^n\to T^1S$ and $T^1S\times_{\rho'}\bar{\mathbb{R}}^n\to T^1S$ be respectively the flat bundles of the Anosov representations $\rho$ and $\rho'=\Lambda^{\varepsilon}\rho$, endowed with the metrics $\left \Vert \ \right \Vert_u$ and $\left \Vert \ \right \Vert'_u$, respectively. Let $V_i\to T^1S$ and $V_i'\to T^1S$ be the associated line bundles of $\rho$ and $\rho'$. 

%Let $\widetilde{X}_i(\ut)$ be the lift of some unit section $X_i(u)$ (for the metric $\left \Vert \ \right\Vert_u$) of the line bundle $V'_i\to T^1S$. 

Consider an open surface $\widehat{U}\supset \La$ as in \S\ref{sect:TransCocycles}. Recall that the complement $\widehat{U}-\La$ is made of one-holed ideal hexagons $\widehat{P}$, where each $\widehat{P}$ is the lift of some (punctured) ideal triangle $P\subset S-\lambda$. Finally, identify the oriented geodesic lamination $\La$ with its corresponding subset in $T^1S$.

Let $\widehat{P} \subset \widehat{U}-\widehat{\lambda}$ be a one-holed hexagon, and let $h_1$, $\dots$ , $h_6\subset \La$ be the six oriented edges of $\widehat{P}$. Along the boundary $\partial \widehat{P}=h_1\cup\cdots \cup h_6$ of the one-holed hexagon $\widehat{P}$, we consider the function $F_{i,\partial \widehat{P}}\colon \partial\widehat{P} \to \mathbb{R}$ defined as follows.  Let $\widetilde{\widehat{P}} \subset \widetilde{\widehat{U}}-\widetilde{\widehat{\lambda}}$ that lifts $\widehat{P}\subset \widehat{U}-\La$, with $\partial \widetilde{\widehat{P}}=\widetilde{h}_1\cup\cdots \cup \widetilde{h}_6$; for every $j=1$, $\ldots$ , $6$, for every $u\in h_j$, set 
$$
{F_{i, \partial \widehat{P}}}_{| h_j}(u)=-\log\left \Vert \varphi_{P_0\widetilde{P}} \widetilde{X}_i(\ut) \right \Vert'_{\ut} 
$$
where: $\ut\in \widetilde{h}_j$ is a lift of $u$; $\widetilde{X}_i$ is the lift of some unit section $X_i$ (for the metric $\left \Vert \ \right\Vert_u$) of the line bundle $V_i\to T^1S$; and $\varphi_{P_0\widetilde{P}}$ is the shearing map associated with the ideal triangle $\widetilde{P}\subset \St- \widetilde{\lambda}$ ($\widetilde{P}$ is the triangle such that the one-holed hexagon $\widetilde{\widehat{P}}\subset \widetilde{\widehat{U}}-\widetilde{\La}$ projects to $\widetilde{P}$; and $P_0\subset \St-\widetilde{\lambda}$ is a triangle that we fix). A key step in proving Proposition~\ref{cor46} is the following observation.

\begin{lem}
\label{Lem38}
For every $u\in \La$ that lies along the boundary $\partial \widehat{P}$ of the one-holed hexagon $\widehat{P}$,
$$
\Delta^{\rho\rho'}_i(u)=d_{u} {F_{i, \partial \widehat{P}}}
$$
where the differential is taken along $\partial \widehat{P}$. 
\end{lem}

In the above statement, $\left \Vert \ \right \Vert'_u$ is the metric chosen on the flat bundle $T^1S\times_{\rho'}\bar{\mathbb{R}}^n\to T^1S$ to define the $1$--form $\omega^{\rho'}_i$; see  \S\ref{LengthFunctions}.

\begin{proof}
It follows from the equivariance property of the shearing map $\varphi_{P_0\widetilde{P}}$ that the function ${F_{i, \partial \widehat{P}}}$ is well defined. We must check that it is smooth. 

Let $u\in \partial \widehat{P}$ that lies along the oriented edge $h_j \subset \La$, and let $\widetilde{\widehat{P}}$, $\widetilde{h}_j$, $\ut\in \widetilde{h}_j$, and $\widetilde{P}$ as above. By Theorem~\ref{CataFlag} and Remark~\ref{RemFlagCurve}, for every $t\in \mathbb{R}$, $\varphi_{P_0\widetilde{P}} \widetilde{X}_i\big (g_t(\ut) \big)\in \widetilde{V}'_i(g_t(\ut))$. Moreover, $X_i(u)\in V_i(u)$ being a unit section (for the metric $\left \Vert \ \right\Vert_u$), and the fibre $V_i(u)$ depending smoothly on $u\in T^1S$ along the leaves of the geodesic foliation $\mathcal{F}$, one easily verifies that the function
$$
t\mapsto\log \left \Vert \varphi_{P_0\widetilde{P}} \widetilde{X}_i\big (g_t(\ut) \big)\right \Vert'_{g_t(\ut)}
$$
is differentiable, which implies that ${F_{i, \partial \widehat{P}}}$ is smooth. Therefore, for every $u\in\widehat{P}$,
$$
d_{u} {F_{i, \partial \widehat{P}}}=d_{\ut}\log\left \Vert \varphi_{P_0\widetilde{P}} \widetilde{X}_i(\ut)  \right \Vert'_{\widetilde{u}}={\frac{d}{dt}\Big (\log\left \Vert \varphi_{P_0\widetilde{P}} \widetilde{X}_i(g_t(\ut))  \right \Vert'_{g_t(\ut)} \Big )dt}_{|t=0}
$$
defines a smooth, closed $1$--form along $\partial \widehat{P}\subset \widetilde{\La}$. 

By definition of the $1$--forms $\omega^\rho_i$ and  $\omega^{\rho'}_i$ (see \S\ref{LengthFunctions}), for every $u\in \widehat{\lambda}$,
\begin{eqnarray*}
\Delta_i^{\rho\rho'}(u)=\big (\omega_i^{\rho'}-\omega_i^{\rho}\big )(u)={\frac{d}{dt} \Big (\log \left \Vert \widetilde{G}_{t}\widetilde{X}_i(\widetilde{u}) \right \Vert_{g_t(\widetilde{u})}-\log \left \Vert \widetilde{G}'_{t}\widetilde{X}'_i(\widetilde{u}) \right \Vert'_{g_t(\widetilde{u})} \Big )dt}_{|t=0}
\end{eqnarray*}
where $(G_{t})_{t\in \mathbb{R}}$ and $(G'_{t})_{t\in \mathbb{R}}$ are the flows on the flat bundles $T^1S\times_{\rho} \bar{\mathbb{R}}^n\to T^1S$ and $T^1S\times_{\rho'} \bar{\mathbb{R}}^n\to T^1S$, respectively; see \S\ref{Rbundle description}.

%Remark~\ref{Eigenbundle}.

Let $\widetilde{X}'_i(\ut)$ be the lift of some unit section $X'_i(u)$ (for the metric $\left \Vert \ \right\Vert'_u$) of the line bundle $V'_i\to T^1S$. Since, for every $\ut\in \widetilde{\La}$, $\varphi_{P_0\widetilde{P}}\big ( \widetilde{V}_i(\widetilde{u})\big )= \widetilde{V}'_i(\widetilde{u})$, we have $\varphi_{P_0\widetilde{P}}\widetilde{X}_i(\widetilde{u})=\mu_{\ut} \widetilde{X}'_i(\widetilde{u})$ for some $\mu_{\ut} \in \mathbb{R}$. In addition, because of the flat connections on $T^1S\times_{\rho} \bar{\mathbb{R}}^n$ and $T^1S\times_{\rho'} \bar{\mathbb{R}}^n$, and since the shearing map $\varphi_{P_0\widetilde{P}}$ is a linear map, for every $t\in \mathbb{R}$,
$$
\widetilde{G}'_t \widetilde{X}'_i(\widetilde{u})=\mu_{\ut} \varphi_{P_0\widetilde{P}}\big(\widetilde{G}_t \widetilde{X}_i(\widetilde{u})\big ).
$$
As a result, for every $t\in \mathbb{R}$, for every $\ut \in \partial \widetilde{\widehat{P}}$,
\begin{eqnarray*}
\log \left \Vert \widetilde{G}_{t}\widetilde{X}_i(\widetilde{u}) \right \Vert_{g_t(\widetilde{u})}  \! \! \! -  \log \left \Vert \widetilde{G}'_{t}\widetilde{X}'_i(\ut) \right \Vert'_{g_t(\widetilde{u})} \! \! &=& \! \!  - \log \frac{\left \Vert \widetilde{G}'_{t}\widetilde{X}'_i(\widetilde{u}) \right \Vert'_{g_t(\widetilde{u})}}{\left \Vert \widetilde{G}_{t}\widetilde{X}_i(\widetilde{u}) \right \Vert_{g_t(\widetilde{u})}}\\
&=&\! \! - \log \frac{{\left \vert \mu_{\ut} \right \vert \left \Vert\varphi_{P_0\widetilde{P}}\big ( \widetilde{G}_t \widetilde{X}_i(\widetilde{u})\big )\right \Vert'}_{g_t(\widetilde{u})}}{\left \Vert \widetilde{G}_{t}\widetilde{X}_i(\widetilde{u}) \right \Vert_{g_t(\widetilde{u})}}\\
&=&\! \! - \log \left \Vert \varphi_{P_0\widetilde{P}} \Bigg (\frac{\widetilde{G}_t \widetilde{X}_i(\widetilde{u})}{\left \Vert \widetilde{G}_{t}\widetilde{X}_i(\widetilde{u}) \right \Vert_{g_t(\widetilde{u})}}  \Bigg ) \right \Vert'_{g_t(\widetilde{u})}\\ \\
& &\qquad \qquad \qquad \qquad  \qquad \qquad - \log\left \vert \mu_{\ut} \right \vert\\ \\
&=&\! \! -\log\left \Vert\varphi_{P_0\widetilde{P}}\widetilde{X}_i(g_t(\widetilde{u}))  \right \Vert'_{g_t(\widetilde{u})} - \log\left \vert \mu_{\ut} \right \vert . \\
\end{eqnarray*} 
Note that the very last step makes use of the fact that $X_i(u)$ is a unit section (for the metric $\left \Vert \ \right\Vert_u$) of the line bundle $V_i\to T^1S$. 

We thus conclude that, for every $u\in \La$ that lies along the boundary $\partial \widehat{P}$ of some one-holed hexagon $\widehat{P}$,
\begin{eqnarray*}
\Delta_i^{\rho\rho'}(u)&=&{\frac{d}{dt}\Big (-\log\left \Vert \varphi_{P_0\widetilde{P}}\widetilde{X}_i(g_t(\widetilde{u})) \right \Vert'_{g_t(\widetilde{u})} - \log\left \vert \mu_{\ut} \right \vert\Big )}_{|t=0}dt\\
&=&-d_{\widetilde{u}}\log\left \Vert \varphi_{P_0\widetilde{P}} \widetilde{X}_i(\widetilde{u}) \right \Vert'_{\widetilde{u}}.
\end{eqnarray*}
\end{proof}

\begin{figure}[htbp]
\SetLabels
( .67*.75 ) $h_1$\\
( .53*.65 ) $\Omega_1$\\
( .32*.75 ) $h_6$ \\
( .17*.5 ) $h_5$ \\
( .67*.24 ) $h_3$ \\
( .32*.24 ) $h_4$ \\
(.83 *.5 ) $h_2$ \\
( .64*.5 ) $\Omega_2$\\
\endSetLabels
\vskip 10pt
%\centerline{\AffixLabels{\includegraphics{Fig4_One_holed_hexagon.eps}}}
\centerline{\AffixLabels{\includegraphics{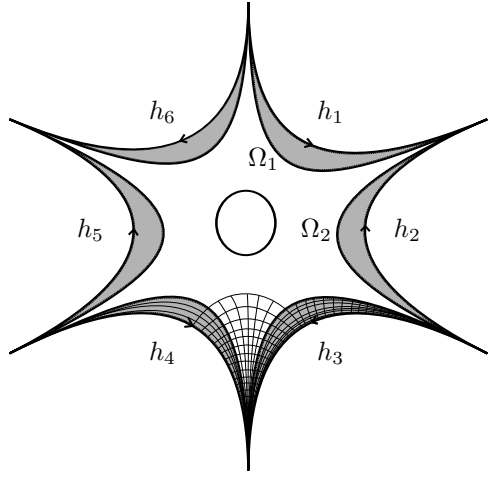}}}
\caption{A one-holed hexagon $\widehat{P}$ in $\widehat{U}-\La$ and its two transverse foliations.}
\label{Fig4}
\end{figure}
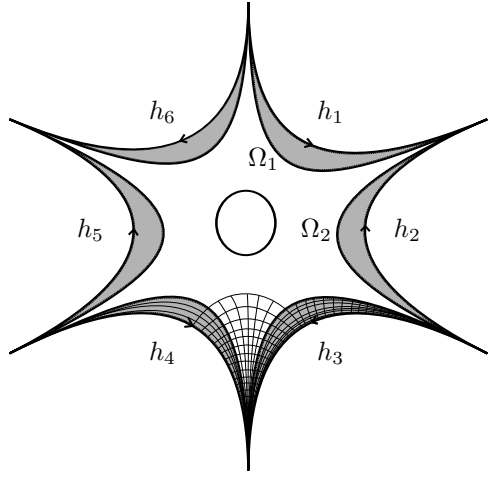

\begin{lem}
\label{FormExtended}
The $1$--form $\Delta^{\rho\rho'}_i$ (defined along the leaves of $\La$) extends a H\"older continuous, closed $1$--form defined on the open surface $\widehat{U}$. 
\end{lem}

A H\"older continuous $1$--form $\omega$ on $\widehat{U}$ is  \emph{closed} if its path integral $\int_\gamma \omega$ along any path $\gamma\subset \widehat{U}$ locally depends only on the endpoints of $\gamma$. In other words, the endpoints of $\gamma$ being frozen, a small perturbation of $\gamma$ does not change the value of $\int_\gamma \omega$.  

\begin{proof}

 Let $\widehat{P} \subset \widehat{U}-\widehat{\lambda}$ be a one-holed hexagon, and let $h_1$, $\dots$ , $h_6\subset \La$ be the six (oriented) edges of $\widehat{P}$. Consider the function $F_{i,\partial \widehat{P}}$ of Lemma~\ref{Lem38}. Let $\Omega_1$, $\ldots$ , $\Omega_6\subset \widehat{P}$ be open neighborhoods of the six edges $h_1$, $\dots$ , $h_6$, respectively, as shown on Figure~\ref{Fig4}; we choose the $\Omega_j$ so that their closures $\bar{\Omega}_i$ in $\widehat{P}$ are pairwise disjoint. For every $j=1$, $\ldots$ , $6$, let $\theta_j\colon \widehat{P}\to \mathbb{R}$ be a bump function that is identically equal to $1$ near $h_j$ and vanishes outside of $\Omega_j$. 

Foliate  the one-holed hexagon $\widehat{P}$ with vertical leaves and horizontal leaves as in Figure~\ref{Fig4} (the horizontal foliation refers to the one that is parallel to the edges  $h_j$). Observe that the two transverse foliations are naturally oriented: the orientation of the hori\-zontal foliation is determined by the oriented edges $h_j$; and the surface $\widehat{U}$ being oriented, the vertical foliation inherits the transverse orientation. Let $s$ and $t$ be local coordinates along the leaves of the vertical foliation and the leaves of the horizontal foliation, respectively. Without loss of gene\-rality, one can arrange that $t$ coincides with the time coordinate along the oriented edges $h_j$ of $\widehat{P}$.

Let $F_{i,\widehat{P}} \colon \widehat{P}   \to \mathbb{R}$ be the function defined by 
$$
F_{i,\widehat{P}}(s,t)=\sum \theta_j(s,t) {F_{i,\partial \widehat{P}}}_{| h_j}(t).
$$ 
One easily verifies that it is a smooth function on the interior of $\widehat{P}$ that extends the previous function $F_{i,\partial \widehat{P}}$ defined along the boundary $\partial\widehat{P}$. Its differential $d F_{i,\widehat{P}}$ provides a smooth, exact $1$--form on the interior of $\widehat{P}$, that extends the previous $d F_{i,\partial \widehat{P}}$ defined along $\partial\widehat{P}$. Hence, by Lemma~\ref{Lem38},
$$
{dF_{i,\widehat{P}}}_{|\partial\widehat{P}}=\Delta_i^{\rho\rho'}.
$$

As a result, $\Delta_i^{\rho\rho'}$ (that is defined along $\La$) extends to a $1$--form defined on $\widehat{U}$ that is smooth, exact on the interior of each one-holed hexagon $\widehat{P}\subset \widehat{U}-\La$. Moreover, it follows from the construction, and the H\"older regularity of the $1$--form $\Delta_i^{\rho\rho'}$ along the leaves of $\La$, that the extension $\Delta_i^{\rho\rho'}$ is H\"older continuous on $\widehat{U}$. In particular, for any path $\gamma \subset \widehat{U}$, the path integral $\int_\gamma \Delta_i^{\rho\rho'}$ is well defined. Besides, the exactness of $\Delta_i^{\rho\rho'}$ on the interior of each one-holed hexagon in $\widehat{U}-\La$ implies that the integral $\int_\gamma \Delta_i^{\rho\rho'}$ locally depends on the endpoints of $\gamma$ only, which proves that the H\"older continuous $1$--form $\Delta_i^{\rho\rho'}$ is closed. 
\end{proof}

Consider the H\"older continuous, closed $1$--form $\Delta^{\rho\rho'}_i$ of Lemma~\ref{FormExtended}, that is defined on $\widehat{U}\supset \La$.

\begin{pro}
\label{cor46}
For $\varepsilon\in \CO$ small enough, for every transverse, simple, nonbacktracking, oriented arc $k\subset \widehat{U}$ to $\widehat{\lambda}$,
$$
\varepsilon_{i}(k)=\int_k \Delta^{\rho\rho'}_i - F_{i,\widehat{P}^+}(u_k^+) + F_{i,\widehat{P}^-}(u_k^-).
$$
where: $u_k^+$ and $u_k^-$ are respectively the positive and the negative endpoints of the oriented arc $k$; and $\widehat{P}^+$ and $\widehat{P}^-$ are the one-holed hexagons containing the endpoints $u_k^+$ and $u_k^-$, respectively.
\end{pro}

\begin{proof}
Let $k\subset \widehat{U}$ be a transverse, oriented arc to $\La$ as above, that lifts to arc $\widetilde{k}\subset \widetilde{\widehat{U}}$ transverse to $\widetilde{\La}$. Then
$$
\int_k \Delta^{\rho\rho'}_i=\int_{\widetilde{k}} \Delta^{\rho\rho'}_i= \sum_{\widetilde{d}\subset \widetilde{k}-\widehat{\La}}\int_{\widetilde{d}} \Delta^{\rho\rho'}_i
$$
where the indexing $\widetilde{d}$ ranges over the set of components of $\widetilde{k}-\widetilde{\La}$.

Recall that the $1$--form $\Delta^{\rho\rho'}_i$ is exact in the interior of each one-holed hexagon in $\widehat{U}-\La$, and that, for every (oriented) subarc $\widetilde{d} \subset \widetilde{k}-\widetilde{\La}$
$$
\int_{\widetilde{d}} \Delta_i^{\rho\rho'}=F_{i,\widetilde{d}}(\ut_{d}^{+}) - F_{i,\widetilde{d}}(\ut_{d}^{-})
$$
where: $F_{i,\widetilde{d}}$ is the function of Lemma~\ref{FormExtended} defined on the interior of the one-holed hexagon in $\widetilde{\widehat{U}}-\widetilde{\La}$ that contains the subarc $\widetilde{d}\subset \widetilde{k}-\widetilde{\La}$; and $\ut_{d}^{+}$ and $\ut_{d}^{-}$ are respectively the positive and the negative endpoints of $\widetilde{d}\subset \widetilde{k}-\widetilde{\widehat{\lambda}}$. In particular, for every subarc $\widetilde{d} \subset \widetilde{k}-\widetilde{\La}$ that does not contain any of the endpoints $u_{\widetilde{k}}^+$ and $u_{\widetilde{k}}^-$ of $k$,
$$
\int_{\widetilde{d}} \Delta_i^{\rho\rho'}=-\log \left \Vert \varphi_{\widetilde{d}} \widetilde{X}_i(\ut^{+}_{d})  \right \Vert'_{\ut_{d}^{+}}   + \log \left \Vert \varphi_{\widetilde{d}} \widetilde{X}_i(\ut^{-}_{d})  \right \Vert'_{\ut_{d}^{-}}
$$
where $\varphi_{\widetilde{d}}\in \SL_n(\mathbb{R})$ is the shearing map associated with the ideal triangle of $\St-\widetilde{\lambda}$ containing the subarc $\widetilde{d}$. In addition,
$$
\int_{\widetilde{d}^+} \Delta^{\rho\rho'}_i=F_{i,\widetilde{d}^+}(\ut_{d^+}^{+})+\log \left \Vert \varphi_{\widetilde{d}^+} \widetilde{X}_i(\ut_{d^+}^{-})  \right \Vert'_{\ut_{d^+}^{-}} 
$$
and
$$
\int_{\widetilde{d}^-} \Delta^{\rho\rho'}_i=-\log \left \Vert \varphi_{\widetilde{d}^-} \widetilde{X}_i(\ut_{d^-}^{+})  \right \Vert'_{\ut_{d^-}^{+}} -F_{i,\widetilde{d}^-}(\ut_{d^-}^{-})
$$
where $\widetilde{d}^+$ and $\widetilde{d}^-$ are the (oriented) subarcs containing the positive and the negative endpoints $u_{\widetilde{k}}^+$ and $u_{\widetilde{k}}^-$. As a result,
\begin{eqnarray*}
\int_k \Delta^{\rho\rho'}_i&=&\sum_{\substack{\widetilde{d}\subset \widetilde{k}-\widetilde{\widehat{\lambda}}\\ \widetilde{d} \neq \widetilde{d}^{\pm}}} \log \frac{\left \Vert \varphi_{\widetilde{d}} \widetilde{X}_i(\ut^{-}_{d})  \right \Vert'_{\ut_{d}^{-}}}{\left \Vert \varphi_{\widetilde{d}} \widetilde{X}_i(\ut^{+}_{d})  \right \Vert'_{\ut_{d}^{+}}} + F_{i,\widetilde{d}^+}(u_{\widetilde{k}}^+)+\log \left \Vert \varphi_{\widetilde{d}^+} \widetilde{X}_i(\ut_{d^+}^{-})  \right \Vert'_{\ut_{d^+}^{-}} \\
&&\qquad\qquad\qquad\qquad\qquad\qquad-\log \left \Vert \varphi_{\widetilde{d}^-} \widetilde{X}_i(\ut_{d^-}^{+})  \right \Vert'_{\ut_{d^-}^{+}} -F_{i,\widetilde{d}^-}(u_{\widetilde{k}}^-)
\end{eqnarray*}
since $\ut_{d^+}^{+}=u_{\widetilde{k}}^+$ and $\ut_{d^-}^{-}=u_{\widetilde{k}}^-$. The result then follows from Proposition~\ref{pro34}, provided that $\varepsilon\in \CO$ is small enough.
\end{proof}

\subsection{Thurston's intersection number}
\label{ThurstonForm}

The vector space $\HO$ of transverse cocycles for $\La$ admits a natural symplectic form $\tau\colon\HO\times\HO\to \mathbb{R}$ known as  \emph{Thurston's intersection number} \cite{Th2, Bon3, Bon4}. This pairing is defined as follows.

Consider an open surface $\widehat{U}\supset \La$ as in \S\ref{sect:TransCocycles}. Let $k_1$, $\ldots$ , $k_m\subset \widehat{U}$ be a finite family of disjoint transverse arcs to the geodesic lamination $\La$ such that every leaf intersects at least one $k_j$. Thus, $\La-\bigcup k_j$ is made of oriented arcs that can be regrouped into finitely many parallel classes; two oriented arcs belong to the same parallel class if their positive (negative resp.) endpoints lie in the same arcs $k_{j'}$ ($k_{j}$ resp.). Collapse each $k_j$ to a point $u_j$, and each parallel class to an oriented edge joining $u_j$ to $u_{j'}$. We obtain an oriented graph $\mathcal{G}_\alpha$ with weights assigned on each of the edges as follows. If $k$ is a transverse arc intersecting exactly all the leaves of a given parallel class, the corresponding edge of $\mathcal{G}_\alpha$ is assigned the weight $\alpha(k)$.

\begin{figure}[htbp]
\vskip 10pt
\centerline{\AffixLabels{\includegraphics{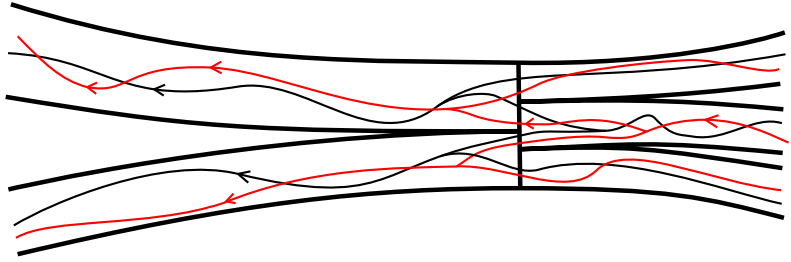}}}
\caption{Thurston's intersection number of two transverse cocycles.}
\label{Figure6}
\end{figure}

Given $\alpha$ and $\beta\in \HO$, the pairing $\tau(\alpha, \beta)$ is the self-intersection number between the two weighted oriented graphs $\mathcal{G}_\alpha$ and $\mathcal{G}_\beta$ defined as follows. Apply to the weighted graph $\mathcal{G}_\beta$  a small perturbation so that the obtained weighted graph $\mathcal{G}'_\beta$ is in transverse position to $\mathcal{G}_\alpha$ as on Figure~\ref{Figure6}; then assign to each intersection point of two edges the product of the corresponding weights, multiplied by $+1$ or $-1$ depending on whether the angle between the two oriented edges is positively or negatively oriented; then take the sum of all of these numbers. It is easy to verify that the resulting number does not depend on the choice of the graphs $\mathcal{G}_{\alpha}$ and $\mathcal{G}'_{\beta}$, and thus that the pairing $\tau(\alpha, \beta)$ is well defined. Note that the intersection number $\tau(\alpha, \beta)$ can be related to the classical self-intersection pairing in homology. Indeed,  it follows from the additivity property of the transverse cocycle $\alpha$ that the oriented weighted graph $\mathcal{G}_{\alpha}$ is a $1$--cycle in $\widehat{U}$. Hence $\alpha\in \HO$ defines a homology class $[\alpha] \in H_1(\widehat{U})$. In particular, Thurston's intersection number on $\HO$ coincides with the classical homology intersection pairing defined on $H_1(\widehat{U})$ (up to a nonzero scalar multiplication).

\subsection{Variation of the length functions}

We now describe the behavior of the length functions $\ell_i^\rho$ of \S\ref{LengthFunctions} under cataclysm deformations. 

Fix a maximal geodesic lamination $\lambda\subset S$ with orientation cover $\La$. Let $\rho$ be an Anosov representation, and let $\rho'=\Lambda^{\varepsilon}\rho$ be a cataclysm deformation for some transverse $n$--twisted cocycle $\varepsilon=(\varepsilon_1, $\ldots$ , \varepsilon_n)\in \CO$ small enough. Let $\ell^\rho_{i} \colon \HO \rightarrow \mathbb{R}$ and $\ell^{\rho'}_{i}\colon \HO \rightarrow \mathbb{R}$ be respectively the length functions associated with $\rho$ and $\rho'$; see \S\ref{LengthFunctions}. 

\begin{thm}
\label{IntCycle}
For every transverse H\"older cocycle $\alpha \in\HO$, 
$$
\ell^{\rho'}_{i}(\alpha)=\ell^{\rho}_{i}(\alpha) + \tau(\alpha, \varepsilon_i)
$$
where $\tau\colon \HO\times\HO\to \mathbb{R}$ is Thurston's intersection number.
\end{thm}

\begin{proof}
[Proof of Theorem~\ref{IntCycle}]
Let $\alpha\in \HO$. Then 
\begin{eqnarray*}
\ell_i^{\rho'}(\alpha)-\ell_{i}^{\rho}(\alpha)&=&\int_{\La} \omega_i^{\rho'}d\alpha -\int_{\La} \omega_i^\rho d\alpha\\
&=&\int_{\La} \Delta_i^{\rho\rho'}d\alpha.
\end{eqnarray*}
Let $\mathcal{G}_{\alpha}=\sum a_ pc_p$ be a weighted graph as in \S\ref{ThurstonForm}. Applying a small deformation, we can arrange that each $1$--simplex $c_j$ is simple and in transverse position to $\La$ and nonbacktracking. Finally, let $\widehat{U}$ an open surface containing $\La$ as in \S\ref{sect:TransCocycles}, and let us extend the $1$--form $\Delta_i^{\rho\rho'}$ defined along the leaves of $\La$ to $\widehat{U}$ as in \S\ref{DifferenceForms}. 

\begin{lem}
For every $\alpha\in \HO$, 
$$
\int_{\La} \Delta_i^{\rho\rho'}d\alpha=\int_{\mathcal{G}_{\alpha}}\Delta_i^{\rho\rho'}d\alpha.
$$
\end{lem}

%; see Figure~\ref{Figure6}

\begin{proof}
As in \S\ref{ThurstonForm}, pick a finite family of transverse, simple, nonbacktracking arcs $k_1$, $\ldots$ , $k_m \subset \widehat{U}$ to $\La$, so that  $\La-\bigcup_j k_j$ consists of oriented arcs of finite length. Given two transverse arcs $k_{j'}$ and $k_j$, consider the set of oriented arcs in $\La-\bigcup_j k_j$ whose all positive endpoints lie in $k_{j'}$, and all negative endpoints lie in $k_j$. Let $c_p$ be a  $1$--simplex  intersecting $k_j$ and $k_{j'}$. By subdividing the chain $\sum a_p c_p$ into a sum of smaller simplexes if necessary, we may assume without loss of generality that the positive and negative endpoints of $c_p$ lie in $k_{j'}$ and $k_j$, respectively. Similarly, by subdividing each transverse arc $k_j$ into smaller transverse subarcs, we may assume that two oriented arcs in $\La-\bigcup_j k_j$ whose negative endpoints lie in the same arc $k_j$ also have their positive endpoints lying in the same arc $k_{j'}$. Recall that the length $\ell_i^\rho(\alpha)$ of the transverse cocycle $\alpha\in \HO$ is defined as 
$$
\ell_i^\rho(\alpha)=\sum_{j=1}^m \alpha_{k_j}(h_j)
$$ 
where $\alpha_{k_j}(h_j)$ is the value of the transverse H\"older distribution $\alpha_{k_j}$; see \S\ref{sect:TransCocycles} and  \S\ref{LengthFunctions}. 

For every $j=1$, $\ldots$ , $m$, for every $u\in k_j$, consider the difference
$$
s_j(u)=h_j(u)-\int_{c_p} \Delta_i^{\rho\rho'}=\int_{\mathrm{arc}_u} \Delta_i^{\rho\rho'}-\int_{c_p} \Delta_i^{\rho\rho'}.
$$ 
where $\mathrm{arc}_u$ denotes the oriented arc in $\La-\bigcup_j k_j$ with $u\in k_j$ as negative endpoint. Assuming the $1$--simplex $c_p$ to be small enough, it is contained in a simply connected open subset of $\widehat{U}$. The $1$--form $\Delta_i^{\rho\rho'}$ being smooth, closed on this open subset, it is thus exact. Therefore,
$$
s_j(u)=\int_{\mathrm{arc}_u} \Delta_i^{\rho\rho'}-\int_{c_p} \Delta_i^{\rho\rho'}=\int_{{k_j}_{u\to c_p(0)}} \Delta_i^{\rho\rho'}-\int_{{k_{j'}}_{u'\to c_p(1)}} \Delta_i^{\rho\rho'}
$$
where: $c_p(0)$ and $c_p(1)$ are respectively the negative and the positive endpoints of the $1$--simplex $c_p$; ${k_j}_{u\to c_p(0)}$ is the oriented subarc contained in the transverse arc $k_j$ joining $u$ to $c_p(0)$; and ${k_{j'}}_{u'\to c_p(1)}$ is the oriented subarc contained in the transverse arc $k_{j'}$ joining $u'$ to $c_p(1)$. Note that the function $s_j \colon k_j\to\mathbb{R}$ is H\"older continuous. As a result,
$$
\sum_{j=1}^m \alpha_{k_j}(s_j)=\sum_{j=1}^m \alpha_{k_j} \Big (\int_{{k_{j}}_{u\to c_p(0)}} \Delta_i^{\rho\rho'} \Big) -\sum_{j=1}^m \alpha_{k_j}\Big (\int_{{k_{j'}}_{u'\to c_p(1)}} \Delta_i^{\rho\rho'} \Big)=0.
$$

We  conclude that
$$
\ell_i^{\rho'}(\alpha)-\ell_{i}^{\rho}(\alpha)=\int_{\La} \Delta_i^{\rho\rho'}d\alpha=\sum_{j=1}^m \alpha_{k_j}(h_j)=\int_{\sum a_p c_p} \Delta_i^{\rho\rho'}d\alpha$$
which proves the requested result.
\end{proof}

By applying Lemma~\ref{IntCycle},
\begin{eqnarray*}
\ell_i^{\rho'}(\alpha)-\ell_{i}^{\rho}(\alpha)&=&\int_{\La} \Delta_i^{\rho\rho'}d\alpha\\ 
&=&\int_{\sum a_p c_p} \Delta_i^{\rho\rho'}d\alpha\\ 
&=& \sum a_p \int_{c_p} \Delta_i^{\rho\rho'}d\alpha\\
&=& \sum  [\pm1]_p a_p \varepsilon_i(c_p) + \sum a_p  \big [F_i(c_p(1)) - F_i(c_p(0))\big ]
\end{eqnarray*}
where the latter step follows from an application of Proposition~\ref{cor46}; note that in the above calculation, $F_i$ denotes invariably any of the functions $F_{i,\widehat{P}}$ of the proof of Lemma~\ref{FormExtended} depending on the one-holed hexagons the endpoints $c_p(1)$ and $c_p(0)$ belong to. Since $\mathcal{G}_{\alpha}=\sum a_p c_p$ represents a cycle, it is immediate that $\sum a_p  \big [F_i(c_p(1)) - F_i(c_p(0))\big]=0$. Hence
\begin{eqnarray*}
\ell_i^{\rho'}(\alpha)-\ell_{i}^{\rho}(\alpha)&=&\sum [\pm1]_p a_p \varepsilon_i(c_p)\\
&=&\tau(\alpha, \varepsilon_i)
\end{eqnarray*}
where the coefficient $[\pm1]_p$ is equal to $+1$ or $-1$ depending on whether the transverse arc $c_p$ is positively or negatively oriented for the transverse orientation of $\La$. This achieves the proof of Theorem~\ref{IntCycle}.
\end{proof}

\subsection*{Acknowledgments} I would like to thank Anna Wienhard, Bill Goldman, Dick Canary and Olivier Guichard, for the constant support and interest they showed in this work, as well as the GEAR community. Last but not least, I would like to thank my advisor, Francis Bonahon, for teaching me, with unlimited patience and enthusiasm, the powerful techniques of transverse structures for geodesic lami\-nations.

\end{document}